\documentclass[10pt]{amsart}
\usepackage{latexsym,amsmath,amssymb,graphicx,yhmath,cancel}
\usepackage{float}
\usepackage[bf, small]{caption}
\usepackage[all,web]{xy}
\usepackage[shortlabels]{enumitem}
\usepackage{color}
\usepackage{hyperref}

\usepackage{xcolor}

\def\xyellowspace{%
  \sbox0{\colorbox{yellow}{\strut\ }}
  \dimen0=\wd0\relax
  \hskip0pt\cleaders\box0\hskip\dimen0\hskip0pt}

\begingroup
\catcode`\ =\active%
\gdef\makeyellowspace{\let \xyellowspace\catcode`\ =\active}%
\endgroup

\def\?#1{\colorbox{yellow}{\strut#1}}

\def\urlfont{\DeclareFontFamily{OT1}{cmtt}{\hyphenchar\font='057}
              \normalfont\ttfamily \hyphenpenalty=10000}

\DeclareFontFamily{OT1}{rsfs10}{}
\DeclareFontShape{OT1}{rsfs10}{m}{n}{ <-> rsfs10 }{}
\DeclareMathAlphabet{\mathscript}{OT1}{rsfs10}{m}{n}

\DeclareMathOperator{\im}{Im}       
\DeclareMathOperator{\id}{id}       
\DeclareMathOperator{\Spec}{Spec}   
\DeclareMathOperator{\Hom}{Hom}     
\DeclareMathOperator{\Pic}{Pic}     
\DeclareMathOperator{\Cl}{Cl}       
\DeclareMathOperator{\Cox}{Cox}     
\DeclareMathOperator{\Aut}{Aut}     
\DeclareMathOperator{\rk}{rk}       
\DeclareMathOperator{\Sing}{Sing}   
\DeclareMathOperator{\Nef}{Nef}     
\DeclareMathOperator{\Relint}{Relint}  
\DeclareMathOperator{\Int}{Int}     
\DeclareMathOperator{\conv}{Conv}   

\makeatletter
\def\widebreve{\mathpalette\wide@breve}
\def\wide@breve#1#2{\sbox\z@{$#1#2$}%
     \mathop{\vbox{\m@th\ialign{##\crcr
\kern0.08em\brevefill#1{0.8\wd\z@}\crcr\noalign{\nointerlineskip}%
                    $\hss#1#2\hss$\crcr}}}\limits}
\def\brevefill#1#2{$\m@th\sbox\tw@{$#1($}%
  \hss\resizebox{#2}{\wd\tw@}{\rotatebox[origin=c]{90}{\upshape(}}\hss$}
\makeatletter

\title[$f$-mirror symmetry of toric complete intersections]{Framed duality and mirror symmetry for toric complete intersections}

\author[M. Rossi]{Michele Rossi}

\date{\today}

\address{Dipartimento di Matematica e Applicazioni, Universit\`a di Milano-Bicocca,
Edificio U5-Ratio, via Roberto Cozzi, 55, 20125 Milano} \email{michele.rossi@unimib.it}

\thanks{The author was partially supported by the I.N.D.A.M. as a member of the G.N.S.A.G.A.}

\def \a{\alpha }
\def \b{\beta }
\def \d{\delta }
\def \l{\lambda }
\def\ll{\boldsymbol{\lambda}}
\def\bg{\boldsymbol{\beta}}
\def \s{\sigma }
\def \D{\Delta }

\def \Si{\Sigma }

\def \vf{\varphi}

\def \vt{\vartheta}
\def \ét{\'{e}tale}
\def \aa{\mathbf{a}}
\def \bb{\mathbf{b}}
\def \cc{\mathbf{c}}
\def \ff{\mathbf{f}}
\def \e{\mathbf{e}}
\def \q{\mathbf{q}}

\def \uu{\mathbf{u}}
\def \v{\mathbf{v}}
\def \n{\mathbf{n}}
\def \m{\mathbf{m}}

\def \z{\mathbf{z}}
\def \x{\mathbf{x}}

\def \1{\mathbf{1}}
\def \0{\mathbf{0}}

\def\P{{\mathbb{P}}}

\def\p2{\mathbb{P}^2}
\def\p3{\mathbb{P}^3}
\def\p4{\mathbb{P}^4}

\def\cv#1{\wideparen{#1}}

\def\cO{\mathcal{O}}
\def\cY{\mathcal{Y}}

\def\Ksz{\mathcal{K}^\bullet}    

\def\rk{\operatorname{rk}}

\def\GL{\operatorname{GL}}

\def\Z{\mathbb{Z}}

\def\H{\mathbb{H}}
\def\C{\mathbb{C}}

\def\R{\mathbb{R}}
\def\M{\mathbf{M}}
\def\Q{\mathbb{Q}}
\def\N{\mathbb{N}}
\def\NN{\nabla}
\def\T{\mathbb{T}}

\def\L{\Lambda}

\def\XX{\mathbb{X}}

\def\SF{\mathcal{SF}}

\def\I{\mathcal{I}}

\def\Weil{\mathcal{W}_T}

\def\U1{\mathfrak{U}^{(1)}}

\def\Est{E_{\text{\rm st}}}
\def\hst{h_{\text{\rm st}}}
\def\rdY{\widehat{Y}^\vee}
\def\rO{\widehat{\Omega}}
\def\cY{\mathcal{Y}}

\theoremstyle{plain}
\newtheorem{theorem}{Theorem}[section]
\newtheorem{proposition}[theorem]{Proposition}

\newtheorem{thm-def}[theorem]{Theorem--Definition}
\newtheorem{corollary}[theorem]{Corollary}

\newtheorem{lemma}[theorem]{Lemma}

\newtheorem*{a-proposition}{Proposition}

\theoremstyle{remark}
\newtheorem{remark}[theorem]{Remark}

\newtheorem{example}[theorem]{Example}

\theoremstyle{definition}
\newtheorem{definition}[theorem]{Definition}

\newtheorem*{step I}{Step I}
\newtheorem*{step II}{Step II}
\newtheorem*{step III}{Step III}
\newtheorem*{step IV}{Step IV}

\newcommand{\oneline}{\vskip12pt}
\newcommand{\halfline}{\vskip6pt}
\newcommand{\cy}{Ca\-la\-bi-Yau }
\newcommand{\ka}{K\"{a}hler }

\begin{document}


\begin{abstract} This paper is devoted to systematically extend $f$-mirror symmetry between families of hypersurfaces in complete toric varieties, as introduced in \cite{R-fTV}, to families of complete intersections subvarieties. Namely, $f$-mirror symmetry is induced by framed duality of framed toric varieties extending Batyrev-Borisov polar duality between Fano toric varieties. Framed duality has been defined and essentially well described for families of hypersurfaces in toric varieties in the previous \cite{R-fTV}. Here it is developed for families of complete intersections, allowing us to strengthening some previous results on hypersurfaces. In particular, the class of projective complete intersections and their mirror partners are studied in detail. Moreover, a (generalized) Landau-Ginzburg/Complete-Intersection correspondence is discussed, extending to the complete intersection setup the LG/CY correspondence firstly studied Chiodo-Ruan and Krawitz.
\end{abstract}

\maketitle

\tableofcontents

\section*{Introduction}
In \cite{R-fTV} the concept of framed duality between (weakly) framed toric varieties has been introduced, as an extension of Batyrev duality between Fano toric varieties, thought of as framed by their anti-canonical divisor. This gave rise to an extension of the mirror symmetric phenomenon between anti-canonical sections of Fano toric varieties to sections of a framing divisor: we called this phenomenon $f$-mirror symmetry. Keeping in mind how Borisov and Batyrev-Borisov generalized Batyrev duality \cite{Borisov}, \cite{BB96}, the present paper is devoted to present a systematic extension of $f$-mirror symmetry to complete intersections subvarieties in complete toric varieties, already sketched in \S~6 of \cite{R-fTV}.

Roughly speaking, $f$-duality for complete intersections behaves as follows (see Algorithm~\ref{algoritmoDnef} for details). Assume $Y=\bigcap Y_k$ be a complete intersection subvariety of a complete toric variety $X$ and let $D_Y=\sum D_{Y_k}$ be a \emph{nef partitioned} (see Definition~\ref{def:nefp}) effective torus invariant divisor (the partitioned weak framing) of $X$ such that $Y_k\in|D_k|$, that is $Y_k\sim D_k$. A suitable combinatorial process, called \emph{partitioned $f$-process}, gives back a dual toric variety $\XX$ and a dual partitioned divisor $D^\vee=\sum D^\vee_k$ of $\XX$ giving an $f$-mirror model $Y^\vee=\bigcap Y^\vee_k\subset\XX$, where $Y_k^\vee\in|D^\vee_k|$. On the one hand, if $D_Y$ is a strictly effective divisor, that is a framing of $X$, then $(\XX, D^\vee)$ is still a framed toric variety and, under suitable conditions (called \emph{calibration}) this turns out to be an involutive process. On the other hand, if $D_Y$ is an effective non strictly effective divisor, that is, a weak framing of $X$, then $\XX$ is a no more complete toric variety and the mirror model $Y^\vee$ can be understood in terms of \emph{generalized} Landau-Ginzburg mirror model, turning out to be strictly related with LG models exhibited by Givental \cite{Givental-ICM} and Hori-Vafa \cite{Hori-Vafa}. Here, generalized means that the superpotential has to be understood in terms of a vectorial function $\T\to\C^l$, being $\T$ the acting torus on the ambient toric variety and $l$ the codimension of $Y\subset X$. This approach allows us to extend, to the complete intersection setup, the LG/Hypersurface correspondence studied in \cite[\S~3.6]{R-fTV} and extending the LG/\cy correspondence studied by Chiodo-Ruan \cite{Chiodo-Ruan} and Chiodo-Kalshnikov-Veniani \cite{CKV},  and the Krawitz duality \cite{Krawitz}.

  In particular, all these expects are studied in the case of complete intersections in $\P^n$ showing that, for non-negative Kodaira dimension, a calibrated $f$-process exists. More precisely,  proofs of several results announced in \cite{R-fTV} are here given, in the broader context of complete intersections, namely:
\begin{enumerate}
  \item Theorem~6.8 in \cite{R-fTV} is here restated by Theorem~\ref{thm:CI}, and asserts the existence of at least one framed mirror partner of every complete intersection in $\P^n$ admitting non-negative Kodaira dimension; this is the generalization of what has been already proved for hypersurfaces in \cite[Thm.~4.1, Cor.~4.3]{R-fTV};
  \item Theorem 5.13 in \cite{R-fTV} is here reformulated by means of the three statements given by Theorems~\ref{thm:B-mirror0}, \ref{thm:B-mirror1} and \ref{thm:B-mirror2}, showing that $f$-mirror symmetry produces $B$-side mirror symmetric models, so extending what proved for the $A$-side in \cite[Thm.~5.3, Prop.~5.9]{R-fTV}.
\end{enumerate}
Techniques used for (1) are the same employed in proving \cite[Thm.~4.1]{R-fTV}. For proving results in (2), we need to compute and compare some Hodge numbers of the two mirror partners. In computing Hodge numbers we used \emph{stringy Hodge numbers} introduced in \cite{Batyrev98} and methods for complete intersections given by Batyrev and Borisov in \cite{BB96}. For the latter, we used a different Koszul-type complex, as polytopes involved are no more connected by polar duality.

On the other hand, for negative Kodaira dimension projective complete intersections, mirror models and generalized LG mirror models are exhibited in \S~\ref{ssez:negKod} and Theorem~\ref{thm:negKod}.
\oneline
The paper is organized as follows.  In \S\ref{sez:dualita-ftv} $f$-duality for toric hypersurfaces and complete intersections is reviewed. Then, Theorem~\ref{thm:CI} is proved in \S\ref{ssez:Mirror-pCI}. \S\ref{sez:hdgnbrs} is devoted to computing the needed Hodge numbers of varieties giving the two sides of $f$-mirror symmetry. Finally, \S\ref{sez:l=1} and \S\ref{sez:l ge 2} are the application of results obtained in the previous sections, to hypersurfaces and complete intersections, respectively. In particular, results summarized in (2) are proved in \S\ref{sez:l=1}. While in \S\ref{sez:l ge 2} a particular example is described, as we are not able to produce a general recipe for checking mirror symmetry of complete intersections. Although computation here performed for the complete intersection $Y_{3,4}\subset\P^5$ can be obviously repeated for any projective complete intersection of sufficiently big degrees, singularities of $f$-mirror partners turns out to be too wild to understand the existence of a quasi-smooth resolution admitting the right number of \ka parameters: this is a concrete limit of checking mirror symmetry in this way, that is, via the so-called \emph{topological mirror test}, whose definition is reviewed in \S\ref{ssez:topMT}.

 \section{Framed mirror partners of toric complete intersections}\label{sez:dualita-ftv}

The present section is devoted to extending $f$-duality to families of complete intersection varieties in a fixed toric variety $X$, keeping in mind how Borisov and Batyrev-Borisov generalized Batyrev duality of hypersurfaces \cite{Borisov}, \cite{BB96}.

For preliminaries on toric varieties and general notation we completely refer the reader to \S1 and \S2 in \cite{R-fTV}.

\begin{definition}[Def.~2.1 and 7.1 in \cite{R-fTV}]\label{def:wftv} A \emph{weakly framed toric variety} (wftv) is a couple $(X,D_\aa)$ (also denoted $(X,\aa)$) where:
   \begin{itemize}
     \item $X=X(\Si)$ is a complete toric variety associated with the fan $\Si$ and
     $$\dim(X)=n\ ,\quad\rk(\Pic(X))=r$$
     \item $D_\aa=\sum_{\rho\in\Si(1)}a_\rho D_\rho=\sum_{i=1}^m a_i D_i\in\Weil(X)$, with $m:=|\Si(1)|=n+r$, is an effective torus invariant Weil divisor, that is $a_\rho\ge 0$, called a \emph{weak framing} of $X$, where $D_\rho$, resp. $D_i=D_{\rho_i}$, denotes the prime torus invariant divisor given by the closure of the torus orbit associated with the ray $\rho$, resp. $\rho_i$.
   \end{itemize}
   Furthermore, if $D_\aa$ is \emph{strictly effective}, that is $a_\rho>0$, then it is called a \emph{framing} of $X$ and the couple $(X,\aa)$ is called a \emph{framed toric variety} (ftv).
  \end{definition}

  \begin{definition}[Def~2.2 in \cite{R-fTV}]\label{def:f-polytope}
  Define the \emph{integer part} of a polytope $\D\subseteq M_\R$ as
 \begin{equation*}
   [\D]:=\conv(\{\m\in M\cap\D\})
 \end{equation*}
   The \emph{framing polytope} (\emph{$f$-polytope}) of a wftv $(X,\aa)$ is the lattice polytope $\D(X,\aa)$ in $M_\R$ so defined:
   \begin{equation}\label{k0}
     \D(X,\aa):=[k_0\D_\aa]\quad,\quad k_0:=
     \begin{cases}
       \begin{array}{cc}
          \min\{k\in\N\,|\,\0\in\Int[k\D_\aa]\} & \text{if exists} \\
          1 & \text{otherwise}
        \end{array}
\end{cases}
   \end{equation}
   being $\D_\aa$ the polytope associated with the divisor $D_\aa$, namely
   $$\D_\aa=\{\m\in M_\R\,|\,V^T\cdot\m \geq -\aa\}$$
 \end{definition}

 \begin{remark}\label{rem:(f)(wf)}
   By Definition~\ref{def:wftv}, one has the following two possibilities:
   \begin{itemize}
     \item[\textbf{(f)}] if $(X,\aa)$ is a ftv then $k_0$ always exists and can be $>1$,
     \item[\textbf{(wf)}] if $(X,\aa)$ is a wftv and not a ftv then $\0\in\partial\D_\aa$ and $k_0$ cannot exists, that is, $\D(X,\aa)=[\D_\aa]$\,.
   \end{itemize}
   In the following we will distinguish these two occurrences as (f) and (wf), namely
   \begin{eqnarray*}
     \text{(f)} &\Longleftrightarrow& \0\in\Int[k_0\D_\aa] \\
     \text{(wf)} &\Longleftrightarrow& \0\in\partial\D_\aa
   \end{eqnarray*}
 \end{remark}

\begin{definition}[Def.~6.1 in \cite{R-fTV}]\label{def:nefp}
  Let $(X,D_\aa=\sum_{j=1}^m a_jD_j)$ be a wftv and
  $$V=\left(
                                    \begin{array}{ccc}
                                      \v_1 & \cdots & \v_m \\
                                    \end{array}
                                  \right)
  $$
  be a fan matrix of $X$, where $m=n+r$. A \emph{partition} of the weak framing $D_\aa$ is the datum of a partition
  $$\exists\,l\in \N:\quad I_1\cup\cdots\cup I_l=\{1,\ldots,m\}\ ,\quad\forall\,i\ I_i\ne\emptyset\ ,\quad\forall\,i\neq j\quad I_i\cap I_j =\emptyset$$
   and divisors $D_{\aa_1},\ldots,D_{\aa_l}$ such that
  $$\forall\,k=1,\ldots,l\quad D_{\aa_k}:=\sum_{i\in I_k}a_iD_i\neq 0$$
  Notice that condition $\aa_k\neq\0$ is overabundant when $(X,\aa)$ is a ftv.
  Clearly
  $$D_\aa=\sum_{k=1}^l D_{\aa_k}\quad \text{i.e.}\quad \aa=\sum_{k=1}^l \aa_k$$
   The wftv $(X,D_\aa)$ with a framing partition $\aa=\sum_{k=1}^l\aa_k$ is called a \emph{partitioned wftv} and denoted by $(X,\aa=\sum_{k=1}^l\aa_k)$\,. If the framing partition is such that $D_{\aa_k}$ is a semi-ample divisor for every $k=1,\dots,l$, then $\aa=\sum_{k=1}^l\aa_k$ is called a \emph{nef} framing partition and $(X,\aa=\sum_{k=1}^l\aa_k)$ a \emph{nef} partitioned wftv.
\end{definition}

We recall, now, the $f$-process algorithm for complete intersections in toric varieties, as proposed in \cite[\S6.1]{R-fTV} for a partitioned ftv, and extend it to the case of a partitioned wftv.

\subsection{$f$-process for complete intersections}
Given a partitioned wftv
$$(X,D_\aa=\sum_{k=1}^l D_{\aa_k})$$
consider the following algorithm.
\subsubsection{The partitioned $f$-process algorithm}\label{algoritmoDnef}
\begin{enumerate}[(A)]
  \item Let $\D_\aa$ and $\D_{\aa_1},\ldots,\D_{\aa_l}$ be the polytopes associated with divisors $D_\aa$ and $D_{\aa_1},\ldots,D_{\aa_l}$, respectively, that is
  \begin{eqnarray}\label{part-polytopi}
    \D_\aa&=&\{\m\in M_\R\,|\,V^T\cdot\m \geq -\aa\}\\
    \nonumber
    \forall\,k=1,\ldots,l\quad\D_{\aa_k}&=&\{\m\in M_\R\,|\,V^T\cdot\m \geq -\aa_k\}
\end{eqnarray}
In particular, it turns out that
\begin{equation}\label{sommaintersezione}
  \bigcap_{k=1}^l\D_{\aa_k}=\{\0\}\quad \text{and}\quad \D_\aa=\sum_{k=1}^l\D_{\aa_k}
\end{equation}
where the sum denotes the Minkowski sum of polytopes.
  \item Define
\begin{equation*}
  \cv{\D}_\aa:=\conv(\D_{\aa_1},\ldots,\D_{\aa_l})\subset M_\R
\end{equation*}
Clearly $\cv{\D}_\aa\subseteq\D_\aa$. Then, recalling notation introduced in Remark~\ref{rem:(f)(wf)}, relations (\ref{sommaintersezione}) imply that one of the following occurs
\begin{eqnarray}\label{0int}
  \text{(f)} &\Longrightarrow&\0\in\Int(\cv{\D}_\aa) \\
  \nonumber
  \text{(wf)} &\Longrightarrow& \0\in\partial\cv{\D}_\aa
\end{eqnarray}
In fact, case (wf) is obvious, while case (f) is an application of the following

\begin{lemma}\label{lem:0int}
  Given polytopes $\D_1,\ldots,\D_l$ such that
  \begin{itemize}
    \item[(a)]\  $\forall k=1,\ldots,l\quad\dim\D_k>0$\,,
    \item[(b)]\  $\0\in\bigcap_{k=1}^l\D_k$\,,
    \item[(c)]\  $\0\in\Relint\left(\sum_k\D_k\right)$
  \end{itemize}
  then $\0\in\Relint\left(\conv(\D_1,\ldots,\D_l)\right)$
\end{lemma}
\begin{proof}
  Let $B$ be a sufficiently small, open ball, centered in $\0$, of dimension $n:=\dim\left(\sum_k\D_k\right)$. By (c), one can assume $B\subset\sum_k\D_k$.
  Consider $\x\in B$ with $\x\neq\0$. By hypothesis (a)
  \begin{equation*}
    \forall\,k\ \exists\,\x_k\in\D_k\setminus\{\0\}:\quad \x=\sum_k\x_k
  \end{equation*}
  Then hypothesis (b) ensures that $\l\x_k\in\D_k$, for any $\l\in[0,1]$, by the convexity of $\D_k$. Therefore, up to shrink $B$,
  \begin{equation*}
    \forall\,k\ \exists\,\l_k\in[0,1],\,\x^0_k\in\D_k:\quad\x=\sum_k\l_k\x^0_k\,,\ \sum_k\l_k=1
  \end{equation*}
  by setting $\x^0_k:=(1/\l_k)\x_k$, which can be assumed a point in $\D_k$ up to choose $\x$, then $\x_k$, small enough. Then $\x\in\conv(\D_1,\ldots,\D_l)$, that is, $$B\subset\conv(\D_1,\ldots,\D_l)$$
  \end{proof}
  To get the first implication in (\ref{0int}), consider $\D_k:=\D_{\aa_k}$ and notice that hypotheses (a), (b), (c) are satisfied by definitions given in (\ref{part-polytopi}).

\noindent Define
\begin{equation*}
  \cv{\D}(X,\aa):=[h_0\cv{\D}_\aa]
  \end{equation*}
  with
  \begin{itemize}
    \item[(wf)]\ $h_0:=1$\,,
    \item[(f)]\ $h_0:=\min\{h\in\N\,|\,\forall\,k\ \dim[h\D_{\aa_k}]>0\ \text{and}\ \0\in\Relint\left(\sum_k[h\D_{\aa_k}]\right)\}$\,.
  \end{itemize}
 In case (f), an application of Lemma~\ref{lem:0int} with $\D_k=[h_0\D_{\aa_k}]$ gives
 \begin{equation*}
   \0\in\Relint\left(\conv([h_0\D_{\aa_1}]\,,\ldots,[h_0\D_{\aa_l}])\right)\subseteq \Relint\left(\cv{\D}(X,\aa)\right)
 \end{equation*}
 where the inclusion follows by observing that
 \begin{equation*}
   \conv([h_0\D_{\aa_1}]\,,\ldots,[h_0\D_{\aa_l}])\subseteq[\conv(h_0\D_{\aa_1}\,,\ldots,h_0\D_{\aa_l})]=\cv{\D}(X,\aa)
 \end{equation*}
  \item Set
\begin{equation*}
  \cv{\XX}_\aa:=\XX_{\cv{\Si}_\aa}\quad\text{where}\quad\cv{\Si}_\aa:=\Si_{\cv{\D}(X,\aa)}
\end{equation*}
(notation as in \cite[\S~1.4]{R-fTV}, denoting a fan spanned by a polytope and the associated toric variety)
and let $\cv{\L}_\aa\in\mathbf{M}(n\times \cv{m};\Z)$ be a fan matrix of $\cv{\XX}_\aa$\,, where $\cv{m}=|\cv{\Si}(1)|$. Notice that:
\begin{itemize}
  \item \emph{the toric variety $\cv{\XX}_\aa$ is complete if and only if case \emph{(f)} occurs, that is $(X,\aa)$ is a ftv.}
\end{itemize}
.
  \item For every $k=1,\ldots,l$, set $m_k:=|I_k|$ and consider the matrix
\begin{equation*}
  \cv{M}_{\aa_k}:= \left(V_{I_k}\right)^T\cdot\cv{\L}_\aa\in\mathbf{M}(m_k\times \cv{m};\Z)
\end{equation*}
and let $\bb_k=(b_{jk})_{j=1}^{\cv{m}}$ be \emph{the minimum non-negative column vector} such that
\begin{equation*}
  \cv{M}_{\aa_k}^T+B_k\geq \0\quad\text{where}\quad B_k:=\underbrace{\left(\,\bb_k\ \cdots\ \bb_k\,\right)}_{m_k\ \text{times}}\,\in \mathbf{M}(\cv{m}\times m_k,\N)
\end{equation*}
Then, define $\cv{\bb}:=\sum_{k=1}^l\bb_k$\,. Calling $\cv{D}_1,\ldots,\cv{D}_{\cv{m}}$ the torus invariant ge\-ne\-ra\-tors of $\Weil(\cv{\XX}_{\aa})$, there is a unique induced partition
\begin{equation*}
  J_1\cup\cdots\cup J_l=\{1,\ldots,\cv{m}\}
\end{equation*}
such that $\cv{D}_{\bb_k}:=\sum_{j\in J_k}b_{jk}\cv{D}_j$. Notice that:
\begin{itemize}
  \item[(f)] in case (f),  $\left(\cv{\XX}_\aa,\cv{D}_{\cv{\bb}}=\sum_{k=1}^l\cv{D}_{\bb_k}\right)$ is a \emph{partitioned ftv} called the \emph{f-dual partitioned ftv} of the partitioned ftv $(X,\aa=\sum_k\aa_k)$;
  \item[(wf)] in case (wf), the couple $\left(\cv{\XX}_\aa,\cv{D}_{\cv{\bb}}\right)$ cannot be a wftv, as $\cv{\XX}_\aa$ is not complete, and will be called the \emph{f-dual model of the partitioned wftv $(X,\aa=\sum_k\aa_k)$}.
\end{itemize}
\begin{proof}
  For (f), setting
  \begin{equation*}
    \cv{\L}_\aa=\left(
                  \begin{array}{ccc}
                    \cv{\ll}_1 & \cdots & \cv{\ll}_{\cv{m}}\\
                  \end{array}
                \right)
  \end{equation*}
  define $J_k$ as follows
  \begin{equation*}
    \forall\,k=1,\ldots,l\quad J_k:=\{j\in\{1,\ldots,\cv{m}\}\,|\,\cv{\ll}_j\in\D_{\aa_k}\}
  \end{equation*}
  Then $J_1\cup\cdots\cup J_l=\{1,\ldots,\cv{m}\}$ is clearly a partition, as $\cv{\ll}_j\in\D_{\aa_i}\cap\D_{\aa_j}$, with $i\neq j$, implies $V^T\cdot\cv{\ll}_j\ge 0$, hence $\cv{\ll}_j=0$ being $X$ complete, contradicting (f).
  Notice that
  \begin{equation*}
    \forall\,k=1,\ldots,l\,,\ \forall\,j\in J_k\,,\ \forall\,i\not\in I_k\quad \cv{\ll}_j^T\cdot\v_i = \v_i^T\cdot\cv{\ll}_j\ge 0\ \Longrightarrow\ b_{jk}= 0
  \end{equation*}
  as $\cv{\ll}_j\in\D_{\aa_k}$ and recalling the definition of $\D_{\aa_k}$ given by relations (\ref{part-polytopi}) in step (A). This suffices to show that $\bb=\sum_k\bb_k$ is a partitioned weak framing. Moreover, by construction, for every $j\in J_k$ there exists a positive integer $h$ such that $h\cv{\ll}_j$ is a vertex of $\D_{\aa_k}$, meaning that there must exists $i\in I_k$ such that
  \begin{equation*}
    (h\cv{\ll}_j)^T\cdot\v_i= h \v_i^T\cdot\cv{\ll}_j=-a_j
  \end{equation*}
  Then $\v_i^T\cdot\cv{\ll}_j$ is a negative integer implying that $b_{jk}>0$, that is, $\bb$ is a framing.

  For (wf) there is no more anything to prove.
\end{proof}

\item Assume we are in case (f). Analogously to step (A), let ${\D}_{\cv{\bb}}$ and $\cv{\D}_{\bb_1},\ldots,\cv{\D}_{\bb_l}$ be the polytopes associated with divisors $\cv{D}_{\cv{\bb}}$ and $\cv{D}_{\bb_1},\ldots,\cv{D}_{\bb_l}$, respectively, that is
  \begin{eqnarray*}
    {\D}_{\cv{\bb}}&=&\{\n\in N_\R\,|\,\cv{\L}_\aa^T\cdot\n \geq -\cv{\bb}\}\\
    \forall\,k=1,\ldots,l\quad\cv{\D}_{\bb_k}&=&\{\n\in N_\R\,|\,\cv{\L}_\aa^T\cdot\n \geq -\bb_k\}
\end{eqnarray*}
Then
\begin{equation}\label{sommaintersezione-b}
  \bigcap_{k=1}^l\cv{\D}_{\bb_k}=\{\0\}\quad \text{and}\quad {\D}_{\cv{\bb}}=\sum_{k=1}^l\cv{\D}_{\bb_k}
\end{equation}
Notice that: in case (wf), $\D_{\cv{\bb}},\cv{\D}_{\bb_k}$ are polyhedra and not, in general, polytopes.
\item Assume we are in case (f). Analogously to step (B), define
\begin{equation*}
  \cv{\D}_{\cv{\bb}}:=\conv(\cv{\D}_{\bb_1},\ldots,\cv{\D}_{\bb_l})\subset N_\R
\end{equation*}
Clearly $\cv{\D}_{\cv{\bb}}\subseteq{\D}_{\cv{\bb}}$ and relations (\ref{sommaintersezione-b}) ensure that $\0\in\Int(\cv{\D}_{\cv{\bb}})$. Then, (\ref{sommaintersezione-b}) still holds for multiple polytopes $h_1{\D}_{\cv{\bb}}$ and $h_1\cv{\D}_{\bb_1},\ldots,h_1\cv{\D}_{\bb_l}$, so giving that
\begin{equation*}
  \bigcap_{k=1}^l[h_1\cv{\D}_{\bb_k}]=\{\0\}\quad \text{and}\quad \0\in\Int(\D(\cv{\XX}_{\aa},\cv{\bb}))
\end{equation*}
since $\D(\cv{\XX}_{\aa},\cv{\bb})=[\sum_{k=1}^lh_1\cv{\D}_{\bb_k}]$, when $h_1$ is defined as the minimum positive integer such that $\0\in\Int([h_1\D_{\cv{\bb}}])$.
Then $\0\in\Int(\cv{\D}(\cv{\XX}_{\aa},\cv{\bb}))$,
being $$\cv{\D}(\cv{\XX}_{\aa},\cv{\bb}):=[h_1\cv{\D}_{\cv{\bb}}]$$
    \item  Assume we are in case (f). Analogously to step (C), set
\begin{equation*}
  \cv{\XX}_{\cv{\bb}}:=\XX_{\cv{\Si}_{\cv{\bb}}}\quad\text{where}\quad\cv{\Si}_{\cv{\bb}}:=
  \Si_{\cv{\D}(\cv{\XX}_{\aa},\cv{\bb})}
\end{equation*}
and let $\cv{\L}_{\cv{\bb}}\in\mathbf{M}(n\times \widetilde{m};\Z)$ be a fan matrix of $\cv{\XX}_{\cv{\bb}}$\,, for some $\widetilde{m}\in\N$\,. As above, $\cv{\XX}_{\cv{\bb}}$ is a complete toric variety.
    \item Assume we are in case (f). Analogously to step (D), for every $k=1,\ldots,l$, set $\cv{m}_k:=|J_k|$ and consider the matrix
\begin{equation*}
  \cv{M}_{\aa_k,\cv{\bb}}:=\left((\cv{\L}_\aa)_{J_k}\right)^T\cdot\cv{\L}_{\cv{\bb}}\in\mathbf{M}(\cv{m}_k\times \widetilde{m};\Z)
\end{equation*}
and let $\cc_k=(c_{j,k})_{j=1}^{\widetilde{m}}$ be \emph{the minimum non-negative column vector} such that
\begin{equation*}
  \cv{M}_{\aa_k,\cv{\bb}}^T+C_k\geq \0\quad\text{where}\quad C_k:=\underbrace{\left(\,\cc_k\ \cdots\ \cc_k\,\right)}_{\cv{m}_k\ \text{times}}\,\in \mathbf{M}(\widetilde{m}\times \cv{m}_k,\N)
\end{equation*}
Then, $(\cv{\XX}_{\cv{\bb}}, \cv{\cc}:=\sum_{k=1}^l\cc_k)$ is a partitioned ftv, whose partitioned framing is given by $\widetilde{D}_{\cv{\cc}}= \sum_{j=1}^{\widetilde{m}}c_{jk}\widetilde{D}_j$, calling $\widetilde{D}_1,\ldots,\widetilde{D}_{\widetilde{m}}$ the torus invariant ge\-ne\-ra\-tors of $\Weil(\cv{\XX}_{\cv{\bb}})$.
\end{enumerate}

\begin{remark}\label{rem:cvLambdak}
  Recalling step (D) in the previous algorithm~\ref{algoritmoDnef}, define
  \begin{equation*}
    \cv{\L}_{\aa_k}:=\left(\cv{\L}_\aa\right)_{J_k}
  \end{equation*}
  Then
  \begin{equation*}
    \left(
       \begin{array}{c}
         \cv{M}_{\aa_1} \\
         \vdots \\
         \cv{M}_{\aa_l} \\
       \end{array}
     \right)=\left(
       \begin{array}{c}
         \left(V_{I_1}\right)^T\cdot\cv{\L}_\aa \\
         \vdots \\
         \left(V_{I_l}\right)^T\cdot\cv{\L}_\aa \\
       \end{array}
     \right)=\left(
               \begin{array}{ccc}
                 V^T\cdot\cv{\L}_{\aa_1}\ \vline & \cdots &\vline\ V^T\cdot\cv{\L}_{\aa_l} \\
               \end{array}
             \right)
  \end{equation*}
  and $(\bb_k)_{J_k}:=\left(
                        \begin{array}{c}
                          b_{j_1,k} \\
                          \vdots \\
                          b_{j_{\cv{m}_k},k} \\
                        \end{array}
                      \right)$, with $\cv{m}_k:=|J_k|$ and $J_k=\{j_1,\ldots,j_{\cv{m}_k}\}$, is the minimum strictly positive column vector such that
\begin{equation*}
  \cv{\L}_{\aa_k}^T\cdot V+\underbrace{\left(\,(\bb_k)_{J_k}\ \cdots\ (\bb_k)_{J_k}\,\right)}_{\cv{m}_k\ \text{times}}\ge \0
\end{equation*}
\end{remark}

\begin{definition}[partitioned $f$-process]
Following the previous algorithm~\ref{algoritmoDnef}, the couple $(\cv{\XX}_\aa,\cv{\bb}=\sum_{k=1}^l\bb_k)$, is called a \emph{partitioned $f$-dual model} of the wftv $(X,\aa=\sum_{k=1}^l\aa_k)$.

\noindent Assume we are in case (f): a double application of partitioned $f$-duality defines a \emph{partitioned $f$-process}
\begin{equation}\label{nefDprocess}
  \left(X,\aa=\sum_{k=1}^l\aa_k\right)\ {\rightsquigarrow}\ \left(\cv{\XX}_\aa,\cv{\bb}=\sum_{k=1}^l\bb_k\right)\ {\rightsquigarrow}\ \left(\cv{\XX}_{\cv{\bb}},\cv{\cc}=\sum_{k=1}^l\cc_k\right)
\end{equation}
 Such a process is called \emph{calibrated} if there exist $\Xi\in\SF(V)$ and $\Xi'\in\SF(\cv{\L}_{\cv{\bb}})$, refining $\Si$ and $\cv{\Si}_{\cv{\bb}}$, respectively\footnote{$\SF$ is defined as in \cite[1.3.1]{R-fTV} and denotes the set of all complete and simplicial fans whose 1-skeleton is given by all the rays generated by the columns of the indicated fan matrix.}, such that
  $$\left(\widehat{X},\vf^*D_\aa\right)\ {\cong}\ \left(\widehat{X}',(\vf')^*\widetilde{D}_{\cv{\cc}}\right)$$
  are isomorphic framed toric varieties, where
  $$\vf:\widehat{X}(\Xi)\longrightarrow X(\Si)\quad\text{and}\quad\vf':\widehat{X}'(\Xi')\longrightarrow \cv{\XX}_{\cv{\bb}}(\cv{\Si}_{\cv{\bb}})$$
  are the small resolutions associated with the choice of $\Xi$ and $\Xi'$, respectively.
\end{definition}

\begin{proposition}\label{prop:calibrato}
  In the above notation, up to identifying lattices $M$ (hence $N$) of $X$ and $\cv{\XX}_{\cv{\bb}}$, the partitioned $f$-process (\ref{nefDprocess}) is calibrated if and only if
  \begin{eqnarray}\label{nefDcalibrato}
\nonumber
  \cv{\L}_{\cv{\bb}} &=& V\quad\text{(up to a permutation of columns)} \\
  \forall\,k=1,\ldots,l\quad \cc_k&=& \aa_k
\end{eqnarray}
\end{proposition}
The proof is analogous to that proving \cite[Prop.~6.3]{R-fTV}.

Before speaking about $f$-mirror partners of complete intersections embedded in a partitioned wftv, consider the following important property of an $f$-dual partitioned framing.

\begin{proposition}\label{prop:nef}
  Let $(\cv{\XX}_\aa,\cv{\bb}=\sum_{k=1}^l\bb_k)$ be the $f$-dual partitioned ftv of a partitioned ftv $(X,\aa=\sum_{k=1}^l\aa_k)$. Then there always exists a small, $\Q$-factorial, projective, partial resolution $\vf:\check{\XX}(\Xi)\longrightarrow\cv{\XX}_\aa$, with $\Xi\in\SF\left(\cv{\L}_\aa\right)$, such that pull back divisors $\vf^*\left(\cv{D}_{\bb_k}\right)\in\Weil(\check{\XX})$ are semi-ample, for every $k=1,\ldots,l$\,. In particular, $\left(\check{\XX},\vf^*(\cv{D}_{\cv{\bb}})=\sum_{k=1}^l\vf^*(\cv{D}_{\bb_k})\right)$ is a nef partitioned ftv and a suitable positive multiple of $\vf^*(\cv{D}_{\cv{\bb}})$ is an ample divisor giving a projective embedding of $\check{\XX}$. Consequently, $\cv{D}_{\cv{\bb}}$ and $\cv{D}_{\bb_k}$, for every $k$, are all movable divisor of $\cv{\XX}_\aa$.
\end{proposition}

\begin{proof}
  The columns of the fan matrix $\cv{\L}_\aa$ of $\cv{\XX}_\aa$ give, by definition, the collection of all the normal inward vectors to the facets of the polytope
  $${\D}_{\cv{\bb}}=\{\n\in N_\R\,|\,\cv{\L}_\aa^T\cdot\n \geq -\cv{\bb}\}$$
  defined in step (E) of the algorithm \ref{algoritmoDnef}. Then the normal fan $\Si^\perp:=\Si^\perp_{\D_{\cv{\bb}}}$ is a refining fan of the fan $\cv{\Si}_\aa:=\Si_{\cv{\D}(X,\aa)}$ of $\cv{\XX}_\aa$, giving rise, by inclusion of fans, to a small partial resolution $\vf^\perp:\P_{\D_{\cv{\bb}}}\longrightarrow\cv{\XX}_\aa$\,. Since
  \begin{equation*}
    \D_{(\vf^\perp)^*(\cv{D}_{\bb_k})}=\D_{\cv{D}_{\bb_k}}=\cv{\D}_{\bb_k}
  \end{equation*}
  the Minkowski sum in (\ref{sommaintersezione-b}) shows that there exists a minimum integer $h>0$ such that
  \begin{equation*}
    h{\D}_{\cv{\bb}}=\sum_{k=1}^lh\cv{\D}_{\bb_k}
  \end{equation*}
  is a Minkowski sum of lattice politopes.
  Then $(\vf^\perp)^*(\cv{D}_{\bb_k})$ is a semi-ample divisor of $\P_{\D_{\cv{\bb}}}$, for every $k$, by \cite[Prop.~2.4]{BB96}.
  Notice that $\P_{\D_{\cv{\bb}}}$ is not necessarily $\Q$-factorial, but there certainly exists a fan $\Xi\in\SF(\cv{\L}_\aa)$ giving rise to a small, $\Q$-factorial, partial resolution $\vf:\check{\XX}(\Xi)\longrightarrow\cv{\XX}_\aa$, factorizing through $\vf^\perp$, that is,
  \begin{equation*}
    \xymatrix{\check{\XX}\ar[rd]^-{\vf'}\ar[rr]^-\vf&&\cv{\XX}_\aa\\
                &\P_{\D_{\cv{\bb}}}\ar[ur]^-{\vf^\perp}&}
  \end{equation*}
  Define $\SF(\cv{\L}_\aa,\cv{\bb}):=\{\Xi\in\SF(\cv{\L}_\aa)\,|\,\Xi\prec\Si^\perp\}$. Then, by \cite[Prop.~1.11]{Hu-Keel},
  \begin{equation}\label{Nef}
    ((\vf^\perp)^*)^{-1}\left(\Nef\left(\P_{\D_{\cv{\bb}}}\right)\right)\cong
    \bigcap_{\Xi\in\SF(\cv{\L}_\aa,\cv{\bb})}(\vf^*)^{-1}\Nef\left(\check{\XX}(\Xi)\right)
  \end{equation}
  Recall that $(\vf^\perp)^*$ and $\vf^*$ are isomorphisms, as induced by isomorphisms in codimension 1.
  Therefore, $\vf^*(\cv{D}_{\bb_k})=(\vf')^*\left((\vf^\perp)^*(\cv{D}_{\bb_k})\right)$ is a semi-ample divisor of $\check{\XX}$, for every $k$. Moreover, by construction, $h(\vf^\perp)^*(\cv{D}_{\cv{\bb}})$ is an ample divisor determining a projective embedding of $\P_{\D_{\cv{\bb}}}$. This means that
  \begin{equation*}
    (\vf')^*\left(h(\vf^\perp)^*(\cv{D}_{\cv{\bb}})\right)=h\vf^*\left(\cv{D}_{\cv{\bb}}\right)
  \end{equation*}
  is an ample divisor of $\check{\XX}$, giving rise to a projective embedding of $\check{\XX}$.

\noindent  Finally, the fact that $\cv{D}_{\cv{\bb}}$ and $\cv{D}_{\bb_k}$, for every $k$, are all movable divisors of $\cv{\XX}_\aa$ still follows from  \cite[Prop.~1.11]{Hu-Keel}.
\end{proof}

\begin{corollary}\label{cor:nef}
  Let $(X,\aa=\sum_{k=1}^l\aa_{k})$ be a partitioned ftv admitting a calibrated $f$-process. Then, there exists a small, $\Q$-factorial, projective, partial resolution
  $$\phi:\widehat{X}\longrightarrow X$$
   such that $(\widehat{X},\phi^*(D_\aa)=\sum_k\phi^*(D_{\aa_k}))$ is a nef partitioned ftv and $h\phi^*(D_\aa)$ is an ample divisor of $\widehat{X}$, for a suitable positive integer $h\in\N$. In particular, $D_\aa$ and $D_{\aa_k}$, for every $k$, are all movable divisor of $X$.
\end{corollary}

\begin{remark}
  The previous Proposition~\ref{prop:nef} and Corollary~\ref{cor:nef} apply, when $l=1$, to the $f$-dual partner of a ftv and to a ftv admitting a calibrated $f$-process.
\end{remark}

\begin{definition}[$f$-mirror of a complete intersection in a ftv]\label{def:nefpmirror}
  Given the partitioned ftv $(X,\aa=\sum_{k=1}^{l}\aa_k)$, assume that the associated partitioned $f$-process (\ref{nefDprocess}) is calibrated. Consider the complete intersection subvariety
  $$Y:=\bigcap_{k=1}^lY_k\subset X\quad\text{with}\quad Y_k\in|D_{\aa_k}|$$
  The generic complete intersection subvariety
  $$Y^\vee:=\bigcap_{k=1}^l Y^\vee_k\subset \cv{\XX}_\aa\quad\text{with}\quad Y_k\in|\cv{D}_{\bb_k}|$$
  is called a \emph{$f$-mirror partner of} $Y$.
\end{definition}

\begin{remark}
  If $l=1$, that is, the framing partition is trivial, the $f$-mirror duality defined by the previous Definition~\ref{def:nefpmirror} reduces to give the $f$-mirror duality between hypersurfaces in toric varieties.
\end{remark}

\begin{remark}\label{rem:CIequazioni}
  The defining polynomials of both $Y$ and $Y^\vee$ in the Cox rings of $X$ and $\cv{\XX}_\aa$, respectively, can be explicitly described as follows. Namely:
  \begin{itemize}
    \item[(a)] for every $k=1,\ldots,l$, the lattice polytope $[\D_{\aa_k}]$ is the Newton polytope of $Y_k\in|D_{\aa_k}|$; call $\overline{\L}_{\aa_k}$ a matrix whose columns are given by all the lattice points in $[\D_{\aa_k}]$: it is well defined up to a permutation of columns; setting $l_k:=|\D_{\aa_k}\cap M|$, then $\overline{\L}_{\aa_k}$ is a $n\times l_k$ integer matrix; define
        \begin{equation*}
          \overline{M}_{\aa_k}:= V^T\cdot \overline{\L}_{\aa_k}\quad\text{and}\quad \overline{A}_k:=\underbrace{\left(\,\aa_k\ \cdots\ \aa_k\,\right)}_{l_k\ \text{times}}\,\in \mathbf{M}(m\times l_k;\N)\,;
        \end{equation*}
        then the polynomial of $Y_k$ is given by
        \begin{equation*}
          f_k=\sum_{j=1}^{l_k} c_j\x^{\m_j} \in \Cox(X)\cong\C[x_1,\ldots,x_m]
        \end{equation*}
        where $\m_j=(m_{i,j})$ is the $j$-th column of $\overline{M}_{\aa_k}+\overline{A}_k$ and $\x^{\m_j}:=\prod_{i=1}^m x_i^{m_{i,j}}$;
    \item[(b)] recalling step (E) in the algorithm~\ref{algoritmoDnef}, the lattice polytope $[\cv{\D}_{\bb_k}]$ is the Newton polytope of $Y_k^\vee\in|\cv{D}_{\bb_k}|$; call $\overline{\L}_{\bb_k}$ a matrix whose columns are given by all the lattice points in $[\cv{\D}_{\bb_k}]$; setting $l'_k:=|\cv{\D}_{\bb_k}\cap N|$, then $\overline{\L}_{\bb_k}$ is a $n\times l'_k$ integer matrix; define
        \begin{equation*}
          \overline{M}_{\aa,\bb_k}:= \cv{\L}_\aa^T\cdot \overline{\L}_{\bb_k}\quad\text{and}\quad \overline{B}_k:=\underbrace{\left(\,\bb_k\ \cdots\ \bb_k\,\right)}_{l'_k\ \text{times}}\,\in \mathbf{M}(\cv{m}\times l'_k;\N)\,;
        \end{equation*}
        then the polynomial of $Y^\vee_k$ is given by
        \begin{equation}\label{fdual_k}
          f^\vee_k=\sum_{j=1}^{l'_k} c_{jk}\x^{\n_j} \in \Cox(\cv{\XX}_\aa)\cong\C[x_1,\ldots,x_{\cv{m}}]
        \end{equation}
        where $\n_j=(n_{i,j})$ is the $j$-th column of $\overline{M}_{\aa,\bb_k}+\overline{B}_k$ and $\x^{\n_j}:=\prod_{i=1}^{\cv{m}} x_i^{n_{i,j}}$.
  \end{itemize}
  Notice that, for every $k$, both $f_k$ and $f_k^\vee$ are homogeneous polynomials, with respect to degrees induced by class groups. In fact, columns of both $\overline{M}_{\aa_k}$ and $\overline{M}_{\aa,\bb_k}$ determine trivial divisors, up to linear equivalence. Then
  $$\deg(f_k)=[D_{\aa_k}]\in\Cl(X)\quad\text{and}\quad\deg(f_k^\vee)=[\cv{D}_{\bb_k}]\in\Cl(\cv{\XX}_\aa)$$
  Moreover, if $(X,\aa)$ is a ftv then Corollary~\ref{cor:nef} asserts that, up to pass to suitable small, $\Q$-factorial, projective, partial resolutions $\widehat{X}$ and $Y$ of $X$ and $\cv{\XX}_\aa$, respectively, polynomials $f_k$ and $f^\vee_k$ give well defined local equations of hypersurfaces in every (maximal) affine subset of $\widehat{X}$ and $Y$, respectively, so well defining a pair of \emph{$f$-mirror dual, complete intersections, subvarieties.}
\end{remark}

\begin{remark}\label{rem:CIequazioni(wf)}
Assume we are in case (wf). Then $\cv{\XX}_\aa$ is non-complete and, in general $\cv{\D}_{\bb_1},\ldots,\cv{\D}_{\bb_l}$ are polyhedra and not polytopes, meaning that $\overline{\L}_{\bb_k}$, as defined in item (b) of the previous Remark~\ref{rem:CIequazioni}, is no more  well defined.\\
To understand how proposing an $f$-mirror partner in the (wf) case, recall that Proposition~\ref{prop:calibrato} gives $\cv{\L}_\bb=V$, up to a permutation on columns, when $(X,\aa=\sum \aa_k)$ is a partitioned ftv admitting a calibrated partitioned $f$-process. Then, in the (wf) case we are going to restate the previous item (b) by setting
\begin{equation*}
  \NN_k:=\conv\left(V_{I_k}\,|\,\0_n\right)
\end{equation*}
and defining $\overline{\L}_{\bb_k}$ as the matrix whose columns are given by all the lattice points in $\NN_k$.
\end{remark}

\begin{definition}[$f$-mirror of a complete intersection in a wftv]\label{def:wf-mirror}
  Given the partitioned wftv $(X,\aa=\sum_{k=1}^{l}\aa_k)$, the generic complete intersection subvariety
  $$Y^\vee:=\bigcap_{k=1}^l Y^\vee_k\subset \cv{\XX}_\aa$$
  where $Y^\vee_k$ is the zero-locus of the polynomial $f^\vee_k$ defined in (\ref{fdual_k}), with $\overline{M}_{\aa,\bb_k}$ defined by following the prescription of the previous Remark~\ref{rem:CIequazioni(wf)},
  is called an \emph{$f$-mirror partner of} $Y$.
\end{definition}

\subsection{A $LG/CI$-correspondence}\label{ssez:LG/CI} The present section is meant to generalize from hypersurfaces to complete intersections what was observed in \cite[\S~3.6.1]{R-fTV}.

Given a partitioned wftv $(X,\aa=\sum_k\aa_k)$ and a complete intersection $Y\subset X$, defined as in Definition~\ref{def:nefpmirror}, let $\T\cong(\C^*)^n$ be the acting torus on $X$. Consider the torus complete intersection $Z:=\T\cap Y$. Recalling part (a) of Remark~\ref{rem:CIequazioni}, $Y$ is the common zero locus of  polynomials $f_k$, with $k=1,\ldots,l$, generated by the columns of the matrix $\overline{M}_{\aa_k}+\overline{A}_k$. By setting
\begin{equation*}
  \x^{\aa_k}:=\prod_{i\in I_k}x_i^{a_{ik}}
\end{equation*}
consider the Laurent polynomial
\begin{equation*}
  f_{\aa_k}:={f_k\over \x^{\aa_k}}\in\C[\x,\x^{-1}]
\end{equation*}
generated by the columns of the matrix $\overline{M}_{\aa_k}$. Notice that, in $\T$ both $f_k$ and $f_{\aa_k}$ admit the same zero-locus, that is,
\begin{equation*}
  Z=\T\cap \ff^{-1}(\0)=\T\cap \ff_\aa^{-1}(\0)
\end{equation*}
where $\ff:=(f_1,\ldots,f_l):X\longrightarrow\C^l$ and $\ff_\aa:=(f_{\aa_1},\ldots,f_{\aa_l}):\T\longrightarrow\C^l$\,.

In the following, $(\T,\ff_\aa)$ will be called a \emph{generalized LG model} admitting a \emph{Laurent superpotential}.

On the other hand, let $\cv{\T}_\aa$ be the acting torus on $\cv{\XX}_\aa$ and $Z^\vee:=\cv{\T}_\aa\cap Y^\vee$, where $Y^\vee\subset\cv{\XX}_\aa$ is the complete intersection defined either in part (b) of Remark~\ref{rem:CIequazioni} or in Definition~\ref{def:wf-mirror}, depending on the occurrence of either case (f) or case (wf) in (\ref{0int}),  respectively. By setting
\begin{equation*}
  \x^{\bb_k}:=\prod_{j\in J_k}x_j^{b_{jk}}
\end{equation*}
consider the Laurent polynomial
\begin{equation*}
  f^\vee_{\bb_k}:={f^\vee_k\over \x^{\bb_k}}\in\C[\x,\x^{-1}]
\end{equation*}
In particular, $\ff^\vee_\bb:=(f^\vee_{\bb_1},\ldots,f^\vee_{\bb_l}):\cv{\T}_\aa\longrightarrow\C^l$ defines an \emph{$f$-mirror} generalized LG model $(\cv{\T}_\aa,\ff^\vee_\bb)$, with a Laurent superpotential, of the generalized LG model $(\T,\ff_\aa)$.

\begin{proposition}\label{prop:LGmirror}
  The generalized LG model $(\cv{\T}_\aa,\ff^\vee_\bb)$ admits a re-parameterization giving rise to a LG mirror model of Givental type \cite{Givental-ICM}
  of the complete intersection $Y\subset X$\,.
\end{proposition}

\begin{proof}
  By (\ref{fdual_k}), recall that $f^\vee_{\bb_k}=\sum_{j=1}^{l'_k} c_j\x^{\m_j}$ is a Laurent polynomial where $\m_j$ is the $j$-th column of the matrix $\overline{M}_{\aa,\bb_k}=\cv{\L}_\aa^T\cdot \overline{\L}_{\bb_k}$, with $\overline{\L}_{\bb_k}$ defined either in part (b) of Remark~\ref{rem:CIequazioni} or in Remark~\ref{rem:CIequazioni(wf)}. Then define
  \begin{eqnarray*}
    u_{jk} &:=& c_{jk}\x^{\m_j}\quad\text{with}\ j\in\{1,\ldots,l'_k\} \\
     q_k &:=& \prod_{j=1}^{l'_k}u_{jk}  \\
     F &:=& \sum_{k=1}^l\sum_{j=1}^{l'_k} u_{jk} \\
     \q &:=& (q_1,\ldots,q_l)
  \end{eqnarray*}
  Setting $N:=\sum_k l'_k$, then $F$ defines a function $F:\C^N\longrightarrow\C$. In particular, the Givental type LG model determined by the generalized LG model $(\cv{\T}_\aa,\ff^\vee_\bb)$ turns out to be given by the restriction of the superpotential $F$ to the counter-image of the torus $\T_l=(\C^*)^l$ by the projection $\pi:\uu\mapsto\q$\,, namely $(\pi^{-1}(\T_l),F)$\,.
\end{proof}

Examples of the LG model construction proposed by the previous Proposition~\ref{prop:LGmirror} are given in Example~\ref{ex:Y23} and Example~\ref{ex:Y13}, in the (wf) case, and at the end of \S~\ref{ssez:Y34}, in the (f) case.

\subsection{Mirroring projective complete intersections}\label{ssez:Mirror-pCI} The present section is devoted to give, on the one hand, a proof of the announced Theorem~6.8 in \cite{R-fTV} and, on the other hand, extending Theorem~7.4 in \cite{R-fTV} to the case of projective complete intersection of negative Kodaira dimension.

\subsubsection{Projective complete intersection of non-negative Kodaira dimension} In \cite{R-fTV} we announced the following result.

\begin{theorem}\label{thm:CI}
  Let $Y_d=\bigcap_kY_{d_k}\subseteq\P^n$ be a complete intersection of $l$ generic projective hypersurface, of degree $(d_1,\ldots,d_l)$ such that $\sum_kd_k\geq n+1$. Then there always exists a partitioned framing ${\aa}=\sum_k\aa_k$ of $\P^n$ such that $Y_{d_k}\sim D_{\aa_k}$, for every $k$, and the associated partitioned $f$-process is calibrated.
\end{theorem}

\begin{proof} To prove the statement, a particular partitioned framing, whose as\-so\-cia\-ted partitioned $f$-process is calibrated, will be produced. At this purpose, throughout the present proof, introduce the following notation
\begin{eqnarray}\label{a,ak}
\nonumber
  \forall\,k=1,\ldots,l\quad \d_k&:=&d_k - m_k+1\geq 1\quad\text{with $m_k=|I_k|$}\\ \aa_k^T&:=&\underbrace{(0,\ldots,0,\overbrace{1,\ldots,1,\d_k}^{I_k},0,\ldots,0)}_{n+1}\\
\nonumber
  |\aa_k|&:=&\sum_{i=1}^{n+1}a_{ki}=m_k-1+\d_k=d_k\\
\nonumber
  |\aa|&:=&\sum_{i=1}^{n+1}a_i=\sum_{k=1}^l|\aa_k|=\sum_{k=1}^l d_k
\end{eqnarray}
Clearly, $(\P^n,\aa=\sum_{k=1}^l\aa_k)$ is a partitioned ftv and the hypersurface $Y_{d_k}$ is linearly equivalent to $D_{\aa_k}$, which is a very ample divisor. Let us now exhibit the $f$-dual partitioned ftv $(\cv{\XX}_\aa,\cv{\bb}=\sum_{k=1}^l\bb_k)$,  under notation introduced in items (C) and (D) of algorithm~\ref{algoritmoDnef}. First of all, consider the fan matrix of $\P^n$ given by
\begin{equation}\label{V}
  V=\left(
      \begin{array}{ccc}
        I_n & | &-\mathbf{1} \\
      \end{array}
    \right)=\left(
              \begin{array}{cccc}
                \e_1 & \cdots & \e_n & -\mathbf{1}\\
              \end{array}
            \right)
    \in\mathbf{M}(n,n+1;\Z)
\end{equation}
Then, by the last part of \cite[Prop.~1.8]{R-fTV},
\begin{eqnarray}\label{DeltaaConv}
   \nonumber
    \D_\aa&=&\conv\left(\left\{-((V^{\{i\}})^T)^{-1}\cdot\aa^{\{i\}}\,|\,i=1,\ldots,n+1\right\}
    \right)\\
      &=&\conv\left(
                          \begin{array}{c}
                            |\aa|- a_1 \\
                            -a_2 \\
                            -a_3\\
                            \vdots \\
                            -a_n \\
                          \end{array}
                                        \begin{array}{c}
                                          -a_1 \\
                                           |\aa|- a_2\\
                                          -a_3\\
                                          \vdots \\
                                          -a_n \\
                                        \end{array}
                                      \cdots
                                        \begin{array}{c}
                                          -a_1 \\
                                          -a_2\\
                                          \vdots \\
                                          -a_{n-1} \\
                                          |\aa|- a_n\\
                                        \end{array}
                                        \begin{array}{c}
                                          -a_1 \\
                                          -a_2\\
                                          \vdots \\
                                          -a_{n-1} \\
                                          -a_n\\
                                        \end{array}
                                      \right)
  \end{eqnarray}
  Replacing $\aa$ by $\aa_k$, one has
  \begin{equation*}
    \D_{\aa_k}=\conv\left(
                          \begin{array}{c}
                            d_k- a_{k,1} \\
                            -a_{k,2} \\
                            -a_{k,3}\\
                            \vdots \\
                            -a_{k,n} \\
                          \end{array}
                                        \begin{array}{c}
                                          -a_{k,1} \\
                                           d_k- a_{k,2}\\
                                          -a_{k,3}\\
                                          \vdots \\
                                          -a_{k,n} \\
                                        \end{array}
                                      \cdots
                                        \begin{array}{c}
                                          -a_{k,1} \\
                                          -a_{k,2}\\
                                          \vdots \\
                                          -a_{k,n-1} \\
                                          d_k- a_{k,n}\\
                                        \end{array}
                                        \begin{array}{c}
                                          -a_{k,1} \\
                                          -a_{k,2}\\
                                          \vdots \\
                                          -a_{k,n-1} \\
                                          -a_{k,n}\\
                                        \end{array}
                                      \right)
  \end{equation*}
 and $\cv{\D}_\aa=\conv\left(\D_{\aa_1},\ldots,\D_{\aa_l}\right)$, by item (B) of algorithm~\ref{algoritmoDnef}.
 Define
 \begin{equation*}
   \forall\,k=1,\ldots,l,\ \forall\,j=1,\ldots,n+1\quad \vt_{kj}:=\left\{\begin{array}{cc}
                                                                           1 & \text{if}\ \forall\,i\neq j\ a_{ki}=0   \\
                                                                           \gcd(a_{ki}\,|\,i\neq j) & \text{otherwise}
                                                                         \end{array}\right.
\end{equation*}
Then, item (C) in \ref{algoritmoDnef} and Remark~\ref{rem:cvLambdak} give
$$\cv{\L}_\aa=\left(
                                                    \begin{array}{ccc}
                                                      \cv{\L}_{\aa_1} & \cdots & \cv{\L}_{\aa_l} \\
                                                    \end{array}
                                                  \right)\in\M(n\times \cv{m};\Z)
$$
where $\cv{\L}_{\aa_k}$ is the matrix obtained by deleting null columns from the following
\begin{equation*}
  \L_{\aa_k}=\left(
                          \begin{array}{c}
                            (d_k- a_{k,1})/\vt_{k,1} \\
                            -a_{k,2}/\vt_{k,1} \\
                            -a_{k,3}/\vt_{k,1}\\
                            \vdots \\
                            -a_{k,n}/\vt_{k,1} \\
                          \end{array}\begin{array}{c}
                                          -a_{k,1}/\vt_{k,2} \\
                                           (d_k- a_{k,2})/\vt_{k,2}\\
                                          -a_{k,3}/\vt_{k,2}\\
                                          \vdots \\
                                          -a_{k,n}/\vt_{k,2} \\
                                        \end{array}\quad \cdots\quad
                                        \begin{array}{c}
                                          -a_{k,1}/\vt_{k,n} \\
                                          -a_{k,2}/\vt_{k,n}\\
                                          \vdots \\
                                          -a_{k,n-1}/\vt_{k,n} \\
                                          (d_k- a_{k,n})/\vt_{k,n}\\
                                        \end{array}
                                        \begin{array}{c}
                                          -a_{k,1}/\vt_{k,n+1} \\
                                          -a_{k,2}/\vt_{k,n+1}\\
                                          \vdots \\
                                          -a_{k,n-1}/\vt_{k,n+1} \\
                                          -a_{k,n}/\vt_{k,n+1}\\
                                        \end{array}
                                      \right)
\end{equation*}
Notice that a null column appears in $\L_{\aa_k}$ if and only if $m_k=1$: in this case $\L_{\aa_k}$ admits one and only one null column. Therefore
\begin{equation*}
  \cv{m}=l(n+1)-\l\quad\text{with}\quad 0\leq\l\leq l\leq n
\end{equation*}
$\l$ being the number of $k$'s such that $m_k=1$.\\
Moreover, notice that $\vt_{kj}= 1$ except for $m_k=1,2$. In these cases,
 \begin{itemize}
   \item $m_k=1$\,:\quad  $\vt_{kj}=\left\{\begin{array}{cc}
                                       \d_k=d_k & \quad\text{if}\ j\neq \overline{j}:=1+\sum_{i=1}^{k-1}m_i \\
                                       1 & \text{if}\ j = \overline{j}\qquad \qquad\qquad\,
                                     \end{array}\right.$
   \item $m_k=2$\,:\quad  $\vt_{kj}=\left\{\begin{array}{cc}
                                       \d_k & \quad\text{if}\ j= \overline{j} \\
                                       1 & \quad\text{if}\ j\neq \overline{j}
                                     \end{array}\right.$
 \end{itemize}
 Consequently,
\begin{equation*}
  \forall\,k=1,\ldots,l\quad m_k\geq 3\ \Longrightarrow\ \left\{\begin{array}{c}
                                                                  \D_{\aa_k}=\conv(\L_{\aa_k})\qquad\hskip3.1truecm\qquad \\
                                                                  \cv{\L}_\aa=\left(
                                                    \begin{array}{ccc}
                                                      \L_{\aa_1} & \cdots & \L_{\aa_l} \\
                                                    \end{array}
                                                  \right)\in\M(n\times (nl+l);\Z)
                                                                \end{array}\right.
\end{equation*}
otherwise
\begin{equation*}
  \forall\,k :\ m_k=1\qquad \cv{\L}_{\aa_k}=\L_{\aa_k}^{\{\overline{j}\}}
\end{equation*}
\begin{equation*}
  \forall\,k=1,\ldots,l\quad m_k\geq 2\ \Longrightarrow\ \cv{\L}_\aa=\left(
                                                    \begin{array}{ccc}
                                                      \L_{\aa_1} & \cdots & \L_{\aa_l} \\
                                                    \end{array}
                                                  \right)\in\M(n\times (nl+l);\Z)
\end{equation*}
Then, passing to item (D) in \ref{algoritmoDnef}, $\cv{M}_{\aa_k}^T=\cv{\L}_\aa^T\cdot V_{I_k}\in\mathbf{M}(\cv{m}\times m_k ;\Z)$, meaning that
\begin{equation}\label{bkh_nulle}
  \bb_k=(b_{hk})_{h=1}^{\cv{m}}\quad \text{with}\quad \forall\,h\not\in J_k\quad b_{hk}=0
\end{equation}
and $J_k$ assigned by the prescription $(\cv{\L}_\aa)_{J_k}=\cv{\L}_{\aa_k}$. Moreover, by Remark~\ref{rem:cvLambdak}, entries $b_{hk}$ for $h\in J_k$ can be obtained by computing $\cv{\L}_{\aa_k}^T\cdot V_{I_k}$. Set $\cv{m}_k:=|J_k|$ and, for any $k$ such that $m_k=|I_k|\geq 2$, let $I'_k\subset J_k$ be the image of $I_k\subset \{1,\dots,n+1\}$ under the translation
\begin{eqnarray*}
   &\xymatrix{\{1,\ldots,n+1\}\ar[r]&J_k}&  \\
   &\xymatrix{\quad\hskip1.4truecm\quad i\quad\ar@{|->}[r]& i+\overline{h} }&\ ,\quad \overline{h}:=\sum_{i=1}^{k-1} \cv{m}_i
\end{eqnarray*}
Then
\begin{equation}\label{bkh_1}
  \forall\,k:\quad m_k=1\quad \cv{\L}_{\aa_k}^T\cdot V_{I_k}=\left(
                                                               \begin{array}{c}
                                                                 -1 \\
                                                                 \vdots \\
                                                                 -1 \\
                                                               \end{array}
                                                             \right)\in\M(n\times 1;\Z)\ \Longrightarrow\ \forall\,h\in J_k\quad b_{hk}=1
\end{equation}
\begin{eqnarray}\label{bkh_2}
  \forall\,k:\quad m_k=2&& \cv{\L}_{\aa_k}^T\cdot V_{I_k}=\L_{\aa_k}^T\cdot V_{I_k}=\left(
                                                               \begin{array}{cc}
                                                                 -1 & -\d_k \\
                                                                 \vdots & \vdots \\
                                                                 -1 & -\d_k \\
                                                                 \hline
                                                                 d_k-1 & -1 \\
                                                                 -1 & d_k-1 \\
                                                                 \hline
                                                                  -1 & -\d_k \\
                                                                 \vdots & \vdots \\
                                                                 -1 & -\d_k \\
                                                               \end{array}
                                                             \right)\left.
                                                                      \begin{array}{c}
                                                                         \\
                                                                         \\
                                                                      \end{array}
                                                                    \right\} I'_k\\
                    \nonumber
                    &\Longrightarrow& b_{hk}=\left\{\begin{array}{cc}
                                                       1 & \text{for}\ h\in I'_k \\
                                                       \d_k & \qquad\text{for}\ h\in J_k\setminus I'_k
                                                     \end{array}
                    \right.
\end{eqnarray}
\begin{eqnarray}\label{bkh_3}
\nonumber
  \forall\,k:\quad m_k\geq 3&& \cv{\L}_{\aa_k}^T\cdot V_{I_k}=\L_{\aa_k}^T\cdot V_{I_k}=\left(
                                                               \begin{array}{cccc}
                                                                 -1 &\cdots&-1& -\d_k \\
                                                                 \vdots &\cdots&& \vdots \\
                                                                 -1 &\cdots&-1& -\d_k \\
                                                                 \hline
                                                                 d_k-1 &\cdots& -1&-\d_k \\
                                                                 \vdots&\ddots&\vdots&\vdots\\
                                                                 -1 &\cdots& d_k-1& -\d_k \\
                                                                 \hline
                                                                 -1 &\cdots & -1& d_k-\d_k\\
                                                                 \hline
                                                                  -1 &\cdots&-1& -\d_k \\
                                                                 \vdots &\cdots&& \vdots \\
                                                                 -1 &\cdots&-1& -\d_k \\
                                                               \end{array}
                                                             \right)\left.
                                                                      \begin{array}{c}
                                                                         \\
                                                                         \\
                                                                         \\
                                                                         \\
                                                                      \end{array}
                                                                    \right\} I'_k\\
                    &\Longrightarrow& b_{hk}=\left\{\begin{array}{cc}
                                                       1 & \quad\,\text{for}\ h=\max(I'_k) \\
                                                       \d_k & \text{for}\ h\in J_k\setminus \{\max(I'_k)\}\quad\
                                                     \end{array}
                    \right.
\end{eqnarray}
Calling $\cv{\bb}=\sum_{k=1}^l\bb_k$ and $\cv{\XX}_\aa$ the toric variety defined in item (C) of \ref{algoritmoDnef}, the pair $(\cv{\XX}_\aa,\cv{\bb})$ turns out to be a ftv.

We are now going to considering the partitioned $f$-process associated with the partitioned ftv $(\P^n,\aa=\sum_k\aa_k)$. At this purpose, recall definitions of $\D_{\cv{\bb}}$, $\cv{\D}_{\bb_k}$ and $\cv{\D}_{\cv{\bb}}$\,, given in items (E) and (F) of \ref{algoritmoDnef}. On the other hand, recall the definition of
\begin{equation*}
  \D_\bb:=\{\n\in N_\R\,|\,\L_\aa^T\cdot\n \geq -\bb\}
\end{equation*}
where $(\XX_\aa,\bb)$ is the $f$-dual ftv of $(\P^n,\aa)$.

\begin{lemma}\label{lem:inclusione}
  In the above notation, $\cv{\D}_{\cv{\bb}}\subseteq\D_\bb$\,.
\end{lemma}

Let us postpone the proof of this Lemma and go on with the proof of Theorem~\ref{thm:CI}. Then one has
\begin{equation}\label{int-inclusione}
  [\cv{\D}_{\cv{\bb}}]\subseteq[\D_\bb]
\end{equation}
Let $\overline{\aa}$ be an increasing reordering of $\aa$. That is, assuming $\d_1\leq\d_2\leq\cdots\leq\d_l$,
 \begin{equation*}
  \overline{\aa}=(\underbrace{1\,,\,\dots\,,\,1}_{n+1-l\ \text{times}},\, \d_1\,,\, \ldots\, ,\, \d_l)
\end{equation*}
Since $2\leq l\leq n$, the framing $\overline{\aa}$ satisfies condition (38) in \cite[Thm.~4.1]{R-fTV} so that
\begin{equation*}
  \xymatrix{\left(\P^n,\aa\right)\cong\left(\P^n,\overline{\aa}\right)\ar@{<~>}[rr]^-{f-\text{duality}}
  &&\left(\XX_\aa,\bb\right)}
\end{equation*}
is a calibrated $f$-process, with
\begin{equation}\label{b}
  \bb=(\d_l/\vt_1,\ldots,\d_l/\vt_n,\d_{l-1})\ ,\quad \text{being}\quad \vt_i:=\gcd\{a_j\,|\,j\neq i\}
\end{equation}
Therefore
\begin{equation}\label{int-uguaglianza}
  [\D_\bb]=\NN:=\conv(V)
\end{equation}
Moreover, $\NN$ is clearly contained in $\cv{\D}_{\cv{\bb}}=\conv(\cv{\D}_{\bb_1},\ldots,\cv{\D}_{\bb_l})$, by relations (\ref{bkh_nulle}), (\ref{bkh_1}), (\ref{bkh_2}) and (\ref{bkh_3}) defining framings $\bb_k$, for $k=1,\ldots,l$. This fact, jointly with previous relations (\ref{int-inclusione}) and (\ref{int-uguaglianza}), gives that
\begin{equation}\label{parteintera}
  [\D_\bb]=\NN \subseteq [\cv{\D}_{\cv{\bb}}]\subseteq[\D_\bb]\ \Longrightarrow\ [\D_\bb]=\NN =\left[\cv{\D}_{\cv{\bb}}\right]
\end{equation}
giving the former of conditions (\ref{nefDcalibrato}) in Proposition~\ref{prop:calibrato}. The second condition in (\ref{nefDcalibrato}) follows by item (H) in \ref{algoritmoDnef}, as
\begin{equation*}
  \cv{M}^T_{\aa_k,\cv{\bb}}=\cv{\L}_{\cv{\bb}}\cdot(\cv{\L}_\aa)_{J_k}=\left\{\begin{array}{cc}
                                                                                V^T\cdot\L_{\aa_k}^{\{\overline{j}\}} & \text{if $m_k=1$}  \\
                                                                                V^T\cdot\L_{\aa_k} & \text{otherwise}
                                                                              \end{array}\right.
                                                                              \ \Longrightarrow\ \cc_k=\aa_k
\end{equation*}
Then, Proposition~\ref{prop:calibrato} definitively proves the statement.
\end{proof}

\begin{proof}[Proof of Lemma~\ref{lem:inclusione}] Consider $\n\in\cv{\D}_{\cv{\bb}}=\conv(\cv{\D}_{\bb_1},\ldots,\cv{\D}_{\bb_l})$. Then
\begin{equation*}
  \n=\sum_{k=1}^l \mu_k\n_k\quad\text{with}\quad \forall\,k\ \mu_k\ge 0\,,\ \sum_k\mu_k=1\,,\ \n_k\in\cv{\D}_{\bb_k}
\end{equation*}
In particular, condition $\n_k\in\cv{\D}_{\bb_k}$ implies that
\begin{equation}\label{cond-k}
  \cv{\L}_\aa^T\cdot \n_k\ge -\bb_k
\end{equation}
First of all, consider the case $m_k\geq 3$, for every $k=1,\ldots, l$. Then
\begin{eqnarray}\label{Lambda}
\nonumber
  \cv{\L}_\aa &=& \left(\begin{array}{ccc}
                                                      \L_{\aa_1} & \cdots & \L_{\aa_l} \\
                                                    \end{array}
                                                  \right)\in\M(n\times (nl+l);\Z) \\
  \D_\aa &=& \sum_{k=1}^l\D_{\aa_k}\ \Longrightarrow\ \L_\aa\ =\ \sum_{k=1}^l\L_{\aa_k}
\end{eqnarray}
where the right equality in the previous implication is realized up to a suitable reordering of columns in matrices $\L_{\aa_k}$.
Then, inequality (\ref{cond-k}) gives
\begin{equation}\label{Lambdak}
  \forall\,k\quad\L_{\aa_k}^T\cdot\n_k\ge -(\bb_k)_{J_k}\ ,\quad\forall\,h\neq k\quad\L_{\aa_h}^T\cdot\n_k\ge \0
\end{equation}
Therefore
\begin{equation*}
  \L_{\aa}^T\cdot\n=\left(\sum_{h=1}^l\L_{\aa_h}^T\right)\cdot\left(\sum_{k=1}^l \mu_k\n_k\right)= \sum_{k=1}^l \mu_k\left(\sum_{h=1}^l\L_{\aa_h}^T\right)\cdot\n_k\ge -\sum_{k=1}^l \mu_k(\bb_k)_{J_k}\ge -\bb
\end{equation*}
where the last inequality is obtained by recalling convention $\d_k\leq \d_{k+1}$ and the writing of $\bb$ given in (\ref{b}).  Then $\n\in\D_\bb$ and $\cv{\D}_{\cv{\bb}}\subseteq\D_\bb$.

Assume now that $m_k\geq 2$, for every $k=1,\ldots, l$. The first equality in (\ref{Lambda}) still holds, meaning that also (\ref{Lambdak}) still holds. But implication in (\ref{Lambda}) no longer applies. Namely, if $m_1=2$ then (\ref{Lambdak}) gives
\begin{equation}\label{b1inequality}
  \L_{\aa_1}^T\cdot\n_1\ge -(\bb_1)_{J_1}\stackrel{(\ref{bkh_2})}{=}-\left(
                    \begin{array}{c}
                      1 \\
                      1 \\
                      \d_1 \\
                      \vdots \\
                      \d_1
                    \end{array}
                  \right)\ ,\quad\forall\,h\neq 1\quad\L_{\aa_h}^T\cdot\n_1\ge \0
\end{equation}
Denote by $\ll_j$ and $\ll_{kj}$ the $j$-th columns of $\L_\aa$ and $\L_{\aa_k}$, respectively. Then, Minkowski sum in (\ref{Lambda}), up to a suitable reordering of columns of matrices $\L_{\aa_k}$, gives
\begin{equation*}
  \vt_j\ll_j^T\cdot\n_1=\sum_{k=1}^l\vt_{kj}\ll_{kj}^T\cdot\n_1=\vt_{1,j}\ll^T_{1,j}\cdot\n_1+\sum_{h\ne 1}\vt_{hj}\ll_{hj}^T\cdot\n_1
\end{equation*}
Since $m_1=2$, then $\aa_1=(1,\d_1,\0_{n-1})$ and $\aa=(1,\d_1,\ldots)$ meaning that
\begin{equation*}
  \vt_{1,1}=\d_1\,,\quad \forall\,j\neq 1\quad\vt_{1,j}=1\,,\quad \vt_1\,|\,\d_1\,,\quad \forall\,j\neq 1\quad\vt_{j}=1
\end{equation*}
Therefore, recalling (\ref{b}) and (\ref{b1inequality}), one has
\begin{equation*}
\left.\begin{array}{cc}
  &\ll_1^T\cdot\n_1  \ge  \d_1/\vt_1 \ll_{1,1}^T\cdot\n_1 \ge  -\d_1/\vt_1  \ge  -\d_l \\
  &\ll_2^T\cdot\n_1  \ge  \ll_{1,2}^T\cdot\n_1 \ge  -1  \ge  -\d_l \\
  \forall\,3\le j\le n&\ll_j^T\cdot\n_1  \ge  \ll_{1,j}^T\cdot\n_1 \ge  -\d_1  \ge  -\d_l \\
  &\ll_{n+1}^T\cdot\n_1  \ge  \ll_{1,n+1}^T\cdot\n_1 \ge  -\d_1  \ge  -\d_{l-1}
\end{array}\right\}\ \Longrightarrow\ \L_\aa^T\cdot\n_1\ge-\bb
\end{equation*}
and, consequently,
\begin{equation*}
  \L_\aa^T\cdot\n =\sum_{k=1}^l \mu_k\L_{\aa}^T\cdot\n_k\ge -\sum_{k:m_k\ge 3} \mu_k(\bb_k)_{J_k}-\sum_{k:m_k=2} \mu_k\bb \ge -\sum_{k=1}^l \mu_k\bb =\bb
\end{equation*}
Then $\n\in\D_\bb$ and $\cv{\D}_{\cv{\bb}}\subseteq\D_\bb$.

Finally, assume $m_h=1$, for some $h$. In this case (\ref{Lambdak}) is replaced by the following
\begin{equation}\label{bhinequality}
  \cv{\L}_{\aa_h}^T\cdot\n_h=\left(\L_{\aa_h}^{\{\overline{j}(h)\}}\right)^T\cdot\n_h\ge -(\bb_h)_{J_h}=-\1_n\ ,\quad\forall\,k\neq h\quad\cv{\L}_{\aa_k}^T\cdot\n_h\ge \0
\end{equation}
where $\overline{j}(h)=1+\sum_{i=1}^{h-1}m_i$\,. As above, let $\ll_j$ be the $j$-th columns of $\L_\aa$ and, for $j\ne \overline{j}$, let $\ll_{hj}$ denote the $j$-th column of $\L_{\aa_h}$. Moreover, set $\ll_{h,\overline{j}}=\0_n$.  Then, Minkowski sum in (\ref{Lambda}), up to a suitable reordering of columns of matrices $\L_{\aa_k}$, gives
\begin{equation*}
  \vt_j\ll_j^T\cdot\n_h=\sum_{k=1}^l\vt_{kj}\ll_{kj}^T\cdot\n_h=\vt_{h,j}\ll^T_{h,j}\cdot\n_h+\sum_{k\ne h}\vt_{kj}\ll_{kj}^T\cdot\n_h
\end{equation*}
Since $m_h=1$, then $\aa_h=(\0_{\overline{j}-1},d_h,\0_{n+1-\overline{j}})$ and $\aa=(\,\ldots,\underbrace{d_h}_{\overline{j}\text{-th place}},\ldots\,)$, so giving
\begin{equation*}
  \vt_{h,\overline{j}}=1\,,\quad \forall\,j\neq \overline{j}\quad\vt_{h,j}=d_h\,,\quad \vt_{\overline{j}}=1\,,\quad \forall\,j\neq \overline{j}\quad\vt_{j}\,|\,d_h
\end{equation*}
where $\vt_{\overline{j}}=1$ because there necessarily exists some $k\neq h$ such that $m_k\ge 2$. Assume, at first, that $h\le l-1$ (then $\overline{j}\le n$). Then, by (\ref{b}) and (\ref{bhinequality}), one has
\begin{equation*}
\begin{array}{cc}
  &\ll_{\overline{j}}^T\cdot\n_h  \ge  \ll_{h,\overline{j}}^T\cdot\n_h =  0  \ge  -\d_l \\
  \forall\,j\le n\,,\ j\neq\overline{j} &\ll_j^T\cdot\n_h  \ge  d_h/\vt_j\ll_{h,j}^T\cdot\n_h \ge  -\d_h/\vt_j  \ge  -\d_l \\
  &\ll_{n+1}^T\cdot\n_h  \ge  d_h/\vt_j\ll_{h,n+1}^T\cdot\n_h \ge  -\d_h/\vt_j  \ge  -\d_{l-1}
\end{array}
\end{equation*}
On the other hand, if $h=l$ (and $\overline{j}=n+1$) then
\begin{equation*}
\begin{array}{cc}
  \forall\,j\le n\,,\ j\neq\overline{j} &\ll_j^T\cdot\n_h  \ge  d_h/\vt_j\ll_{h,j}^T\cdot\n_h \ge  -\d_h/\vt_j  \ge  -\d_l \\
  &\ll_{n+1}^T\cdot\n_h  \ge  \ll_{h,n+1}^T\cdot\n_h=0  \ge  -\d_{l-1}
\end{array}
\end{equation*}
Therefore $\L_\aa^T\cdot \n_h\ge -\bb$ and
\begin{equation*}
  \L_\aa^T\cdot\n =\sum_{k=1}^l \mu_k\L_{\aa}^T\cdot\n_k\ge -\sum_{k:m_k\ge 3} \mu_k(\bb_k)_{J_k}-\sum_{k:m_k\le 2} \mu_k\bb \ge -\sum_{k=1}^l \mu_k\bb =\bb
\end{equation*}
Then $\n\in\D_\bb$ and $\cv{\D}_{\cv{\bb}}\subseteq\D_\bb$ and one can finally deduce that $\cv{\D}_{\cv{\bb}}\subseteq\D_\bb$\,.
\end{proof}

Recalling Corollary~\ref{cor:nef}, the previous proof of Theorem~\ref{thm:CI} actually allows one to conclude something more, namely

\begin{corollary}\label{cor:CI}
  Consider the nef partitioned ftv $(\P^n,\aa=\sum_k\aa_k)$, where the nef partitioned framing $\aa$ is defined by (\ref{a,ak}), up to possible permutations of entries of $\aa$ and $\aa_k$. Then the $f$-dual nef partitioned ftv is given by $(\cv{\XX}_\aa,\cv{\bb}=\sum_k\bb_k)$ where $\bb_k$ is defined by (\ref{bkh_nulle}), (\ref{bkh_1}), (\ref{bkh_2}) and (\ref{bkh_3}) and the associated partitioned $f$-process is calibrated.

  \noindent In particular, an $f$-mirror partner of the complete intersection
  $$Y:=\bigcap_{k=1}^lY_k\subset \P^n\quad\text{with}\quad Y_k\in|D_{\aa_k}|$$
  is given by the complete intersection $Y^\vee\subset\cv{\XX}_\aa$ constructed in part (b) of Remark~\ref{rem:CIequazioni}.
\end{corollary}

\begin{remark}\label{rem:LG}
  Recalling the LG/CI-correspondence described in \S~\ref{ssez:LG/CI}, one can construct a generalized LG mirror model $(\cv{\T}_\aa,\ff_\bb^\vee)$ of the projective complete intersection $Y=Y_{d_1,\ldots,d_l}$. Moreover, Proposition~\ref{prop:LGmirror} produces a Givental type LG mirror model $(\pi^{-1}(\T_l),F)$ of the complete intersection $Y\subset\P^n$, given by a suitable re-parametrization of $(\cv{\T}_\aa,\ff_\bb^\vee)$. According with Clarke \cite[\S7.2, Rem.~7.4]{Clarke}, this LG mirror model should be compared with the LG mirror model assigned to a projective complete intersection of positive Kodaira dimension by Hori and Vafa \cite{Hori-Vafa}.
\end{remark}

\subsubsection{Kodaira negative, projective complete intersections}\label{ssez:negKod}
In \cite[Thm.~7.4]{R-fTV} we proposed f-mirror partners for projective hypersurfaces of negative Kodaira dimension, both in terms of hypersurfaces of non-complete toric varieties and of Landau-Ginzburg (LG) models. Moreover, we checked that the latter is consistent with LG models proposed by Givental in \cite[Thm.~5]{Givental-ICM}. The present subsection is devoted to extending these results to the case of Kodaira negative complete intersection in $\P^n$.

\begin{theorem}\label{thm:negKod}
  Let $Y=Y_{d_1,\ldots,d_l}\subset\P^n$ be a generic complete intersection with
  $$d:=\sum_{k=1}^{l}d_k\le n$$
Consider the partitioned weak framing $\aa_d=\sum_{k=1}^{l}\aa_{d_k}$ with
\begin{equation}\label{adk}
  \aa_{d_k}^T=(0,\ldots,0,\overbrace{\underbrace{1,\ldots,1}_{d_k},0,\ldots,0}^{I_k},0,\ldots,0)
\end{equation}
Consider the non-complete toric variety $\cv{\XX}_{\aa_d}$, as defined in step (C) of algorithm~\ref{algoritmoDnef}. Then $(\P^n,\aa_d=\sum_k\aa_{d_k})$ admits an $f$-dual partitioned weak framing given by
  \begin{eqnarray*}
    \bb&=&\sum_{k=1}^l\bb_k\quad\text{with}\\
    \bb^T_k&=&
    \begin{cases}
      (0,\ldots,0,\overbrace{0,1,\ldots,1}^{J_k},0,\ldots,0)& \text{if}\ d_k=1\\
      (0,\ldots,0,\underbrace{1,\ldots,1}_{J_k},0,\ldots,0)& \text{if}\ d_k\ge 2
    \end{cases}\\
\bigsqcup_{k=1}^l J_k&=&\{1,\ldots,\cv{m}\}
\end{eqnarray*}
recalling that $\cv{m}$ is the columns number of the fan matrix $\cv{\L}_{\aa_d}$ of $\cv{\XX}_{\aa_d}$. \\
In particular, an $f$-mirror partner of $Y$ is given by the complete intersection $Y^\vee$ in $\cv{\XX}_{\aa_d}$ constructed in Definition~\ref{def:wf-mirror}.
\end{theorem}

\begin{proof}
  To prove that $\bb=\sum_{k=1}^l\bb_k$, as defined in the statement, is an $f$-dual partitioned weak framing, one has to apply steps (C) and (D) of algorithm~\ref{algoritmoDnef}. Then the proof goes on as for Theorem~\ref{thm:CI}, so obtaining a null column in $\L_{\aa_k}$ if and only if $|I_k|=1$. In particular, for any $k$, matrix $\L_{\aa_k}$ looks like matrix $\L_{\aa_d}$ in the proof of \cite[Thm.~7.4]{R-fTV}. Constructing the associated matrix $\cv{\L}_{\aa_k}$ and recalling Remark~\ref{rem:cvLambdak}, one gets $\bb$ as above.
\end{proof}

\begin{remark}
  Analogously to what observed in Remark~\ref{rem:LG}, recalling the LG/CI-correspondence described in \S~\ref{ssez:LG/CI}, one can construct a generalized LG mirror model $(\cv{\T}_\aa,\ff_\bb^\vee)$ of the projective complete intersection $Y=Y_{d_1,\ldots,d_l}$ with $d=\sum_kd_k\leq n$. Moreover, Proposition~\ref{prop:LGmirror} produces a Givental type LG mirror model $(\pi^{-1}(\T_l),F)$ of the complete intersection $Y\subset\P^n$, given by a suitable re-parametrization of $(\cv{\T}_\aa,\ff_\bb^\vee)$, so generalizing to complete intersections what stated for hypersurfaces in \cite[Thm.~7.4]{R-fTV}.
\end{remark}

\begin{example}[$Y_{2,3}\subset\P^5$]\label{ex:Y23}
  Consider the bi-degree $(2,3)$ in $\P^5$, giving the maximum degree projective 3-dimensional complete intersection, with $l=2$, of negative Kodaira dimension. Recalling (\ref{adk}), choose the weak partitioned framing of $\P^5$ given by
  \begin{eqnarray*}
    \aa^T&:=&\left(\begin{array}{cccccc}
          1 & 1 & 0 & 1 & 1 & 1
        \end{array}\right)=\aa_1^T+\aa_2^T\quad\text{with}\\
        \aa_1^T&:=&\left(
                                    \begin{array}{cccccc}
                                      1 & 1 & 0 & 0 & 0 & 0 \\
                                    \end{array}
                                  \right)\\
        \aa_2^T&:=&\left(
                                                          \begin{array}{cccccc}
                                                            0 & 0 & 0 & 1 & 1 & 1 \\
                                                          \end{array}
                                                        \right)
  \end{eqnarray*}
  Then step (A) of algorithm~\ref{algoritmoDnef} gives $\D_{\aa_k}=\conv(\L_{\aa_k})$, with $k=1,2$ and
  \begin{equation*}
    \L_{\aa_1}=\left( \begin {array}{cccccc} 1&-1&-1&-1&-1&-1\\
    -1&1&-1&-1&-1&-1\\
    0&0&2&0&0&0\\
    0&0&0&2&0&0\\
    0&0&0&0&2&0\end {array} \right)
  \end{equation*}
  \begin{equation*}
    \L_{\aa_2}=\left( \begin {array}{cccccc} 3&0&0&0&0&0\\
    0&3&0&0&0&0\\
    0&0&3&0&0&0\\
    -1&-1&-1&2&-1&-1\\
    -1&-1&-1&-1&2&-1\end {array} \right)
  \end{equation*}
   Therefore $\cv{\D}_\aa=\conv(\D_{\aa_1},\D_{\aa_2})=\conv\left(\cv{\L}_\aa\right)$, with $\cv{\L}_\aa=\left(
                                                                                               \begin{array}{c}
                                                                                                 \L_{\aa_1} \,\vline\ \L_{\aa_2} \\
                                                                                               \end{array}
                                                                                             \right)
                                                                                             $.
  Step (D) in algorithm~\ref{algoritmoDnef} then gives
  \begin{eqnarray*}
    \cv{\bb}^T&:=&\left(\begin{array}{cccccccccccc}
          1 & 1 & 1 & 1 & 1 & 1 & 1 & 1 & 1 & 1 & 1 & 1\\
        \end{array}\right)=\bb_1^T+\bb_2^T\quad\text{with}\\
        \bb_1^T&:=&\left(
                                    \begin{array}{cccccccccccc}
                                      1 & 1 & 1 & 1 & 1 & 1 & 0 & 0 & 0 & 0 & 0 & 0\\
                                    \end{array}
                                  \right)\\
        \bb_2^T&:=&\left(
                                                          \begin{array}{cccccccccccc}
                                                            0 & 0 & 0 & 0 & 0 & 0 & 1 & 1 & 1 & 1 & 1 & 1 \\
                                                          \end{array}
                                                        \right)
  \end{eqnarray*}
  Then, apply Remark~\ref{rem:CIequazioni(wf)}, to get
  \begin{eqnarray*}
    \overline{\L}_{\bb_1} &=& \left(V_{1,2,3}\,|\,\0_5\right) \\
    \overline{\L}_{\bb_2} &=& \left(V_{4,5,6}\,|\,\0_5\right)
  \end{eqnarray*}
  being $V$ the fan matrix of $\P^5$, as given in display (\ref{V}) by setting $n=5$. Therefore, by Definition~\ref{def:wf-mirror}, construct matrices $\overline{M}_{\aa,\bb_k}$, for $k=1,2$, to get, up to  a variables rescaling, the following polynomials in $\Cox(\cv{\XX}_\aa)\cong\C[x_1,\ldots,x_{12}]$
  \begin{eqnarray*}
    f^\vee_1 &=&  \left(\prod_{i=1}^6x_i\right)(\psi+x_3^2x_9^3)+x_1^2x_7^3+x_2^2x_8^3\\
    f^\vee_2 &=&  \prod_{i=7}^{12}x_i+x_4^2x_{10}^3+x_5^2x_{11}^3+x_6^2x_{12}^3
  \end{eqnarray*}
  where $\psi$ is the unique complex modulus of the $f$-mirror family whose generic e\-le\-ment is given by $Y^\vee=Y^\vee_1\cap Y^\vee_2$, being $Y^\vee_k$ the hypersurface of $\cv{\XX}_\aa$ assigned by the quotient, by the Cox action of $\Hom(\Cl(\cv{\XX}_\aa),\C^*)$, of the zero locus of $f^\vee_k$.\\
  The associated generalized LG mirror model of $Y$ is then given by $(\cv{\T}_\aa, \ff_\bb^\vee)$, where
  \begin{equation*}
    \ff^\vee_\bb=\left({f^\vee_1\over\x^{\bb_1}},{f^\vee_2\over\x^{\bb_2}}\right)=\left(\psi+x_3^2x_9^3+
    {x_1^2x_7^3+x_2^2x_8^3\over\prod_{i=1}^6x_i}, 1+{x_4^2x_{10}^3+x_5^2x_{11}^3+x_6^2x_{12}^3\over\prod_{i=7}^{12}x_i}\right)
  \end{equation*}
  Therefore, Proposition~\ref{prop:LGmirror} gives the Givental type LG mirror model $(\pi^{-1}(\T_2),F)$ with
  \begin{eqnarray*}
    F &=& {f^\vee_1\over\x^{\bb_1}}+{f^\vee_2\over\x^{\bb_2}}= u_{11}+\cdots+u_{41}+u_{12}+\cdots+u_{42}:\C^8\longrightarrow\C\\
    \pi(\uu) &=&(q_1,q_2):(\C^*)^8\longrightarrow(\C^*)^2\\
    u_{11}&=&\psi\\
    u_{21}&=&x_3^2x_9^3\\
    u_{31}&=&{x_1^2x_7^3\over x_1x_2x_3x_4x_5x_6}\\
    u_{41}&=&{x_2^2x_8^3\over x_1x_2x_3x_4x_5x_6}\\
    u_{12}&=& 1\\
    u_{22}&=& {x_4^2x_{10}^3\over x_7x_8x_9x_{10}x_{11}x_{12}}\\
    u_{32}&=& {x_5^2x_{11}^3\over x_7x_8x_9x_{10}x_{11}x_{12}}\\
    u_{42}&=&{x_6^2x_{12}^3\over x_7x_8x_9x_{10}x_{11}x_{12}}\\
    q_1&=&u_{11}\cdots u_{41}=\psi{x_7^3x_8^3x_9^3\over x_4^2x_5^2x_6^2}\\
    q_2&=&u_{12}\cdots u_{42}={x_4^2x_5^2x_6^2\over x_7^3x_8^3x_9^3}
  \end{eqnarray*}
  Notice that $q_1q_2=\psi$, so that $\pi$ is actually a fibration over the complex parameter of the mirror family, up to the composition with $(q_1,q_2)\in\T_2\mapsto q_1q_2\in\C^*$.
\end{example}

\begin{example}[$Y_{1,3}\subset\P^5$]\label{ex:Y13} Choosing the bi-degree $(1,3)$ in $\P^5$, what is differing from the previous Example~\ref{ex:Y23} about the bi-degree $(2,3)$, is:
\begin{eqnarray*}
 \aa^T&:=&\left(\begin{array}{cccccc}
          1 & 0& 0 & 1 & 1 & 1
        \end{array}\right)=\aa_1^T+\aa_2^T\quad\text{with}\\
        \aa_1^T&:=&\left(
                                    \begin{array}{cccccc}
                                      1 & 0 & 0 & 0 & 0 & 0 \\
                                    \end{array}
                                  \right)\\
\end{eqnarray*}
Then
\begin{equation*}
   \L_{\aa_1}=\left( \begin {array}{ccccc} -1&-1&-1&-1&-1\\
    1&0&0&0&0\\
    0&1&0&0&0\\
    0&0&1&0&0\\
    0&0&0&1&0\end {array} \right)
\end{equation*}
and
\begin{equation*}
  \bb_1^T:=\left(
                                    \begin{array}{cccccccccccc}
                                      1 & 1 & 1 & 1 & 1 &  0 & 0 & 0 & 0 & 0 & 0\\
                                    \end{array}
                                  \right)
\end{equation*}
Therefore $\Cox(\cv{\XX}_\aa)\cong \C[x_1,\ldots,x_{11}]$ and
\begin{eqnarray*}
    f^\vee_1 &=&  \left(\prod_{i=1}^5x_i\right)(\psi+x_1x_7^3+x_2x_8^3)+x_6^3\\\\
    f^\vee_2 &=&  \prod_{i=6}^{11}x_i+x_3x_{9}^3+x_4x_{10}^3+x_5x_{11}^3  \end{eqnarray*}
The associated Givental type LG mirror model is then assigned by the reparameterization
\begin{eqnarray*}
  u_{11}&=&\psi\\
    u_{21}&=&x_1x_7^3\\
    u_{31}&=&x_2x_8^3\\
    u_{41}&=&{x_6^3\over x_1x_2x_3x_4x_5}\\
    u_{12}&=& 1\\
    u_{22}&=& {x_3x_{9}^3\over x_6x_7x_8x_9x_{10}x_{11}}\\
    u_{32}&=& {x_4x_{10}^3\over x_6x_7x_8x_9x_{10}x_{11}}\\
    u_{42}&=&{x_5x_{11}^3\over x_6x_7x_8x_9x_{10}x_{11}}\\
    q_1&=&u_{11}\cdots u_{41}=\psi{x_7^3x_8^3x_9^3\over x_3x_4x_5}\\
    q_2&=&u_{12}\cdots u_{42}={x_3x_4x_5\over x_7^3x_8^3x_9^3}
\end{eqnarray*}
\end{example}

\section{Computing Hodge numbers}\label{sez:hdgnbrs}

 Given a normal algebraic variety $X$ there is a well defined sheaf
 $$\widehat{\Omega}_X:=i_*\Omega_X$$
  called the \emph{sheaf of Zariski 1-forms}, being $i$ the open embedding of the smooth locus of $X$.

\subsection{Preliminary results}
Let us recall, and apply to our situation, some results due to Batyrev-Cox \cite{BC} and Batyrev-Borisov \cite{BB96}.

\begin{theorem}[Thm.~12.1 in \cite{BC}, Thm.~8.1.6 in \cite{CLS}]\label{thm:BC}
  Let $X(\Si)$ be a $\Q$-factorial toric variety with no torus factors. Then there is an exact sequence
  \begin{equation*}
    \xymatrix{0\ar[r]&\widehat{\Omega}_X\ar[r]&\bigoplus_{\rho\in\Si(1)}\cO_X(-D_\rho)\ar[r]&\Cl(X)\otimes_\Z\cO_X\ar[r]&0}
  \end{equation*}
\end{theorem}

\begin{theorem}[Danilov Vanishing, Prop.~12.11 in \cite{Danilov}, Thm.~9.3.2 in \cite{CLS}]\label{thm:Danilov} If $X(\Si)$ is a complete and $\Q$-factorial toric variety then
\begin{equation*}
  \forall\,p\neq q\quad H^p(X,\widehat{\Omega}^q_X)=0
\end{equation*}

\end{theorem}

\begin{theorem}[Batyrev-Borisov Vanishing, Thm.~2.5 in \cite{BB96} and Thm.~9.2.7 in \cite{CLS}]\label{thm:BBvanish}
  Let $D=\sum_\rho a_\rho D_\rho$ be a $\Q$-Cartier divisor on a complete toric variety $X(\Si)$. If $D$ is semi-ample then
  \begin{equation*}
    h^p(X,\cO_X(-D))=\left\{\begin{array}{cc}
                              l^*(\D_D) & \text{for $p=\dim\D_D$} \\
                              0 & \text{otherwise}
                            \end{array}
    \right.
  \end{equation*}
\end{theorem}

\begin{proposition}[nef and big partitioned framing]\label{prop:nef&big}
  Let $(X,\aa=\sum_{k=1}^l\aa_k)$ be a nef partitioned ftv such that $D_{\aa_k}$ is a big divisor, for every $k$. Consider the generic complete intersection
  $$Y:=\bigcap_{k=1}^lY_k\subset X\quad\text{with}\quad Y_k\in|D_{\aa_k}|$$
  Then, for every semi-ample and big $\Q$-Cartier divisor $D$ of $X$
  \begin{equation*}
    h^p(Y,\cO_Y(-D))=\left\{\begin{array}{cc}
                              \sum\limits_{J\subseteq\{1,\ldots,l\}}(-1)^{l-|J|}l^*\left(\D_{D}+\sum_{j\in J}\D_{\aa_{j}}\right) & \text{for $p=n-l$} \\
                              0 & \text{otherwise}
                            \end{array}
    \right.
  \end{equation*}
\end{proposition}

\begin{proof}
  A proof of this proposition goes exactly as for \cite[Prop.~8.1]{BB96}. Here we will re-propose the whole argument essentially for the following three reasons:
  \begin{itemize}
    \item[-] the given statement is more general than the one given in \cite[Prop.~8.1]{BB96},
    \item[-] the same argument will be often applied in proving many results presented in this paper,
    \item[-] it is not explicitly described in the proof of \cite[Prop.~8.1]{BB96}.
  \end{itemize}
  Start by considering the following \emph{Koszul-type complex}
  \begin{eqnarray*}
        \Ksz_D :& \xymatrix{0\ar[r]&\cO_{X}\left(-D-D_{\aa}\right)\ar[r]^-{\d_0}&\cdots
        \bigoplus\limits_{\{k_1,\ldots,k_{l-p}\}\subseteq\{1,\ldots,l\}}
        \cO_{X}\left(-D-\sum\limits_{j=1}^{l-p}D_{\aa_{k_j}}\right)}&\\
        &\cdots\xymatrix{\ar[r]^-{\d_{l-2}}&
        \bigoplus\limits_{k=1}^l\cO_{X}\left(-D-D_{\aa_k}\right)\ar[r]^-{\d_{l-1}}&\cO_{X}(-D)}&
      \end{eqnarray*}
      giving rise to an acyclic resolution of the direct image sheaf $j_*\cO_Y(-D)$, under the closed embedding $j:Y\hookrightarrow X$  (notation as in \cite[\S II.1]{Hartshorne}). Then, the $q$-th sheaf of cohomology of $\Ksz_D$ is given by
      \begin{equation}\label{sheafcoh}
        \mathcal{H}^q(\Ksz_D)\cong\left\{\begin{array}{cc}
                                         j_*\cO_Y(-D) & \text{for}\ q=l \\
                                         0 & \text{otherwise}
                                       \end{array}
        \right.
      \end{equation}
      There are two spectral sequences $'E$ and $''E$ abutting to the hypercohomology $\H^*(\Ksz_D)$ (see e.g. \cite[\S 3.5]{GH}) whose bi-graded second level terms are given by
      \begin{eqnarray*}
        'E_2^{p,q} &=& H^p(X,\mathcal{H}^q(\Ksz_D)) \\
        ''E_2^{p,q} &=& H^q(X,H^p(\Ksz_D))
      \end{eqnarray*}
      On the one hand, by (\ref{sheafcoh}), the first spectral sequence degenerates to giving
      \begin{equation*}
        \H^{p+l}(\Ksz_D)\cong\ 'E_2^{p,l}=H^p(X,j_*\cO_Y(-D))\cong H^p(Y,\cO_Y(-D))
      \end{equation*}
      On the other hand, notice that
      \begin{equation*}
        ''E_1^{p,q}=\bigoplus_{\{k_1,\ldots,k_{l-p}\}\subseteq\{1,\ldots,l\}} H^q\left(X,\cO_{X}\left(-D-\sum_{j=1}^{l-p} D_{\aa_{k_j}}\right)\right)
      \end{equation*}
      so that the Batyrev-Borisov vanishing~\ref{thm:BBvanish} gives
      \begin{equation*}
        \dim\left(''E_1^{p,q}\right)=\left\{\begin{array}{cc}
                                              \sum\limits_{\{k_1,\ldots,k_{l-p}\}\subseteq\{1,\ldots,l\}}l^* \left(\D_D+\sum\limits_{j=1}^{l-p}\D_{\aa_{k_j}}\right) &  \text{for}\ q=n \\
                                              0 &  \text{otherwise}
                                            \end{array}
        \right.
      \end{equation*}
      This suffices to show that $''E$ degenerates at the second level, as
      \begin{equation*}
        ''E_2^{p,q}={\ker(d_1:\,''E_1^{p,q}\longrightarrow\,''E_1^{p+1,q})
        \over\im(d_1:\,''E_1^{p-1,q}\longrightarrow\,''E_1^{p,q})}={\ker H^q(\d_p)\over \im H^q(\d_{p-1})}=0\quad\text{for}\ q\neq n
      \end{equation*}
      and
      \begin{equation*}
        d_2:\,''E_2^{p,n}\longrightarrow\,''E_2^{p+2,n-1}=0
      \end{equation*}
      Therefore
      \begin{eqnarray*}
        H^p(Y,\cO_Y(-D))&\cong&\H^{p+l}(\Ksz_D)\cong\,''E_2^{p+l-n,n}\\
                        &\cong&{\ker H^n(\d_{p+l-n})\over \im H^n(\d_{p+l-n-1})}=0\quad\text{for}\ p\neq n-l
                        \end{eqnarray*}
      To explicitly compute $H^{n-l}(Y,\cO_Y(-D))\cong \ker H^n(\d_0)$, consider the following relation on Euler characteristics  associated with the given Koszul resolution:
      \begin{equation}\label{euler}
        (-1)^{l+1}\chi\left(j_*\cO_Y(-D)\right)+\sum_{p=0}^{l}(-1)^p\chi\left(\mathcal{F}_p\right)=0
      \end{equation}
      where
      \begin{equation*}
        \mathcal{F}_p=\bigoplus\limits_{\{k_1,\ldots,k_{l-p}\}\subseteq\{1,\ldots,l\}}
        \cO_{X}\left(-D-\sum\limits_{j=1}^{l-p}D_{\aa_{k_j}}\right)
      \end{equation*}
      Since
      \begin{eqnarray*}
        \chi\left(j_*\cO_Y(-D)\right) &=& (-1)^{n-l}h^{n-l}(\cO_Y(-D)) \\
        \forall\,p=0,\ldots,l\hskip1truecm\chi\left(\mathcal{F}_p\right) &=& (-1)^n h^n\left(\mathcal{F}_p\right)
        \end{eqnarray*}
        then, (\ref{euler}) and Batyrev-Borisov vanishing~\ref{thm:BBvanish} gives
        \begin{eqnarray*}
          h^{n-l}(\cO_Y(-D))&=&\sum_{p=0}^{l}(-1)^ph^n(\mathcal{F}_p)\\
                            &=&\sum_{p=0}^{l}(-1)^p\sum\limits_{\{k_1,\ldots,k_{l-p}\}\subseteq\{1,\ldots,l\}} l^*\left(\D_D+\sum_{j=0}^{l}\D_{\aa_{k_j}}\right)\\
                            &=&\sum\limits_{J\subseteq\{1,\ldots,l\}}(-1)^{l-|J|}l^*\left(\D_{D}+\sum_{j\in J}\D_{\aa_{j}}\right)
        \end{eqnarray*}
\end{proof}

\begin{definition}[$h$-dependence of polyhedra; see Def.~3.1 in \cite{BB96}]
  Given $l$ polyhedra $\D_1\,\ldots,\D_l$, they are called \emph{$h$-dependent} if there exist a positive $s\in\N$ and an $s$-element subset
  \begin{equation*}
    \{\D_{i_1},\ldots,\D_{i_s}\}\subseteq\{\D_1\,\ldots,\D_l\}
  \end{equation*}
  such that
  \begin{equation*}
    \dim\left(\sum_{j=1}^s\D_{i_j}\right)< s+h-1
  \end{equation*}
\end{definition}

\begin{theorem}[Thm.~3.3 in \cite{BB96}]\label{thm:h-indipendenza}
  Given the same notation as in Proposition~\ref{prop:nef&big}, but without necessarily assuming bigness for $D_{\aa_k}$, then:
  \begin{itemize}
    \item[(i)] $Y\neq\emptyset$ if and only if $\D_{\aa_1},\ldots,\D_{\aa_l}$ are $1$-independent,
    \item[(ii)] if $\D_{\aa_1},\ldots,\D_{\aa_l}$ are $2$-independent then $Y$ is irreducible,
    \item[(iii)] if $\D_{\aa_1},\ldots,\D_{\aa_l}$ are $h$-independent, with $h\ge 3$, then
    \begin{equation*}
      h^p(\cO_Y)=0 \quad\text{for}\quad 1\le p\le h-2
    \end{equation*}
  \end{itemize}
\end{theorem}

\begin{corollary}[Cor.~3.5 in \cite{BB96}]\label{cor:hpO}
  Assume that $\D_{\aa_1},\ldots,\D_{\aa_l}$ have positive dimension.
  Then:
  \begin{enumerate}
    \item $Y=\emptyset$ if and only if $l=n+1$,
    \item $Y$ consists of a finite number of points if and only if $l=n$,
    \item $Y$ is an irreducible variety of dimension $n-l\ge 1$ if and only if $l\le n-1$: in particular $h^0(\cO_Y)=1$.
  \end{enumerate}
  Moreover, if for any proper subset $\{k_1,\ldots,k_s\}\subset\{1,\ldots,l\}$ it happens that
  \begin{equation*}
    l^*\left(\sum_{j=1}^s \D_{\aa_{k_j}}\right)=0
  \end{equation*}
  then:
  \begin{enumerate}
  \setcounter{enumi}{3}
    \item $Y$ is an irreducible variety of dimension $n-l\ge 2$, with $$h^1(\cO_Y)=\cdots=h^{n-l-1}(\cO_Y)=0$$
        if and only if $l\le n-2$\,.
  \end{enumerate}
\end{corollary}

\subsection{Hodge numbers of projective complete intersections}\label{ssez:ProjCI}
We are now in a position to prove the following

\begin{theorem}\label{thm:h21Y}
  Let $Y\subset\P^n$ be the general complete intersection
defined in Theorem~\ref{thm:CI}, with $n\ge 4$ and $\dim(Y)=n-l\ge 2$, and consider the partitioned framing $\aa=\sum_k\aa_k$ of $\P^n$ defined in (\ref{a,ak}).
  Then $Y$ is irreducible and
  \begin{equation*}
    h^p(\cO_Y)=\left\{\begin{array}{cc}
                        1 & \text{for $p=0$} \\
                        \displaystyle\sum_{J\subseteq\{1,\ldots,l\}}(-1)^{l-|J|}l^*\left(\sum_{j\in J}\D_{\aa_j}\right) & \text{for $p=n-l$}\\
                        0 & \text{otherwise}
                      \end{array}
    \right.
  \end{equation*}
  \begin{equation*}
    h^p(\Omega_Y)=\left\{\begin{array}{cc}
                           1 & \text{for $p=1$ and $n-l\ge 3$} \\
                           1+\displaystyle\sum_{J\subseteq\{1,\ldots,l\}}(-1)^{l-|J|} K(\aa,J) & \text{for $p=1$ and $n-l=2$} \\
                           \displaystyle\sum_{J\subseteq\{1,\ldots,l\}}(-1)^{l-|J|} K(\aa,J) & \text{for $p=n-l-1$ and $n-l\ge 3$} \\
                           0 & \text{otherwise}
                         \end{array}
    \right.
  \end{equation*}
  with
  $$K(\aa,J)=l^*\left(\sum_{j\in J}\D_{\aa_j}\right)+\sum_{k=1}^l l^*\left(\D_{\aa_k}+\sum_{j\in J}\D_{\aa_{j}}\right)-\sum_{i=1}^{n+1}l^*\left(\D_i+\sum_{j\in J}\D_{\aa_{j}}\right)$$
  being $\D_i:=\D_{D_i}$  the lattice polytope associated with the prime torus invariant divisor $D_i$.
  \end{theorem}

  \begin{proof} Since $Y$ is general, it is smooth and an iteration of the Hyperplane Lefschetz Theorem gives all the Hodge numbers in the statement unless $h^{n-l}(\cO_Y)$ and $h^{n-l-1}(\Omega_Y)$.
    In particular, with the partitioned framing $\aa=\sum_k\aa_k$ of $\P^n$ defined in (\ref{a,ak}), condition $n-l\ge 2>1$ implies that $Y$ is non-empty and irreducible, by item (3) in Corollary~\ref{cor:hpO}. Moreover, item (iii) in Theorem~\ref{thm:h-indipendenza} gives a further way to computing $h^p(\cO_Y)$ for $1\le p\le n-l-1$, as $D_{\aa_k}$ is an ample divisor of $\P^n$, implying that $\D_{\aa_1},\ldots,\D_{\aa_l}$ are $(n-l+1)$-independent. Proposition~\ref{prop:nef&big} implies that
    \begin{equation*}
      h^p(Y,\cO_Y(-D_{\aa_k}))=0\quad\text{for $p\neq n-l$}
    \end{equation*}
    Then, the following exact sequence
    \begin{equation}\label{seqY}
  \xymatrix{0\ar[r]&\bigoplus_{k=1}^l\cO_{Y}(-D_{\aa_k})\ar[r]&\Omega_{\P^n|Y}\ar[r]&
\Omega_Y\ar[r]&0}
\end{equation}
    gives
    \begin{eqnarray*}
      h^1(\Omega_Y)&=&h^{1}(\Omega_{\P^n|Y})\quad\text{for $n-l\ge 3$}\\
      h^{n-l-1}(\Omega_{Y})&=&h^{n-l-1}(\Omega_{\P^n|Y})+\sum_k h^{n-l}(\cO_{Y}(-D_{\aa_k}))-h^{n-l}(\Omega_{\P^n|Y})
      \end{eqnarray*}
Again Proposition~\ref{prop:nef&big} gives
      \begin{equation*}
      h^{n-l}(\cO_{Y}(-D_{\aa_k}))=\sum_{J\subseteq\{1,\ldots,l\}}(-1)^{l-|J|}l^*\left(\D_{\aa_k}+\sum_{j\in J}\D_{\aa_{k_j}}\right)
      \end{equation*}
      and, noticing that $D_i$ is an ample divisor of $\P^n$,
      \begin{equation*}
      h^{p}(\cO_{Y}(-D_i))=\left\{\begin{array}{cc}
                                    \sum_{J\subseteq\{1,\ldots,l\}}(-1)^{l-|J|}l^*\left(\D_i+\sum_{j\in J}\D_{\aa_{j}}\right) & \text{for}\ p=n-l \\
                                    0 & \text{otherwise}
                                  \end{array}
\right.
      \end{equation*}
      where $\D_i:=\D_{D_i}$.
      Theorem~\ref{thm:BC} gives the following exact sequence
      \begin{equation}\label{seqP}
  \xymatrix{0\ar[r]&\Omega_{\P^n|Y}\ar[r]&\bigoplus_{i=1}^{n+1}\cO_{Y}(-D_i)\ar[r]&
\cO_{Y}\ar[r]&0}
\end{equation}
      so that
      \begin{eqnarray*}
        h^1(\Omega_{\P^n|Y})&=&h^0(\cO_Y)=1 \\
        h^{n-l-1}(\Omega_{\P^n|Y}) &=& 0 \quad\text{for}\ n-l\ge 3\\
        h^{n-l}(\Omega_{\P^n|Y}) &=& \sum_{i=1}^{n+1}h^{n-l}(\cO_Y(-D_i))-h^{n-l}(\cO_Y)
      \end{eqnarray*}
      Putting all together, one has
      \begin{equation}\label{h^n-l-1}
        h^{n-l-1}(\Omega_{Y})=\left\{\begin{array}{cc}
                                       1+h^{2}(\cO_Y) + K_\aa & \text{for $n-l=2$} \\
                                       h^{n-l}(\cO_Y) + K_\aa & \text{for $n-l\ge 3$}
                                     \end{array}
        \right.
      \end{equation}
      with
      \begin{equation}\label{Ka}
        K_\aa= \sum_{J\subseteq\{1,\ldots,l\}}(-1)^{l-|J|}\left(\sum_{k=1}^l l^*\left(\D_{\aa_k}+\sum_{j\in J}\D_{\aa_{j}}\right)-\sum_{i=1}^{n+1}l^*\left(\D_i+\sum_{j\in J}\D_{\aa_{j}}\right)\right)
      \end{equation}
      It remains then to show that
      \begin{equation}\label{hOY}
        h^{n-l}(\cO_Y)=\displaystyle\sum_{J\subseteq\{1,\ldots,l\}}(-1)^{l-|J|}l^*\left(\sum_{j\in J}\D_{\aa_j}\right)
      \end{equation}
      At this purpose, let $j:Y\hookrightarrow\P^n$ be the projective embedding of $Y$ as a complete intersection of $\P^n$, and consider the Koszul complex
      \begin{eqnarray*}
        \Ksz :& \xymatrix{0\ar[r]&\cO_{\P^n}\left(-D_{\aa}\right)\ar[r]&\cdots
        \bigoplus\limits_{\{k_1,\ldots,k_{l-p}\}\subseteq\{1,\ldots,l\}}
        \cO_{\P^n}\left(-\sum\limits_{j=1}^{l-p}D_{\aa_{k_j}}\right)\cdots}&\\
        &\xymatrix{\ar[r]&\bigoplus\limits_{1\le j< k\le l}\cO_{\P^n}\left(-D_{\aa_j}-D_{\aa_k}\right)\ar[r]&
        \bigoplus\limits_{k=1}^l\cO_{\P^n}\left(-D_{\aa_k}\right)\ar[r]&\cO_{\P^n}}&
      \end{eqnarray*}
      which is an acyclic resolution of the direct image sheaf $j_*\cO_Y$. The same argument proving Proposition~\ref{prop:nef&big} shows that $h^p(\cO_Y)=0$, for every $p\neq 0,n-l$, and (\ref{hOY}) is obtained by constraints on the Euler characteristics involved in this Koszul resolution of $j_*\cO_Y$.
  \end{proof}

  \subsection{Hodge numbers of $f$-mirror partners of projective c.i.}\label{ssez:mirrorCI}

  Let
  $$\left(\cv{\XX}_\aa,\cv{\bb}=\sum_{k=1}^l\bb_k\right)$$
  be the nef partitioned ftv defined in Corollary~\ref{cor:CI} and $Y^\vee\subset\cv{\XX}_\aa$ be the $f$-dual partner of the projective complete intersection $Y\subset\P^n$, as described in Definition~\ref{def:nefpmirror}. By construction, $Y^\vee$ is the complete intersection
  $$Y^\vee:=\bigcap_{k=1}^l Y^\vee_k\subset \cv{\XX}_\aa\quad\text{with}\quad Y_k\in|\cv{D}_{\bb_k}|$$
 In particular, $Y^\vee_k$ is defined by a polynomial $f^\vee_k\in\Cox(\cv{\XX}_\aa)$, as explicitly  described in item (b) of \cite[Rem.~6.6]{R-fTV}.
Let $\phi:\XX\longrightarrow\cv{\XX}_\aa$ be a toric (partial) resolution of singularities and set
 \begin{equation}\label{Ytrasformata}
   \widehat{Y}^\vee_k:=\phi^{-1}(Y^\vee_k)\ ,\quad \widehat{Y}^\vee:=\phi^{-1}(Y^\vee)=\bigcap_{k=1}^l\widehat{Y}^\vee_k
 \end{equation}
being $\widehat{Y}^\vee_k\in\left|\phi^{-1}\left(\cv{D}_{\bb_k}\right)\right|$ the hypersurface of $\XX$ defined by $\phi^*f^\vee_k=f^\vee_k\circ\phi\in\Cox(\XX)$. The fan $\Xi$, defining $\XX$, is a suitable subdivision of the fan $\cv{\Si}_\aa$, defining $\cv{\XX}_\aa$. Then, the birational map $\phi$ turns out to be induced by the identity $\id:M\longrightarrow M$, whose extension to $M_\R$ gives rise to the inclusion of fans $\Xi\prec\cv{\Si}_\aa$. Dually, also lattices $N$ can be identified, making sense of the following relations between associated polytopes
\begin{eqnarray}\label{politopi}
  \nonumber
  \forall\,k=1,\ldots,l\quad \D_{\phi^{-1}(\bb_k)}&=&\cv{\D}_{\bb_k}\\
                        \D_{\phi^{-1}(\cv{\bb})}&=&\sum_{k=1}^l\D_{\phi^{-1}(\bb_k)}=
                        \sum_{k=1}^l\cv{\D}_{\bb_k}=\D_{\cv{\bb}}
\end{eqnarray}

\begin{theorem}\label{thm:h21Yd}
Given notation above, assume $n-l\ge 2$. Then, $\widehat{Y}^\vee$ is an irreducible variety of dimension $n-l$, with
\begin{equation}\label{h0dY}
    h^p(\cO_{\rdY})=h^p(\cO_{Y^\vee})=\left\{\begin{array}{cc}
                        1 & \text{for $p=0,n-l$} \\
                        0 & \text{otherwise}
                      \end{array}
    \right.
  \end{equation}
  Moreover, if
  \begin{itemize}
    \item[-] $\XX$ is $\Q$-factorial,
    \item[-] $\rk(\Cl(\XX))$ is big enough (see the following Remark~\ref{rem:>>0}),
    \item[-] for any integer $q\ge 2$, $n-l\ge q+1$ and $\cv{\D}_{\bb_1},\ldots,\cv{\D}_{\bb_l}$ are $h$-independent with $h\ge q+2$,
  \end{itemize}
  then
  \begin{equation*}
    h^p\left(\widehat{\Omega}_{\rdY}\right)=\left\{\begin{array}{cc}
                                \rk\left(\Cl(\XX)\right)=h^p\left(\widehat{\Omega}_{\XX}\right) & \text{for $p=1$ and $n-l\ge 3$} \\
                                0=h^p\left(\widehat{\Omega}_{\XX}\right) & \text{for either $p=0$ or $q\ge 3$ and $2\le p\le q-1$}
                              \end{array}\right.
  \end{equation*}
\end{theorem}

\begin{proof}
  In the following, calling $\phi:\XX\longrightarrow\cv{\XX}_\aa$ the considered (partial) resolution of $\cv{\XX}_\aa$, denote
  \begin{eqnarray*}
    \forall\,k=1,\ldots,l\quad\bg_k&:=&\phi^{-1}(\bb_k)\,,\quad\text{that is,}\quad \widehat{D}_{\bg_k}:=\phi^{-1}\left(\cv{D}_{\bb_k}\right)\\
    \bg&:=&\phi^{-1}\left(\cv{\bb}\right)\,,\quad\text{that is,}\quad \widehat{D}_{\bg}:=\phi^{-1}\left(\cv{D}_{\cv{\bb}}\right)=\sum_{k=1}^l \widehat{D}_{\bg_k}
  \end{eqnarray*}
  After Proposition~\ref{prop:nef}, we may assume $\widehat{D}_{\bg_k}$ to be semi-ample, but in general it is not big, meaning that Proposition~\ref{prop:nef&big} cannot directly apply.

  Let us start by computing $h^p(\cO_{\rdY})$. Recalling the definition of $\cv{\D}_{\bb_k}$, given in item (E) of \ref{algoritmoDnef}, and the construction of $\bb_k$, given in (\ref{bkh_nulle}), (\ref{bkh_1}), (\ref{bkh_2}) and (\ref{bkh_3}, it turns out that
  \begin{enumerate}
    \item $\0$ is always on the boundary of $\cv{\D}_{\bb_k}$, for every $k$,
    \item for any column $\v_i$ of $V$, there always exits $k\in\{1,\ldots,l\}$ such that $\v_i$ belongs to the boundary of $\cv{\D}_{\bb_k}$, that is
        \begin{equation*}
          \exists\,k\in J_k : \left(\cv{\L}_\aa^T\cdot\v_i\right)_h=-b_{hk}
        \end{equation*}
        where $\left(\cv{\L}_\aa^T\cdot\v_i\right)_h$ denotes the $h$-th entry of the $\cv{m}$-vector $\cv{\L}_\aa^T\cdot\v_i$.
  \end{enumerate}
  By contradiction, assume that there exists a column $\v_i$ of $V$ giving an interior lattice point of $\D_{\cv{\bb}}=\sum_k\cv{\D}_{\bb_k}$, that is, by the definition of $\D_{\cv{\bb}}$ given in item (E) of \ref{algoritmoDnef},
  \begin{equation*}
    \cv{\L}_\aa^T\cdot\v_i>-\cv{\bb}\ \Longrightarrow\ \forall\,k=1,\ldots,l\quad \cv{\L}_\aa^T\cdot\v_i>-\bb_k
  \end{equation*}
  contradicting the previous assertion (2). Recalling Lemma~\ref{lem:inclusione} and the consequent relation (\ref{parteintera}), one has
  \begin{equation*}
    \conv(V)=:\NN=\left[\cv{\D}_{\cv{\bb}}\right]\subseteq\left[\D_{\cv{\bb}}\right]
  \end{equation*}
  and, by convexity, this suffices to conclude that the only interior lattice point of $\D_{\cv{\bb}}$ is given by the origin $\0$. Therefore the previous assertion (1) implies that
  \begin{equation*}
    l^*\left(\sum_{j\in J} \D_{\bb_j}\right)=0
  \end{equation*}
  for any proper subset $J\subset\{1,\ldots,l\}$. Notice that, item (4) of Corollary~\ref{cor:hpO} allows us to already deduce that
  \begin{equation*}
    h^1(\cO_{\rdY})=\cdots=h^{n-l-1}(\cO_{\rdY})=0
  \end{equation*}
  Clearly, by item (3) in Corollary~\ref{cor:hpO}, $h^0(\cO_{\rdY})=1$. Moreover, the computation of $h^{n-l}(\cO_{\rdY})$ proceeds as for the computation of $h^{n-l}(\cO_Y)$ in the proof of Theorem~\ref{thm:h21Y}. The needed Koszul complex is the following
  \begin{eqnarray*}
        \widehat{\Ksz} :& \xymatrix{0\ar[r]&\cO_{\XX}\left(-\widehat{D}_{\bg}\right)\ar[r]&\cdots
        \bigoplus\limits_{\{k_1,\ldots,k_{l-p}\}\subseteq\{1,\ldots,l\}}
        \cO_{\XX}\left(-\sum\limits_{j=1}^{l-p}\widehat{D}_{\bg_{k_j}}\right)\cdots}&\\
        &\xymatrix{\ar[r]&\bigoplus\limits_{1\le j< k\le l}\cO_{\XX}\left(-\widehat{D}_{\bg_j}-\widehat{D}_{\bg_k}\right)\ar[r]&
        \bigoplus\limits_{k=1}^l\cO_{\XX}\left(-\widehat{D}_{\bg_k}\right)\ar[r]&\cO_{\XX}}&
      \end{eqnarray*}
      which is an acyclic resolution of the direct image sheaf $j_*\cO_{\rdY}$, being $j:\rdY\hookrightarrow\XX$ the given embedding.
      Recalling Corollary~\ref{cor:hpO}, the same argument proving Proposition~\ref{prop:nef&big} shows that $h^p(\cO_{\widehat{Y}^\vee})=h^p(\cO_{Y^\vee})=0$, for every $p\neq 0,n-l$, and (\ref{h0dY}) is obtained by constraints on the Euler characteristics involved in this Koszul resolution of $j_*\cO_Y$, just recalling that $l^*(\D_{\cv{\bb}})=1$.
      Notice that, until now, $\phi$ can also be assumed to be trivial, that is, $\XX=\cv{\XX}_\aa$, so getting the same results for $h^p(\cO_{Y^\vee})$: one has to be careful in applying the Batyrev-Borisov vanishing \ref{thm:BBvanish}, as $\cv{D}_{\bb_k}$ may not be a nef divisor. Anyway, this fact can be bypassed by recalling that $\widehat{D}_{\bg_k}$ can be assumed nef and relations (\ref{politopi}).

      Let us then pass to computing $h^p\left(\rO_{\rdY}\right)$.
      Assume $\XX$ to be $\Q$-factorial and consider the following exact sequence, obtained as a restriction to $\rdY$ of the standard Euler sequence and its generalization to Zariski 1-forms (see e.g. \cite[Thm.~8.1.6]{CLS})
      \begin{equation}\label{seqX}
  \xymatrix{0\ar[r]&\widehat{\Omega}_{\XX|\rdY}\ar[r]&\bigoplus_{i=1}^{\widehat{m}}\cO_{\rdY}(-\widehat{D}_i)
  \ar[r]&\Cl(\XX)\otimes\cO_{\rdY}\ar[r]&0}
\end{equation}
where $\widehat{m}:=\rk(\Weil(\XX))=n+\rk(\Cl(\XX))$\,.

\begin{lemma}\label{lem:cohDi}
  If $\widehat{D}_i$ is a torus invariant prime generator of $\Weil(\XX)$, $1\le i\le \widehat{m}$, and $\rk(\Cl(\XX))$ is big enough, then
\begin{equation*}
h^q\left(\cO_{\rdY}(-\widehat{D}_i)\right)=\left\{\begin{array}{cc}
                                                    0 & \text{for}\ q\neq n-l \\
                                                    1 &  \text{for}\ q=n-l
                                                  \end{array}
\right.
\end{equation*}
\end{lemma}
\noindent Let us postpone the proof of this Lemma to carry on the proof of Theorem~\ref{thm:h21Yd}. Then (\ref{seqX}) gives
\begin{equation}\label{hpOmegaX|Y}
  h^q\left(\rO_{\XX|\rdY}\right)=\left\{\begin{array}{cc}
                                          \rk(\Cl(\XX)) & \text{for}\ q=1 \\
                                          0 & \text{for}\ q\neq 1,n-l,n-l+1
                                        \end{array}
  \right.
\end{equation}
Consider the following exact sequence
\begin{equation}\label{seqYdual}
  \xymatrix{0\ar[r]&\bigoplus_{k=1}^l\cO_{\rdY}\left(-\cv{D}_{\bg_k}\right)\ar[r]&\widehat{\Omega}_{\XX|\rdY}
  \ar[r]&
\widehat{\Omega}_{\rdY}\ar[r]&0}
\end{equation}
obtained by the definition of $\rdY$ as a complete intersection in $\XX$. Let us now assume $n-l\ge q+1$, for some $q\ge 2$, and $\cv{\D}_{\bb_1},\ldots,\cv{\D}_{\bb_l}$ be $h$-independent, for some $h\ge q+2$. Then, for every $s$-element subset $$\{\cv{\D}_{\bb_{k_1}},\ldots,\cv{\D}_{\bb_{k_s}}\}\subseteq\{\cv{\D}_{\bb_1},\ldots,\cv{\D}_{\bb_l}\}$$ one has
\begin{equation}\label{dimensione-3ind}
  \dim\left(\sum_{j=1}^s\cv{\D}_{\bb_{i_j}}\right)\ge s+h-1 \ge s+q+1
\end{equation}
For any $i=1,\ldots,l$, consider the Koszul complex
\begin{eqnarray*}
        \Ksz_{\bg_i} :& \xymatrix{0\ar[r]&\cO_{\XX}\left(-\widehat{D}_{\bg
        _i}-\widehat{D}_{\bg}\right)\ar[r]&\cdots
        \bigoplus\limits_{\{k_1,\ldots,k_{l-p}\}\subseteq\{1,\ldots,l\}}
        \cO_{\XX}\left(-\widehat{D}_{\bg
        _i}-\sum\limits_{j=1}^{l-p}\widehat{D}_{\bg_{k_j}}\right)}&\\
        &\xymatrix{\cdots\ar[r]&
        \bigoplus\limits_{k=1}^l\cO_{\XX}\left(-\widehat{D}_{\bg
        _i}-\widehat{D}_{\bg_k}\right)\ar[r]&\cO_{\XX}\left(-\widehat{D}_{\bg
        _i}\right)}&
      \end{eqnarray*}
      which is an acyclic resolution of the direct image sheaf $j_*\cO_{\rdY}\left(-\cv{D}_{\bg_i}\right)$. The associated spectral sequences give, on the one hand,
      \begin{equation*}
        \H^{p+l}(\Ksz_{\bg_i})\cong\ 'E_2^{p,l}=H^p\left(\XX,j_*\cO_{\rdY}\left(-\cv{D}_{\bg_i}\right)\right)\cong H^p\left(\rdY,\cO_{\rdY}\left(-\cv{D}_{\bg_i}\right)\right)
      \end{equation*}
      and, on the other hand,
      \begin{equation*}
        \H^{l+p}\left(\Ksz_{\bg_i}\right)
        \cong\bigoplus_{a+b=l+p}\,''E_\infty^{a,b}
      \end{equation*}
      and, for $a+b=l+p$,
      \begin{equation*}
        ''E_1^{a,b}=\bigoplus_{\{k_1,\ldots,k_{l-a}\}\subseteq\{1,\ldots,l\}} H^b\left(\XX,\cO_{\XX}\left(-\widehat{D}_{\bg_i}-\sum_{j=1}^{l-a} \widehat{D}_{\bg_{k_j}}\right)\right)
      \end{equation*}
      But,
      \begin{equation}\label{E0}
        ''E_1^{l+p,0}=0
      \end{equation}
      (obviously for $p>0$ and, for $p=0$, recall that $''E_1^{l,0}= H^0\left(\XX,\cO_{\XX}\left(-\widehat{D}_{\bg_i}\right)\right)=0$). Moreover,  the Batyrev-Borisov vanishing~\ref{thm:BBvanish} gives
      \begin{equation}\label{El+t}
        ''E_1^{0,l+p}=H^{l+p}\left(\XX,\cO_{\XX}\left(-\widehat{D}_{\bg_i}-\widehat{D}_{\bg}\right)\right)=0 \quad\text{for}\ l+p\le n-1
      \end{equation}
      Finally
      \begin{equation}\label{Ep}
        ''E_1^{l+p-t,t}=\bigoplus_{\{k_1,\ldots,k_{t-p}\}\subseteq\{1,\ldots,l\}}
        H^t\left(\XX,\cO_{\XX}\left(-\widehat{D}_{\bg_i}-\sum_{j=1}^{t-p}\widehat{D}_{\bg_{k_j}}\right)\right)
        =0
      \end{equation}
      for $t\le t-p+h-2$, that is, $p\le h-2$, recalling (\ref{dimensione-3ind}). Being $q\le h-2$, this vanishing holds for $0\le p\le q$. Putting together (\ref{E0}),(\ref{El+t}) and (\ref{Ep}), one has
      \begin{equation*}
        \forall\,0\le t\le q\quad \bigoplus_{a+b=l+t}{''E_1^{a,b}}=0\,\Rightarrow\,\bigoplus_{a+b=l+t}{''E_\infty^{a,b}}\cong 0 \,\Rightarrow\, H^t\left(\cO_{\rdY}\left(-\widehat{D}_{\bg_i}\right)\right)=0
      \end{equation*}
      Then, (\ref{hpOmegaX|Y}) and (\ref{seqYdual}) give that
      \begin{eqnarray*}
        \forall\,q\ge 2\hskip3.36truecm h^0\left(\rO_{\rdY}\right) &=& h^0\left(\rO_{\XX|\rdY}\right)= 0\\
        h^1\left(\rO_{\rdY}\right) &=& h^1\left(\rO_{\XX|\rdY}\right)=\rk(\Cl(\XX))  \\
        \forall\,q\ge 3\,,\forall\,p: 2\le p\le q-1 \quad  h^p\left(\rO_{\rdY}\right) &=& h^p\left(\rO_{\XX|\rdY}\right)=0
      \end{eqnarray*}
      The complete statement is then obtained recalling the Danilov vanishing~\ref{thm:Danilov} and the exact sequence in Batyrev-Cox Theorem~\ref{thm:BC}.
      \end{proof}

      \subsubsection{Proof of Lemma~\ref{lem:cohDi}}

      For any $i=1,\ldots,\widehat{m}$, consider the Koszul complex
\begin{eqnarray*}
        \Ksz_{i} :& \xymatrix{0\ar[r]&\cO_{\XX}\left(-\widehat{D}_{     i}-\widehat{D}_{\bg}\right)\ar[r]&\cdots
        \bigoplus\limits_{\{k_1,\ldots,k_{l-p}\}\subseteq\{1,\ldots,l\}}
        \cO_{\XX}\left(-\widehat{D}_{i}-\sum\limits_{j=1}^{l-p}\widehat{D}_{\bg_{k_j}}\right)}&\\
        &\xymatrix{\cdots\ar[r]&
        \bigoplus\limits_{k=1}^l\cO_{\XX}\left(-\widehat{D}_{i}-\widehat{D}_{\bg_k}\right)\ar[r]&
        \cO_{\XX}\left(-\widehat{D}_{i}\right)}&
      \end{eqnarray*}
      which is an acyclic resolution of the direct image sheaf $j_*\cO_{\rdY}\left(-\widehat{D}_{i}\right)$. On the one hand, one  have
      \begin{equation*}
        \H^{p+l}(\Ksz_{i})\cong\ 'E_2^{p,l}=H^p\left(\XX,j_*\cO_{\rdY}\left(-\widehat{D}_{i}\right)\right)\cong H^p\left(\rdY,\cO_{\rdY}\left(-\widehat{D}_{i}\right)\right)
      \end{equation*}
      On the other hand,
      \begin{equation*}
        ''E_1^{p,q}=\bigoplus_{\{k_1,\ldots,k_{l-p}\}\subseteq\{1,\ldots,l\}} H^q\left(\XX,\cO_{\XX}\left(-\widehat{D}_{i}-\sum_{j=1}^{l-p} \widehat{D}_{\bg_{k_j}}\right)\right)
      \end{equation*}
      If $\rk(\Cl(\XX))$ is big enough then $\D_{\widehat{D}_i}=\{\0\}$ and Batyrev-Borisov vanishing~\ref{thm:BBvanish} gives
      \begin{equation*}
        ''E_1^{0,q}= H^q\left(\XX,\cO_{\XX}\left(-\widehat{D}_{i}- \widehat{D}_{\bg}\right)\right)\cong H^q\left(\XX,\cO_{\XX}\left(- \widehat{D}_{\bg}\right)\right)\cong \left\{\begin{array}{cc}
                                 0 & \text{for}\ q\neq n \\
                                 \C & \text{for}\ q=n
                               \end{array}
        \right.
      \end{equation*}
      recalling that $l^*(\D_{\cv{\bb}})=1$. Moreover,
      \begin{equation*}
        \forall\,p: 1\le p\le l-1\quad ''E_1^{p,q}\cong \bigoplus_{\{k_1,\ldots,k_{l-p}\}\subseteq\{1,\ldots,l\}} H^q\left(\XX,\cO_{\XX}\left(-\sum_{j=1}^{l-p} \widehat{D}_{\bg_{k_j}}\right)\right)=0
      \end{equation*}
      as $l^*\left(\sum_{j=1}^{l-p} \widehat{\D}_{\bg_{k_j}}\right)=0$. We claim that
      \begin{equation}\label{E1l}
      \forall\,q\quad  ''E_1^{l,q}= H^q\left(\XX,\cO_{\XX}\left(-\widehat{D}_{i}\right)\right)=0
      \end{equation}
      Therefore, $\,''E_1^{p,q}= 0$ for $(p,q)\neq(0,n)$, and
      \begin{equation*}
h^q\left(\cO_{\rdY}(-\widehat{D}_i)\right)=\dim\left(\H^{l+q}\left(\Ksz_{i}\right)\right)=\left\{\begin{array}{cc}
                                                    0 & \text{for}\ q\neq n-l \\
                                                    1 &  \text{for}\ q=n-l
                                                  \end{array}
\right.
\end{equation*}
It remains to prove (\ref{E1l}). At this purpose recall that, for every torus invariant divisor $\widehat{D}_\cc=\sum_{i=1}^{\widehat{m}}c_i\widehat{D}_i\in\Weil(\XX)$, there is a natural decomposition of the $\check{\mathrm C}$ech cohomology
\begin{equation}\label{Cech-decomp}
  H^q\left(\XX,\cO_{\XX}\left(\widehat{D}_\cc\right)\right)=\bigoplus_{\n\in N}H^q\left(\XX,\cO_{\XX}\left(\widehat{D}_\cc\right)\right)_\n
\end{equation}
(see e.g. \cite[\S9.1]{CLS}) where
\begin{equation}\label{singular-coh}
  H^q\left(\XX,\cO_{\XX}\left(\widehat{D}_\cc\right)\right)_\n\cong \widetilde{H}^{q-1}\left(V_{\widehat{D}_\cc,\n},\C\right)
\end{equation}
being $\widetilde{H}^\bullet$ the reduced singular cohomology (see \cite[Thm.~9.1.3]{CLS}). For the definition of $V_{\widehat{D}_\cc,\n}$, recall that we are assuming $\XX$ to be $\Q$-factorial. Calling $\Xi$ the fan of $\XX$, then $\Xi\in\SF\left(\widehat{\L}\right)$, being
\begin{equation*}
  \widehat{\L}=\left(
                 \begin{array}{ccc}
                   \widehat{\ll}_1 & \cdots & \widehat{\ll}_{\widehat{m}} \\
                 \end{array}
               \right)
\end{equation*}
the fan matrix of $\XX$, obtained by adding primitive generators of the rays associated with exceptional divisors, in the resolution $\phi:\XX\longrightarrow\cv{\XX}_\aa$, to the columns of the fan matrix $\cv{\L}_\aa$ of $\cv{\XX}_\aa$, as defined in item (C) of \ref{algoritmoDnef}. Set
\begin{eqnarray*}
  \mathcal{I}(\Xi)&:=&\{I\in\wp(\{1,\ldots,\widehat{m}\})\,|\,\langle \widehat{\L}_I\rangle\in\Xi(n)\}\\
  \n_I&:=& \left(\widehat{\L}_I^T\right)^{-1}\cdot\left((\cc^T)_I\right)^T\in N_\Q
\end{eqnarray*}
where $\cc$ is considered as a column vector.
Since $\Xi$ is complete and simplicial, $\{\n_I\,|\,I\in\I(\Xi)\}$ defines the support function $\vf_{\widehat{D}_\cc}$ of $\widehat{D}_\cc$ by
\begin{equation*}
  \forall\,I\in\I(\Xi)\,,\ \forall\,\m\in\langle\widehat{\L}_I\rangle\quad\vf_{\widehat{D}_\cc}(\m):=-\n_I^T\cdot \m
\end{equation*}
Therefore
\begin{equation*}
  V_{\widehat{D}_\cc,\n}:=\left\{\m\in|\Xi|\,|\,\n^T\cdot\m<\vf_{\widehat{D}_\cc}(\m)\right\}=
  \bigcup_{I\in\I(\Xi)}\left\{\m\in\langle\widehat{\L}_I\rangle\,|\,(\n_I+\n)^T\cdot\m<0\right\}
\end{equation*}
turns out to be contained in a finite union of open semi-spaces $\Int(H_{\n_I+\n}^-)$, for $I\in\I(\Si)$.
In our case, $\widehat{D}_\cc=-\widehat{D}_i$, that is, $\cc^T=(0,\ldots,0,\underset{i}{-1},0,\ldots,0)$ and $\n_I=\0_n$ for every $I\in\I(\Xi)$ such that $i\not\in I$.
This means that
\begin{equation*}
  V_{-\widehat{D}_i,\n}\subseteq \Int(H^-_{\n})\ \cup\bigcup_{I\in\I(\Si):i\in I}\Int(H^-_{\n_I+\n})
\end{equation*}
We want to show that $V_{-\widehat{D}_i,\n}$ is contractible, so that (\ref{E1l}) can be obtained from (\ref{Cech-decomp}) and (\ref{singular-coh}). In particular, we will prove that there exists a whole ray $\langle \widehat{\ll}_j\rangle$ not contained in $V_{-\widehat{D}_i,\n}$, and this suffices.

Start by assuming that
\begin{equation}\label{Hp}
  \bigcup_{I\in\I(\Xi):i\in I} \langle\widehat{\L}_I\rangle\subseteq H^-_\n
\end{equation}
By completeness, there should exist a generator $\widehat{\ll}_j$ of a cone $\langle\widehat{\L}_I\rangle$ such that $i\not\in I$ and $\n^T\cdot\widehat{\ll}_j >0$. Then, recalling that $\n_I=\0_n$, for every $I$ such that $i\not\in I$,
\begin{eqnarray*}
  \forall\,I\in\I(\Xi):j\in I\quad 0<\n^T\cdot\widehat{\ll}_j=(\n_I+\n)^T\cdot\widehat{\ll}_j\ &\Longrightarrow&\widehat{\ll}_j\not\in H^-_{\n_I+\n}\\
  &\Longrightarrow& \left\langle\widehat{\ll}_j\right\rangle\nsubseteq V_{-\widehat{D}_i,\n}
\end{eqnarray*}
Assume, now, the contrary of (\ref{Hp}). Then, there should exist a generator $\widehat{\ll}_j$ of a cone $\langle\widehat{\L}_I\rangle$, such that $i\in I$ and $\n^T \cdot \widehat{\ll}_j >0$. If $\rk(\Cl(\XX))$ is big enough then we can assume that there exist at least two such generators not belonging to $H^-_\n$, so that, $\widehat{\ll}_j$ can be considered distinct from $\widehat{\ll}_i$. Then
\begin{equation*}
  \n_I^T\cdot \widehat{\ll}_j=(\cc^T)_I\cdot\left(\widehat{\L}_I\right)^{-1}\cdot \widehat{\ll}_j= \d_{ij}=0
\end{equation*}
where $\d_{ij}$ is the Kronecker delta and $i\neq j$. Consequently,
\begin{eqnarray*}
  \forall\,I\in\I(\Xi):j\in I\quad(\n_I+\n)^T\cdot\widehat{\ll}_j=0+\n^T\cdot\widehat{\ll}_j >0
  &\Longrightarrow&\widehat{\ll}_j\not\in H^-_{\n_I+\n}\\
  &\Longrightarrow& \left\langle\widehat{\ll}_j\right\rangle\nsubseteq V_{-\widehat{D}_i,\n}
\end{eqnarray*}

\begin{remark}\label{rem:>>0}
  The hypothesis \emph{$\rk(\Cl(\XX))$ big enough} in Lemma~\ref{lem:cohDi}, and then in Theorem~\ref{thm:h21Yd}, can be more explicitly expressed by requiring that:
  \begin{itemize}
    \item $\dim{\D_i}=0$, for any $i=1,\ldots,\widehat{m}=n+\rk(\Cl(\XX))$,
    \item for any $\n\in N$ there exists at least two rays $\rho_1,\rho_2\in\Xi(1)$ such that $\rho_i\nsubseteq H^-_\n$, for both $i=1,2$.
  \end{itemize}
\end{remark}

\subsubsection{The hypersurface case} When $l=1$ one can say much more about Hodge numbers $h^p(\widehat{\Omega}_{\widehat{Y}^\vee})$, being $\widehat{Y}^\vee=\phi^{-1}(Y^\vee)$, for a suitable (partial) resolution $\phi:\XX\longrightarrow\XX_\aa$ and a general $Y^\vee\in |D'_\bb|$. First of all, notice that assumptions (\ref{a,ak}) for $l=1$ give necessarily
\begin{eqnarray*}
  \aa=\aa_0&:=&(\1_n,\d)\ ,\quad\text{with}\quad \d=d-n\\
  \bb=\bb_0&:=&(\d\1_n,1)
\end{eqnarray*}
Moreover, recall that $\XX_{\aa_0}\cong\P(\aa_0)/G_{\aa_0}$, being $G_{\aa_0}\cong\left(\Z/d\Z\right)^{n-1}$ \cite[Thm.~4.1, Lemma~5.2]{R-fTV}.

\begin{theorem}
  Assume $n\ge 3$. Let $\phi:\XX\longrightarrow\XX_{\aa_0}$ be a (possibly trivial or partial) resolution and $\rdY:=\phi^{-1}(Y^\vee)$ the transformed hypersurface defined in (\ref{Ytrasformata}). Then, $\rdY$ is an irreducible variety of dimension $n-1$ and
  \begin{equation*}
    h^p(\cO_{\rdY})=h^p(\cO_{Y^\vee})=\left\{\begin{array}{cc}
                        1 & \text{for $p=0, n-1$} \\
                        0 & \text{otherwise}
                      \end{array}
    \right.
  \end{equation*}
  \begin{equation*}
    h^p\left(\widehat{\Omega}_{Y^\vee}\right)=\left\{\begin{array}{cc}
                                    0=h^p\left(\widehat{\Omega}_{\XX_{\aa_0}}\right) & \text{for $p\neq 1, n-2,n-1,n$}\\
                                1=h^1\left(\widehat{\Omega}_{\XX_{\aa_0}}\right) & \text{for $p=1$ and $n\ge 4$}
                              \end{array}\right.
  \end{equation*}
  Moreover, calling $\D_i:=\D_{D'_i}$, for $i=1,\ldots,n+1$, being $D'_i$ a torus invariant prime generator of $\Weil(\XX_{\aa_0})$, one has
  \begin{itemize}
    \item for $n>3$,  $$h^{n-2}\left(\widehat{\Omega}_{Y^\vee}\right)-h^{n-1}\left(\widehat{\Omega}_{Y^\vee}\right)+
        h^{n}\left(\widehat{\Omega}_{Y^\vee}\right)=l^*(2\D_{\bb_0})-\sum_{i=1}^{n+1}l^*(\D_i+\D_{\bb_0})$$
    \item for $n=3$,
    $$h^{1}\left(\widehat{\Omega}_{Y^\vee}\right)-h^{2}\left(\widehat{\Omega}_{Y^\vee}\right)+
  h^{3}\left(\widehat{\Omega}_{Y^\vee}\right)=
  l^*(2\D_{\bb_0})-\sum_{i=1}^{4}l^*(\D_i+\D_{\bb_0})+1$$
  \end{itemize}
   with $0\le h^{n}\left(\widehat{\Omega}_{Y^\vee}\right)\le 1$.

\noindent Furthermore, assuming $r:=\rk(\Cl(\XX))\gg 1$ big enough then
  \begin{equation*}
    h^p\left(\widehat{\Omega}_{\rdY}\right)=\left\{\begin{array}{cc}
                                0=h^p\left(\widehat{\Omega}_{\XX}\right) & \text{for $p\neq 1, n-2,n-1,n$}\\
                                r=h^p\left(\widehat{\Omega}_{\XX}\right) & \text{for $p=1$ and $n\ge 4$}
                              \end{array}\right.
  \end{equation*}
and
\begin{itemize}
  \item for $n>3$,
  $$ h^{n-2}\left(\widehat{\Omega}_{\rdY}\right)-h^{n-1}\left(\widehat{\Omega}_{\rdY}\right)+ h^n\left(\widehat{\Omega}_{\rdY}\right)=
  l^*(2\D_{\bb_0})-n+r-2$$
  \item for $n=3$,
  $$ h^{1}\left(\widehat{\Omega}_{\rdY}\right)-h^{2}\left(\widehat{\Omega}_{\rdY}\right)+
  h^{3}\left(\widehat{\Omega}_{\rdY}\right)=
  l^*(2\D_{\bb_0})+2r-5$$
\end{itemize}
with $0\le h^n\left(\widehat{\Omega}_{\rdY}\right)\le r$.

\noindent Finally, if $\rdY$ is smooth and $\XX$ is $\Q$-factorial, then
$h^{n-1}\left(\widehat{\Omega}_{\rdY}\right)=h^{n}\left(\widehat{\Omega}_{\rdY}\right)=0$ and
\begin{equation}\label{h^n-2}
  h^{n-2}\left(\Omega_{\rdY}\right)=\left\{\begin{array}{cc}
                                           l^*(2\D_{\bb_0})+2r-5 & \text{for $n=3$} \\
                                           l^*(2\D_{\bb_0})-n+r-2 & \text{for $n>3$}
                                         \end{array}
  \right.
\end{equation}
\end{theorem}

\begin{proof}
Computing $h^p(\cO_{\rdY})=h^p(\cO_{Y^\vee})$ goes exactly as in Theorem~\ref{thm:h21Yd}, observing that, for $l=1$, $\D_{\cv{\bb}}=\D_\bb$\,.

\noindent Recalling (\ref{seqYdual}) and (\ref{seqX}), setting $\widehat{D}_{\bg_0}=\phi^{-1}(D'_{\bb_0})$ and being $\widehat{D}_i$ a prime torus invariant divisor generating $\Weil(\XX)$, the following exact sequences
\begin{eqnarray}\label{seq-fasci}
  &\xymatrix{0\ar[r]&\cO_{\rdY}\left(-\widehat{D}_{\bg_0}\right)\ar[r]^-\a&\widehat{\Omega}_{\XX|\rdY}
  \ar[r]^-\b&
\widehat{\Omega}_{\rdY}\ar[r]&0}& \\
\nonumber
   &\xymatrix{0\ar[r]&\widehat{\Omega}_{\XX|\rdY}\ar[r]^-\xi&\bigoplus_{i=1}^{\widehat{m}}\cO_{\rdY}
   (-\widehat{D}_i)
  \ar[r]^-\eta&\Cl(\XX)\otimes\cO_{\rdY}\ar[r]&0}&
\end{eqnarray}
give rise to the complex
\begin{equation*}
  \mathcal{Q}^\bullet:\xymatrix{0\ar[r]&\cO_{\rdY}\left(-\widehat{D}_{\bg_0}\right)\ar[r]^-{\xi\circ\a}&\bigoplus_{i=1}^{\widehat{m}}\cO_{\rdY}(-\widehat{D}_i)
  \ar[r]^-\eta&\Cl(\XX)\otimes\cO_{\rdY}\ar[r]&0}
\end{equation*}
whose sheaf cohomology is given by
\begin{equation}\label{hypcoh}
        \mathcal{H}^q(\mathcal{Q}^\bullet)\cong\left\{\begin{array}{cc}
                                         \widehat{\Omega}_{\rdY} & \text{for}\ q=1 \\
                                         0 & \text{otherwise}
                                       \end{array}
        \right.
      \end{equation}
The first associated spectral sequence $'E$, abutting to the hypercohomology $\H^*(\mathcal{Q}^\bullet)$ and degenerating at $2^{\text{nd}}$ level, gives
\begin{equation}\label{HpQ}
        \H^{p+1}(\mathcal{Q}^\bullet)\cong\ 'E_2^{p,1}=H^p\left(\rdY,\widehat{\Omega}_{\rdY}\right)
      \end{equation}
On the other hand, the second associated spectral sequence $''E$ gives
\begin{equation*}
  ''E_1^{p,q}=\left\{\begin{array}{cc}
                       0 & \text{for $p\neq 0,1,2$} \\
                       H^q\left(\rdY,\cO_{\rdY}(-\widehat{D}_{\bg_0})\right) & \text{for $p=0$} \\
                       \bigoplus_{i=1}^{\widehat{m}}H^q\left(\rdY,\cO_{\rdY}(-\widehat{D}_i)\right) & \text{for $p=1$} \\
                       H^q\left(\rdY,\cO_{\rdY}\right)^{\oplus r} & \text{for $p=2$}
                     \end{array}
  \right.
\end{equation*}
with $r:=\rk(\Cl(\XX))$.

\noindent Assume, first of all, $\phi$ is trivial, that is $\XX=\XX_{\aa_0}$, $\widehat{D}_{\bg_0}=D'_{\bb_0}$ and $\widehat{D}_i=D'_i$. Then $r=\rk(\Cl(\XX))=\rk(\Cl(\XX_{\aa_0}))=1$ and $\widehat{m}=m=n+1$. Moreover, $\XX_{\aa_0}$ turns out to be a finite quotient of the weighted projective space $\P(\aa_0)$ \cite[Lem.~5.2]{R-fTV}, meaning that suitable multiples of both $D'_{\bb_0}$ and $D'_i$ are ample divisors. Then Proposition~\ref{prop:nef&big} applies to give $''E_1^{p,q}=0$ excepts for the following cases
\begin{eqnarray*}
  \dim\left(\,''E_1^{0,n-1}\right) &=&  h^{n-1}\left(Y^\vee,\cO_{Y^\vee}(-D'_{\bb_0})\right)=l^*(2\D_{\bb_0})-1 \\
  \dim\left(\,''E_1^{1,n-1}\right) &=& \sum_{i=1}^{n+1}h^{n-1}\left(Y^\vee,\cO_{Y^\vee}(-D'_i)\right)=\sum_{i=1}^{n+1}l^*(\D_i+\D_{\bb_0}) \\
  \dim\left(\,''E_1^{2,0}\right) &=& h^0\left(Y^\vee,\cO_{Y^\vee}\right)=1\\
  \dim\left(\,''E_1^{2,n-1}\right) &=& h^{n-1}\left(Y^\vee,\cO_{Y^\vee}\right)=1
\end{eqnarray*}
Since $d_2:\,''E_2^{p,q}\longrightarrow\,''E_2^{p+2,q-1}$ is always zero, this suffices to show that the spectral sequence $''E$ degenerates at the second level, so giving, by (\ref{HpQ}), that
\begin{equation}\label{HpOmegaYd}
  H^p\left(Y^\vee,\widehat{\Omega}_{Y^\vee}\right)\cong\left\{\begin{array}{cc}
                                                       ''E_2^{2,n-1}  & \text{for $p=n$} \\
                                                       ''E_2^{1,n-1}  & \text{for $p=n-1$} \\
                                                     ''E_2^{0,n-1}\oplus\,''E_2^{2,n-3} & \text{for $p=n-2$} \\
                                                     ''E_2^{0,2}\oplus\,''E_2^{2,0} & \text{for $p=1$} \\
                                                     0 & \text{otherwise}
                                                   \end{array}
  \right.
\end{equation}
Assume $n>3$. Then $''E_2^{2,n-3}\cong\{0\}\cong\,''E_2^{0,2}$ and
\begin{eqnarray*}
   ''E_2^{2,n-1} &\cong&{H^{n-1}\left(Y^\vee,\cO_{Y^\vee}\right)\over\im\left(\bigoplus_{i=1}^{{n+1}}
   H^{n-1}\left(Y^\vee,\cO_{Y^\vee}(-D'_i)\right)\stackrel{d_1}
  {\longrightarrow}H^{n-1}\left(Y^\vee,\cO_{Y^\vee}\right)\right)}\\
  ''E_2^{1,n-1} &\cong& {\ker\left(\bigoplus_{i=1}^{{n+1}}
   H^{n-1}\left(Y^\vee,\cO_{Y^\vee}(-D'_i)\right)\stackrel{d_1}
  {\longrightarrow}H^{n-1}\left(Y^\vee,\cO_{Y^\vee}\right)\right)
  \over\im\left(H^{n-1}\left(Y^\vee,\cO_{Y^\vee}(-D'_{\bb_0})\right)\stackrel{d_1}
  {\longrightarrow}\bigoplus_{i=1}^{{n+1}}H^{n-1}\left(Y^\vee,\cO_{Y^\vee}(-D'_i)\right)\right)} \\
  ''E_2^{0,n-1} &\cong& \ker\left(H^{n-1}\left(Y^\vee,\cO_{Y^\vee}(-D'_{\bb_0})\right)\stackrel{d_1}
  {\longrightarrow}\bigoplus_{i=1}^{{n+1}}H^{n-1}\left(Y^\vee,\cO_{Y^\vee}(-D'_i)\right)\right) \\
  ''E_2^{2,0} &\cong& \C \quad\Longrightarrow\ h^1\left(\widehat{\Omega}_{Y^\vee}\right)=1
\end{eqnarray*}
Recalling  exact sequences (\ref{seq-fasci}) and the definition of $\mathcal{Q}^\bullet$, the map
\begin{equation*}
  \xymatrix{d_1:\bigoplus_{i=1}^{{n+1}}
   H^{n-1}\left(Y^\vee,\cO_{Y^\vee}(-D'_i)\right)\ar[r]&H^{n-1}\left(Y^\vee,\cO_{Y^\vee}\right)}
\end{equation*}
is given by $H^{n-1}(\eta)$, which is surjective if and only if $h^n(\widehat{\Omega}_{\XX_{\aa_0}|Y^\vee})=0$\,: notice that, alternatively, it can only happen that $h^n(\widehat{\Omega}_{\XX_{\aa_0}|Y^\vee})=1$ and $H^{n-1}(\eta)=0$. Moreover, the map $$d_1:H^{n-1}\left(Y^\vee,\cO_{Y^\vee}(-D'_{\bb_0})\right)
\longrightarrow\bigoplus_{i=1}^{{n+1}}H^{n-1}\left(Y^\vee,\cO_{Y^\vee}(-D'_i)\right)$$
is obtained as the composition $H^{n-1}(\xi)\circ H^{n-1}(\a)$. Clearly $H^{n-1}(\xi)$ is injective, as  $h^{n-2}(\cO_{Y^\vee})=0$. Then, recalling that $h^n\left(\widehat{\Omega}_{\XX_{\aa_0}|Y^\vee}\right)=h^n\left(\widehat{\Omega}_{Y^\vee}\right)$ by the first exact sequence in (\ref{seq-fasci}),
\begin{equation*}
  h^{n-2}\left(\widehat{\Omega}_{Y^\vee}\right)-h^{n-1}\left(\widehat{\Omega}_{Y^\vee}\right)+ h^n\left(\widehat{\Omega}_{Y^\vee}\right)=
  l^*(2\D_{\bb_0})-\sum_{i=1}^{n+1}l^*(\D_i+\D_{\bb_0})
\end{equation*}
where $h^n\left(\widehat{\Omega}_{Y^\vee}\right)\in\{0,1\}$.

\noindent Assume now $n=3$. Then
\begin{eqnarray*}
  ''E_2^{2,n-3} &=& ''E_2^{2,0}\,\cong\ \C \\
  ''E_2^{0,2} &=& ''E_2^{0,n-1}\cong\ker\left(H^{n-1}(\xi)\circ H^{n-1}(\a)\right) = \ker\left(H^{n-1}(\a)\right)
\end{eqnarray*}
and
\begin{equation*}
  h^{1}\left(\widehat{\Omega}_{Y^\vee}\right)-h^{2}\left(\widehat{\Omega}_{Y^\vee}\right)+
  h^{3}\left(\widehat{\Omega}_{Y^\vee}\right)=
  l^*(2\D_{\bb_0})-\sum_{i=1}^{4}l^*(\D_i+\D_{\bb_0})+1
\end{equation*}
where $h^3\left(\widehat{\Omega}_{Y^\vee}\right)\in\{0,1\}$.

Assume, now, $\phi:\XX\longrightarrow\XX_{\aa_0}$ is non-trivial and $\rk(\Cl(\XX))$ is sufficiently big such that $\D_i=\{\0\}$, for any $i=1,\ldots,\widehat{m}$. By construction, $\D_{\bg_0}=\D_{\bb_0}$ and Corollary~\ref{cor:nef} ensures that a suitable multiple of $\widehat{D}_{\bg_0}$ can be assumed ample, meaning that Proposition~\ref{prop:nef&big} can still be applied to computing $H^q\left(\rdY,\cO_{\rdY}(-\widehat{D}_{\bg_0})\right)$. This fact no more holds for $\widehat{D}_i$, but possibly assuming $r:=\rk(\Cl(\XX))\gg 1$ to be even bigger, Lemma~\ref{lem:cohDi} applies to computing $H^q\left(\rdY,\cO_{\rdY}(-\widehat{D}_i)\right)$.

\noindent Assume $n>3$. Then $''E_2^{2,n-3}\cong\{0\}\cong\,''E_2^{0,2}$ and (\ref{HpOmegaYd}) gives
\begin{eqnarray*}
   ''E_2^{2,n-1} &\cong&{H^{n-1}\left(\rdY,\cO_{\rdY}\right)^{\oplus r}\over\im\left(\bigoplus_{i=1}^{{n+1}}
   H^{n-1}\left(\rdY,\cO_{\rdY}(-\widehat{D}_i)\right)\stackrel{d_1}
  {\longrightarrow}H^{n-1}\left(\rdY,\cO_{\rdY}\right)^{\oplus r}\right)}\\
  ''E_2^{1,n-1} &\cong& {\ker\left(\bigoplus_{i=1}^{{n+1}}
   H^{n-1}\left(\rdY,\cO_{\rdY}(-\widehat{D}_i)\right)\stackrel{d_1}
  {\longrightarrow}H^{n-1}\left(\rdY,\cO_{\rdY}\right)^{\oplus r}\right)
  \over\im\left(H^{n-1}\left(\rdY,\cO_{\rdY}(-\widehat{D}_{\bg_0})\right)\stackrel{d_1}
  {\longrightarrow}\bigoplus_{i=1}^{{n+1}}H^{n-1}\left(\rdY,\cO_{\rdY}(-\widehat{D}_i)\right)\right)} \\
  ''E_2^{0,n-1} &\cong& \ker\left(H^{n-1}\left(\rdY,\cO_{\rdY}(-\widehat{D}_{\bg_0})\right)\stackrel{d_1}
  {\longrightarrow}\bigoplus_{i=1}^{{n+1}}H^{n-1}\left(\rdY,\cO_{\rdY}(-\widehat{D}_i)\right)\right) \\
  ''E_2^{2,0} &\cong& \C^r \quad\Longrightarrow\ h^1\left(\widehat{\Omega}_{\rdY}\right)=r
\end{eqnarray*}
\begin{equation*}
  h^{n-2}\left(\widehat{\Omega}_{\rdY}\right)-h^{n-1}\left(\widehat{\Omega}_{\rdY}\right)+ h^n\left(\widehat{\Omega}_{\rdY}\right)=
  l^*(2\D_{\bb_0})-n+r-2
\end{equation*}
where $0\le h^n\left(\widehat{\Omega}_{\rdY}\right)\le r$.

\noindent Assume now $n=3$. Then
\begin{eqnarray*}
  ''E_2^{2,n-3} &=& ''E_2^{2,0}\,\cong\ \C^r \\
  ''E_2^{0,2} &=& ''E_2^{0,n-1}\cong\ker\left(H^{n-1}(\xi)\circ H^{n-1}(\a)\right) = \ker\left(H^{n-1}(\a)\right)
\end{eqnarray*}
and
\begin{equation*}
  h^{1}\left(\widehat{\Omega}_{\rdY}\right)-h^{2}\left(\widehat{\Omega}_{\rdY}\right)+
  h^{3}\left(\widehat{\Omega}_{\rdY}\right)=
  l^*(2\D_{\bb_0})+2r-5
\end{equation*}
where $0\le h^3\left(\widehat{\Omega}_{\rdY}\right)\le r$.

Finally, assume that $\rdY$ is smooth and $\XX$ is $\Q$-factorial. Then, the weak Lefschetz theorem and the Danilov vanishing \ref{thm:Danilov} give
\begin{equation*}
  h^{n-1}\left(\Omega_{\rdY}\right)=0=h^n\left(\Omega_{\rdY}\right)
\end{equation*}
and (\ref{h^n-2}) follows immediately by previous relations.
\end{proof}

\subsection{The topological mirror test}\label{ssez:topMT}
 Consider a partitioned ftv $(X,\aa=\Sigma_{k=1}^l\aa_k)$ admitting a calibrated partitioned $f$-process
 \begin{equation*}
   \left(X,\aa=\sum_{k=1}^l\aa_k\right)\ {\leftrightsquigarrow}\ \left(\cv{\XX}_\aa,\cv{\bb}=\sum_{k=1}^l\bb_k\right)
 \end{equation*}
the family of complete intersection subvarieties
$$\mathcal{Y}_\aa:=\left\{Y:=\bigcap_{k=1}^lY_k\subset X\,|\, Y_k\in|D_{\aa_k}|\right\}$$
and its $f$-dual mirror family
\begin{equation*}
  \mathcal{Y}^\vee_{\cv{\bb}}:=\left\{Y^\vee:=\bigcap_{k=1}^lY^\vee_k\subset \cv{\XX}_{\aa}\,|\, Y_k\in|D'_{\bb_k}|\right\}
\end{equation*}
Following notation introduced in \cite[\S3.2]{R-fTV}, making the \emph{topological mirror test} for $\cY_\aa$ and $\cY^\vee_{\cv{\bb}}$ means performing the following double check for the ordered couple $(Y,Y^\vee)$:
\begin{itemize}
  \item \emph{$A$-side}: there exists a (partial) resolution of singularities $\widehat{Y}\longrightarrow Y$ such that $\widehat{Y}$ is (quasi-)smooth, for $Y$ generic in $\cY_\aa$, and
      \begin{equation*}
        k_{\widehat{Y}}= m_{Y^\vee}
      \end{equation*}
      where $k_{\widehat{Y}}$ is the number of \ka moduli of the resolved family $\widehat{\cY}_\aa$ and $m_{Y^\vee}$ is the number of complex moduli of the family $\cY^\vee_{\cv{\bb}}$; in this case,  $Y^\vee$ is called \emph{an $A$-mirror of $Y$};
  \item \emph{$B$-side}: there exists a (partial) resolutions of singularities $\widehat{Y}^\vee\longrightarrow Y^\vee$ such that $\widehat{Y}^\vee$ is (quasi-)smooth, for $Y^\vee$ generic in $\cY^\vee_{\cv{\bb}}$, and
      \begin{equation*}
        k_{\widehat{Y}^\vee}=m_Y
      \end{equation*}
       where $k_{\widehat{Y}^\vee}$ is the number of \ka moduli of the resolved family $\widehat{\cY}^\vee_{\cv{\bb}}$ and $m_{Y}$ is the number of complex moduli of the family $\cY_\aa$; in this case,  $Y$ is called \emph{a $B$-mirror of $Y^\vee$}.
\end{itemize}
In general $k_{\widehat{Y}}, m_{Y^\vee},k_{\widehat{Y}^\vee},m_Y$ are not well defined, but their meaning will turn out to be clear in all examples we are going to consider. Namely, for the $A$-side mirror check we will prove that
\begin{equation*}
  k_Y=h^1(\Omega_Y)=1=m_{Y^\vee}
\end{equation*}
which is enough because $Y$ is generically smooth and $m_{Y^\vee}$ does not depend on the resolution level. But the $B$-side mirror check will turn to be many more intricate. For hypersurfaces in $\P^n$ of degree $d\ge n+1$ we will check the existence of suitable (non canonical and partial) resolution $\widehat{Y}\longrightarrow Y$ such that
\begin{equation*}
  k_{\widehat{Y}^\vee}=h^1(\widehat{\Omega}_{\widehat{Y}^\vee})=m_Y
\end{equation*}
But, for the complete intersection $Y_{3,4}\subset\P^5$ we will consider in \S~\ref{sez:l ge 2}, we will only able to show the existence of a resolution $\cv{Y}^\vee\longrightarrow Y^\vee$ induced by a resolution of singularities of the ambient toric variety $\cv{\phi}:\widehat{\XX}_{\aa_0}\longrightarrow\cv{\XX}_\aa$, such that
\begin{equation*}
  k_{\cv{Y}}=h^1(\widehat{\Omega}_{\cv{Y}})> m_{Y^\vee}
\end{equation*}
See considerations ending up \S~\ref{sssez:moduli}.

\section{Case $l=1$: mirror test for projective hypersurfaces}\label{sez:l=1}
Assume
\begin{equation*}
  X=\P^n\quad \text{and}\quad \aa=\aa_0:=(\1_n, \d)
\end{equation*}
with $l=1$ and $\d=d-n\ge 1$. In \cite[Thm.~4.1, Rem.~4.4, Lem.~5.2]{R-fTV} a complete description of the $f$-dual ftv $(\XX_{\aa_0},\bb_0)$ is given, as a suitable framed quotient of the weighted projective space whose weights are assigned by $\aa_0$, namely:
\begin{equation*}
  \XX_{\aa_0}=\XX_{(\1_n,\d)} \cong \P(\1_n,\d)/G_{(\1_n,\d)}
\end{equation*}
by the action of $G_{(\1_n,\d)}\cong\Z^{n-1}_\d$ described in \cite[Lem.~5.2]{R-fTV}, and
\begin{equation*}
  \bb_0 =  (\d\cdot\1_n,1)
\end{equation*}
Calling $Y\in|D_{(\1_n,\d)}|$ a generic, hence smooth, hypersurface of degree $d$, its $f$-mirror partner is the hypersurface $Y^\vee\in|D'_{(\d\cdot\1_n,1)}|$ of $\XX_{(\1_n,\d)}$ defined as the zero locus of the weighted homogeneous polynomial
\begin{equation}\label{f}
  f^\vee=\prod_{i=1}^n x_i^{\d-1}\left(\sum_{i=1}^n x_i^d +\psi\,\prod_{j=1}^{n+1} x_j\right)+x_{n+1}^{n+1}\in\Cox\left(\XX_{(\1_n,\d)}\right)\cong\Cox\left(\P(\1_n,\d)\right)
\end{equation}
where $\psi$ is the unique complex modulus of the $f$-mirror family (see \cite[Thm.~5.3]{R-fTV}). Then, one has
\begin{equation}\label{Aside}
  k_Y=h^1(\Omega_Y)=1=m_{Y^\vee}
\end{equation}
by the weak Lefschetz Theorem and \cite[Thm.~5.3, Rem.~5.4]{R-fTV}. That is:
\begin{itemize}
  \item \emph{$Y^\vee$ is an $A$-mirror of $Y$.}
\end{itemize}
For what concerning the $B$-side checking, recall that
\begin{equation}\label{mY}
  m_Y=\dim \P\left(H^0(\cO_{\P^n}(d)\right) - \dim \P\GL(n+1) ={n+d\choose d}-(n+1)^2
\end{equation}
Then, one has to exhibit a suitable resolution $\widehat{Y}^\vee\longrightarrow Y^\vee$ to computing
 $k_{\widehat{Y}^\vee}$ and making the required comparison.

\subsubsection{Stringy Hodge numbers of $\Q$-factorial complete toric varieties}
Let us now recall the definition of \emph{stringy $E$-function} and \emph{stringy Hodge numbers} in the particular case of a $\Q$-factorial and complete toric variety $X(\Si)$.

Being $X$ $\Q$-factorial, for every maximal cone $\s\in \Si(n)$ there exists a unique subset $I\subset\{1,\ldots,|\Si(1)|\}$ such that $|I|=n$ and $\s=\langle V_I\rangle$, where $V$ is a fan matrix of $X$. Then, since $X$ is complete, there exists a well defined continuous function $\vf_K:N_\R\longrightarrow\R$, called the \emph{canonical support function}, constructed by setting
\begin{equation}\label{KsuppFunct}
  \forall\,\n\in N_\R ,\ \forall\,I:\n\in\langle V_I\rangle\quad\vf_K(\n):=-\m_I^T\cdot\n\quad\text{where}\quad \m_I:=-(V_I^T)^{-1}\cdot\1_n
\end{equation}
In particular, $\vf_K$ satisfies the following properties:
\begin{enumerate}
  \item $\vf_K(\v_i)=1$ for every column $\v_i$ of the fan matrix $V$ of $X$\,,
  \item $\vf_K$ is linear on each cone $\s\in\Si$\,.
\end{enumerate}

\begin{definition}[Stringy $E$-function]\label{def:stringy}
  The \emph{stringy $E$-function} of a complete and $\Q$-factorial toric variety $X(\Si)$ is the following
  \begin{equation}\label{E-funzione}
    \Est(X;u,v):=(uv-1)^{\dim X}\sum_{\s\in\Si}\left(\sum_{\n\in\Int(\s)\cap N}(uv)^{-\vf_K(\n)}\right)
  \end{equation}
  where $\Int(\s)$ denotes the relative interior of $\s$. The \emph{stringy Euler number of $X$} is then defined by
  \begin{equation*}
    e_{\text{st}}(X):=\lim_{u,v\to 1} \Est(X;u,v)
  \end{equation*}
\end{definition}

\begin{remark}\label{rem:gorenstein}
  Batyrev introduced the stringy $E$-function and Euler number for a normal, irreducible, algebraic variety $X$ with at worst log-terminal singularities \cite[Def.~3.1, Def.~3.3]{Batyrev98} and proved that $\Est(X;u,v)$ is given by (\ref{E-funzione}) when $X$ is a $\Q$-Gorenstein toric variety \cite[Thm.~4.3]{Batyrev98}. Moreover,
  \begin{itemize}
    \item[$(*)$] \emph{if $X$ has Gorenstein singularities then $\Est(X;u,v)$ turns out to be a polynomial} \cite[Prop.~4.4]{Batyrev98}.
  \end{itemize}
\end{remark}

\begin{definition}[Stringy Hodge numbers]
  Assume that $\Est(X;u,v)$ is a polynomial with
  \begin{equation*}
    \Est(X;u,v)=\sum_{p,q}a_{p,q}u^pv^q
  \end{equation*}
  Up to a sign, its coefficients are defined to be the \emph{stringy Hodge numbers} of $X$, namely
  \begin{equation*}
    \hst^{p,q}(X):=(-1)^{p+q}a_{p,q}
  \end{equation*}
\end{definition}

\begin{remark}\label{rem:Poincare}
  By Poincar\'{e} duality \cite[Thm.~3.7, Rem.~3.9]{Batyrev98}, if $\Est(X;u,v)$ is a polynomial then
  \begin{equation*}
    \deg \Est(X;u,v)=2\dim X
  \end{equation*}
  In particular, it turns out that
  \begin{equation*}
    \forall\,p,q\quad\hst^{p,q}(X)=\hst^{q,p}(X)\ ,\quad\hst^{0,0}(X)=1=\hst^{\dim X,\dim X}(X)
  \end{equation*}
\end{remark}

Consider the case of a smooth and complete toric variety $X$. The \emph{$E$-polynomial} of $X$ is defined as:
\begin{equation*}
  E(X;u,v):=\sum_{p,q}(-1)^{p+q}h^{p,q}(X)u^pv^q
\end{equation*}
\begin{proposition}[Cor.~3.6 in \cite{Batyrev98}]\label{prop:smooth-E}
Let $X$ be a smooth and complete toric variety. Then
$\Est(X;u,v)=E(X;u,v)$. In particular,
\begin{equation*}
\forall\,p,q\quad \hst^{p,q}(X)=h^{p,q}(X)
\end{equation*}
\end{proposition}

\begin{definition}\label{def:crepant}
  Let $X$ be a projective algebraic variety with at worst canonical Gorenstein singularities. A birational morphism $\phi:Y \longrightarrow X$ is called a \emph{crepant partial resolution} of $X$ if $\phi^*K_X=K_Y$. If $Y$ is smooth than $\phi$ is called a \emph{crepant resolution} of $X$.
\end{definition}

\begin{proposition}[Thm.~3.12 in \cite{Batyrev98}]\label{prop:E-crepant}
  Let $\phi:Y\longrightarrow X$ be a crepant partial re\-so\-lu\-tion. Then $\Est(Y;u,v)=\Est(X;u,v)$. In particular, Proposition \ref{prop:smooth-E} ensures that, if $\phi$ is a crepant resolution then
  \begin{equation*}
    E(Y;u,v)=\Est(X;u,v)\ \Longrightarrow\ \forall\,p,q\quad h^{p,q}(Y)=\hst^{p,q}(X)
  \end{equation*}
\end{proposition}

\subsubsection{Stringy Hodge numbers of $\XX_\aa$}
Consider, now, the $\Q$-factorial and complete toric variety $\XX_{\aa}$, associated with the fan $\Si_\aa=\Si_{\D(\P^n,\aa)}$ over the polytope $\D(\P^n,\aa)$. Assume that
\begin{equation}\label{convenzione}
  a_1\leq a_2\leq \cdots \leq a_{n+1}\quad\text{and}\quad \gcd(a_1,\ldots,a_{n+1})=1
\end{equation}
and conditions (a) and (b) in \cite[Thm.~4.1]{R-fTV} are satisfied: this is the case, e.g., when $\aa=\aa_0$.

\begin{proposition}\label{prop:n+1}
  Let $\L_{\aa}=\left(
                                      \begin{array}{ccc}
                                        \ll_1 & \cdots & \ll_{n+1} \\
                                      \end{array}
                                    \right)
  $ be the fan matrix of $\XX_{\aa}$ obtained by the polytope $\D_\aa$ described in (\ref{DeltaaConv}). Then the following are equivalent:
  \begin{enumerate}
    \item $|\aa|=n+1$,
    \item $\D_\aa$ is a reflexive polytope,
    \item $\XX_\aa$ has Gorenstein singularities.
  \end{enumerate}
  In particular, by setting
  \begin{eqnarray*}
    \psi_\aa(0)&:=&1\\
    \forall\,h\in\N\setminus\{0\}\quad\psi_\aa(h)&:=&\left|(h\D_{\aa}\setminus(h-1)\D_\aa)\cap M\right|=l(h\D_\aa)-l((h-1)\D_\aa)
  \end{eqnarray*}
  if one of facts from (1) to (3) holds then
    \begin{equation}\label{E-funzione a}
    \Est(\XX_{\aa};u,v)= (uv-1)^{n}\sum_{h\geq 0}\psi_\aa(h)(uv)^{-h}
  \end{equation}
 is a polynomial of degree $2n$\,.
\end{proposition}

\begin{proof}
(1)\ $\Rightarrow$\ (2): clearly $|\aa|=n+1$ if and only if $\aa=\1_{n+1}$. Then (\ref{DeltaaConv}) shows that $\D_\aa$ is a lattice polytope admitting $\0$ as the unique interior lattice point.

(2)\ $\Rightarrow$\ (3):
  To checking that $\XX_{\aa}$ has only Gorenstein singularities is the same as checking that
  \begin{equation}\label{gorenstein}
    \forall\,\m\in M\quad \vf_K(\m)\in\Z
  \end{equation}
  Given $\m\in M$ let
  \begin{equation}\label{h}
    h(\m):=\min\{l\in\N\,|\,\m\in l\D_\aa\}
  \end{equation}
Being $\D_\aa$ a reflexive polytope, $\m$ has to belong to the boundary of $h\D_\aa$, otherwise $\m\in(h-1)\D_\aa$, against the definition of $h$. Therefore, $\vf_K(\m)=-h\in\Z$\,, so giving (\ref{gorenstein}).

  (3)\ $\Rightarrow$\ (1): by contradiction, assume $|\aa|>n+1$; then
  $$\aa=\left(
                                                                         \begin{array}{c}
                                                                           a_1 \\
                                                                           \vdots \\
                                                                           a_n \\
                                                                           a_{n+1} \\
                                                                         \end{array}
                                                                       \right)\geq \left(
                                                                                     \begin{array}{c}
                                                                                       1 \\
                                                                                       \vdots \\
                                                                                       1 \\
                                                                                       2 \\
                                                                                     \end{array}
                                                                                   \right)=:\aa'
  $$
  By (\ref{DeltaaConv}), this means that
   \begin{equation*}
    \D_\aa\supset \D_{\aa'}=\conv\left(
                          \begin{array}{c}
                            n+1 \\
                            -1 \\
                            -1\\
                            \vdots \\
                            -1 \\
                          \end{array} \,
                                        \begin{array}{c}
                                          -1 \\
                                           n+1\\
                                          -1\\
                                          \vdots \\
                                          -1 \\
                                        \end{array}
                                     \,\cdots\,
                                        \begin{array}{c}
                                          -1 \\
                                          -1\\
                                          \vdots \\
                                          -1 \\
                                          n+1\\
                                        \end{array}
                                      \,
                                        \begin{array}{c}
                                          -1 \\
                                          -1\\
                                          \vdots \\
                                          -1 \\
                                          -1\\
                                        \end{array}
                                      \right)
  \end{equation*}
  Notice that $\e_1,\dots,\e_n$ are all interior lattice points of $\D_{\aa'}$. Then they are also interior lattice point of $\D_\aa$, which cannot be a reflexive polytope.

 Finally, (\ref{E-funzione a}) follows by (\ref{E-funzione}), recalling that $\Si_{\aa}$ is a complete fan in $M_\R$ and noticing that, being $\D_\aa$ a reflexive polytope, $\vf_K(\m)=-h$, where $h$ is defined in (\ref{h}). The fact that $\Est(\XX_\aa;u,v)$ is a polynomial of degree $2n$ follows by statement $(*)$ in Remark~\ref{rem:gorenstein}, recalling that $\XX_\aa$ has Gorenstein singularities by item (3), and by Remark~\ref{rem:Poincare}.
\end{proof}

Putting together Proposition~\ref{prop:E-crepant} and Proposition~\ref{prop:n+1}, one gets immediately the following
\begin{corollary}\label{cor:hdgnbr_n+1}
  If $|\aa|=n+1$ then $\Est(\XX_\aa;u,v)=\sum_{h=0}^n c_hu^hv^h$\,, with
  \begin{eqnarray*}
   \forall\,h\in\N:0\leq h\leq n\quad c_h&=&\sum_{i=0}^{n-h}(-1)^i{n \choose n-i}\psi_\aa(n-h-i)\\
   &=&\sum_{j=0}^{n-h}(-1)^{n-h-j}{n\choose h+j}\psi_\aa(j)
  \end{eqnarray*}
  In particular, if $\phi:\widehat{\XX}_\aa\longrightarrow\XX_\aa$ is a crepant resolution then
  \begin{equation*}
    h^{p,q}(\widehat{\XX}_\aa)=\hst^{p,q}(\XX_\aa)=\left\{\begin{array}{cc}
                              c_p & \text{when $p=q$} \\
                              0 & \text{otherwise}
                            \end{array}
    \right.
  \end{equation*}
\end{corollary}

Assume now $\aa=\aa_0:=(\1_{n},\d)$ with $\d\geq 2$. By Proposition~\ref{prop:n+1}, $\XX_{\aa_0}$ has not Gorenstein singularities, but $\Q$-Gorenstein ones. This is clear also by recalling the anti-canonical polytope of $\XX_{\aa_0}$, given by
\begin{equation}\label{polycanonico}
  \D_{-K_{\XX_{\aa_0}}}=\left(
                          \begin{array}{cccc}
                            \e_1&\cdots&\e_n&-1/(d-n)\1_n
                          \end{array}
                        \right)
\end{equation}
This means that $\vf_K$ assumes integer values on $M\cap|\Si_{\aa_0}\setminus\langle\L_{\aa_0}^{\{n+1\}}\rangle|$, while it may assume rational, non integer, values on
\begin{equation*}
  M\cap\langle\L_{\aa_0}^{\{n+1\}}\rangle=M\cap\left\langle
                          \begin{array}{c}
                            d-1 \\
                            -1 \\
                            \vdots \\
                            -1 \\
                          \end{array}\,
                                        \begin{array}{c}
                                          \\
                                          \ldots\\
                                          \\
                                           \\
                                        \end{array}\,
                                        \begin{array}{c}
                                          -1 \\
                                          \vdots \\
                                          -1 \\
                                          d-1\\
                                        \end{array}\right\rangle
\end{equation*}

\begin{proposition}\label{prop:risoluzione}
Consider the $n\times (2n+1)$ integer matrix
\begin{equation*}
  \L'_{\aa_0}:=\left(
                 \begin{array}{ccccc}
                   \L_{\aa_0} & | & \e_1 & \cdots & \e_n \\
                 \end{array}
               \right)
\end{equation*}
Let $\Si'_{\aa_0}\in\SF(\L'_{\aa_0})$ be a subdivision of the fan $\Si_{\aa_0}$ and consider the associated toric variety $\XX'_{\aa_0}(\Si'_{\aa_0})$. Then:
\begin{enumerate}
  \item $\XX_{\aa_0}'$ has at worst Gorenstein singularities,
  \item there exists a birational morphism $f:\XX'_{\aa_0}\longrightarrow\XX_{\aa_0}$ which is a partial resolution of $\XX_{\aa_0}$,
  \item there exists a crepant resolution $\phi':\widehat{\XX}_{\aa_0}\longrightarrow\XX'_{\aa_0}$ such that  $$\phi=f\circ\phi':\widehat{\XX}_{\aa_0}\longrightarrow\XX_{\aa_0}$$
      is a resolution of singularities of $\XX_{\aa_0}$\,.
\end{enumerate}
In particular, (1) shows that the canonical support function of $\XX'_{\aa_0}$ is a well defined continuous function $\vf_K:M_\R\longrightarrow\R$ such that
\begin{itemize}
  \item[$(i)$] $\vf_K(\ll'_i)=1$ for every column $\ll'_i$ of the fan matrix $\L'_{\aa_0}$ of $\XX_{\aa_0}'$\,,
  \item[$(ii)$] $\vf_K$ is linear on each cone $\s'\in\Si'_{\aa_0}$\,,
  \item[$(iii)$] $\vf_K(\m)\in\Z$ for every $\m\in M$.
\end{itemize}
Then, by setting
\begin{eqnarray*}
    \forall\,h\in\N\quad\vf_{\aa_0}(h)&:=&\left|\{\m\in M\,|\,\vf_K(\m)=-h\}\right|\\
    \forall\,p\in\N:0\leq p\leq n\quad\quad\quad c'_p&:=&\sum_{h=0}^{n-p}(-1)^h{n \choose n-h}\vf_{\aa_0}(n-p-h)\\
   &=&\sum_{h=0}^{n-p}(-1)^{n-p-h}{n\choose p+h}\vf_{\aa_0}(h)
  \end{eqnarray*}
  one has
\begin{equation*}
    h^{p,q}(\widehat{\XX}_{\aa_0})=\hst^{p,q}(\XX'_{\aa_0})=\left\{\begin{array}{cc}
                              c'_p & \text{when $p=q$} \\
                              0 & \text{otherwise}
                            \end{array}
    \right.
  \end{equation*}
 \end{proposition}

 \begin{proof}
    Keeping in mind expression (\ref{polycanonico}) of the anti-canonical polytope $\D_{-K_{\XX_{\aa_0}}}$ and definition, given in (\ref{KsuppFunct}), of the canonical support function $\vf_K$, lattice points $\e_1,\ldots,\e_n\in N$ turn out defining $\vf_K$ on every maximal cone in $\Si'_{\aa_0}(n)$ which is not contained in the $n$-cone $\langle\L_{\aa_0}^{\{n+1\}}\rangle\in\Si_{\aa_0}$. Moreover, the introduction of new rays $\langle\e_1\rangle,\ldots,\langle\e_n\rangle\subset\langle\L_{\aa_0}^{\{n+1\}}\rangle$ determines a subdivision of the $n$-cone $\langle\L_{\aa_0}^{\{n+1\}}\rangle\in\Si(n)$ in $n+1$ simplicial maximal cones of $\Si'_{\aa_0}(n)$. The definition of $\vf_K$ on these $n+1$ cones is then, respectively, assigned by the following $n+1$ vectors
   \begin{equation*}
     \left(\begin{array}{c}
            -d+n \\
             -\1_{n-1}\end{array}\right)
               \,,\cdots,\,
      \left(\begin{array}{c}
             -\1_{n-1} \\
                  -d+n
               \end{array}\right)\,,\,-\1_n \in N
\end{equation*}
   Then $\vf_K:M_\R\longrightarrow\R$ is well defined and satisfying properties $(i),(ii),(iii)$ in the statement. This suffices to guarantee that $\XX'_{\aa_0}$ admits at worst Gorenstein singularities, so proving (1).
   For (2), notice that $\Si'_{\aa_0}$ is a subdivision of $\Si_{\aa_0}$. Then the identity map $\id_{M_\R}$ induces a map of fans $f_\#:\Si'_{\aa_0}\longrightarrow\Si_{\aa_0}$ and then a well defined birational morphism $f:\XX'_{\aa_0}\longrightarrow\XX_{\aa_0}$\,.

   \noindent Finally, the crepant resolution $\widehat{\XX}_{\aa_0}$ is obtained by further subdividing $\Si'_{\aa_0}$ by adding all the new rays associated with the $A$-triangulation of $\D_{\aa_0}$, in the sense of \cite[Def.~2.2.15]{Batyrev94}, obtained by setting
\begin{equation}\label{A}
  A=\{\m\in M\,|\,\vf_K(\m)\leq 1\}
\end{equation}
Then $\phi'$ is constructed like $f$ and it is a crepant resolution of $\XX'_{\aa_0}$ \cite[Thm.~2.2.24]{Batyrev94}, so proving (3).

   \noindent The last part of the statement, about the computation of Hodge numbers of $\widehat{\XX}_{\aa_0}$, follows immediately by Definition~\ref{def:stringy} and Propositions~\ref{prop:smooth-E} and \ref{prop:E-crepant}\,.
 \end{proof}

 \subsubsection{\ka moduli of the generic $f$-dual hypersurface $\widehat{Y}^\vee\subset\widehat{\XX}_{\aa_0}$} Consider the calibrated $f$-process $(\P^n,D_{\aa_0})\leftrightsquigarrow(\XX_{\aa_0},D'_{\bb_0})$.
 Let $Y^\vee$ be the generic hypersurface of $\XX_{\aa_0}$ in the linear system $|D'_{\bb_0}|$ whose defining polynomial is $f^\vee\in\Cox(\XX_{\aa_0})$, described in (\ref{f}). Define the transformed hypersurface
\begin{equation}\label{trasformata}
  \widehat{Y}^\vee:=\phi^{-1}(Y^\vee)
\end{equation}
 as the zero-locus of $\phi^*(f^\vee)\in\Cox(\widehat{\XX}_{\aa_0})$ under a (partial) resolution $\phi:\widehat{\XX}_{\aa_0}\longrightarrow\XX_{\aa_0}$, where $\phi$ can be either the one constructed in Proposition~\ref{prop:risoluzione} or a suitable partial resolution factorizing it.  Assuming $n\geq 4$, we now subdivide the computation of $h^{1}(\widehat{\Omega}_{\widehat{Y}^\vee})$ in three different cases: $d=n+1$, $d=n+2$ and $d\geq n+3$\,.

 \oneline $d=n+1$\,.\quad Let $Y\subset\P^n$ be a generic hypersurface of degree $d=n+1$. It is a \cy variety and what follows is a particular case of the Batyrev duality for anti-canonical hypersurfaces of Fano toric varieties described in \cite{Batyrev94}\,.

 \noindent Since $d=n+1$ then $\aa_0=\1_{n+1}$ and also $\bb_0=\1_{n+1}$.

 \begin{theorem}\label{thm:B-mirror0}
   There exists a partial crepant resolution $\phi:\widehat{\XX}_{\1}\longrightarrow\XX_1$ such that $\widehat{Y}^\vee=\phi^{-1}(Y^\vee)$ is a smooth \cy resolution of $Y^\vee$ and
   \begin{equation*}
     h^{1}\left(\Omega_{\widehat{Y}^\vee}\right)=-(n+1){n\choose n-1}+\sum_{i=1}^n (-1)^{n-i}{n+1\choose i+1}{in+i-1\choose n}
   \end{equation*}
   In particular,
   \begin{equation*}
    h^{2}(\Omega_Y)=m_Y={2n+1\choose n+1}-(n+1)^2=h^{1}(\Omega_{\widehat{Y}^\vee})
   \end{equation*}
   so giving the B-side mirror symmetry between the generic anti-canonical hypersurfaces $Y\subset \P^n$ and $Y^\vee\subset\XX_\1$\,.
   Recalling (\ref{Aside}), this means that $(Y,Y^\vee)$ is a pair of topological and Hodge mirror symmetric partners (notation as in \cite[Def.~3.5]{R-fTV}).
 \end{theorem}

 \begin{proof} Consider the generic hypersurface $Y^\vee\in|D'_\1|$.  It is the zero-locus of
 \begin{equation*}
  f^\vee=\sum_{i=1}^{n+1}  x_i^{n+1} +\psi\,\prod_{j=1}^{n+1} x_j
\end{equation*}
as can be immediately deduced by setting $d=n+1$ and $\d=1$ in (\ref{f})\,. Then $Y^\vee$ is quasi-smooth, as $\widetilde{Y}^\vee=\pi^{-1}(Y^\vee)$ is a smooth hypersurface of $\P^n$, where $\pi:\P^n\twoheadrightarrow\XX_\1$ is the canonical projection of the $G_{\1}$-action described in \cite[Lem.~5.2]{R-fTV}. More in detail, in the resolution process described in Proposition~\ref{prop:risoluzione}, $Y^\vee$ can be assumed not passing through centers of those blowups whose exceptional divisors are the closure of the torus orbits of rays generated by  interior points of a facet of $\D_\1$, since those centers are just given by isolated points. This means that the resolution process can be stopped earlier, thus avoiding those blowups, and obtaining a partial crepant resolution $\phi:\widehat{\XX}_\1\longrightarrow\XX_\1$ factorizing the one described in Proposition~\ref{prop:risoluzione}.
In particular, $\widehat{Y}^\vee$ is a \cy hypersurface as $\phi$ is crepant. Then, Corollary~\ref{cor:hdgnbr_n+1} gives
\begin{eqnarray*}
   h^{1}\left(\widehat{\Omega}_{\widehat{\XX}_\1}\right)&=&\hst^{1,1}(\XX_\1) -\sum_{\Theta<^1\D_\1}l^*(\Theta)=c_1-\sum_{\Theta<^1\D_\1}l^*(\Theta)\\
   &=&\sum_{j=0}^{n-1}(-1)^{n-1-j}{n\choose j+1}\psi_\1(j)-\sum_{\Theta<^1\D_\1}l^*(\Theta)\\
   &=&\sum_{j=0}^{n-1}(-1)^{n-1-j}{n \choose j+1} \left[l(j\D_\1)-l((j-1)\D_\1)\right]-\sum_{\Theta<^1\D_\1}l^*(\Theta)\\
   &=&\sum_{j=0}^{n-1}(-1)^{n-1-j}{n+1\choose j+2}l(j\D_\1)-\sum_{\Theta<^1\D_\1}l^*(\Theta)\\
   &=&\sum_{i=1}^{n}(-1)^{n-i}{n+1\choose i+1}{in+i-1\choose n}-(n+1){n\choose n-1}
\end{eqnarray*}
where, in the order, we applied the definition of $\psi_\1$ given in Proposition~\ref{prop:n+1}, the identity ${n\choose j+1}+{n\choose j+2}={n+1\choose j+2}$, the substitution $i=j+1$ and observed that
\begin{equation*}
  l((i-1)\D_\1)=h^0\left(\cO_{\P^n}((i-1)(n+1))\right)={in+i-1\choose n}
\end{equation*}
 The computation of $h^{1}(\Omega_{\widehat{Y}^\vee})$ follows, now, from Theorem~\ref{thm:h21Yd}, that is
\begin{eqnarray*}
  h^{1}\left(\Omega_{\widehat{Y}^\vee}\right)&=&\rk\left(\Cl\left(\widehat{\XX}_\1\right)\right)= h^{1}\left(\widehat{\Omega}_{\widehat{\XX}_\1}\right)\\
  &=&-(n+1){n\choose n-1}+\sum_{i=1}^n (-1)^{n-i}{n+1\choose i+1}{in+i-1\choose n}
\end{eqnarray*}
Finally, the second part of the statement is just the evaluation for $d=n+1$ of the combinatorial Lemma~\ref{lem:combinatorica}.
 \end{proof}

\oneline $d=n+2$.\quad  Let $Y\subset\P^n$ be a generic hypersurface of degree $d=n+2$. Then $Y$ is the lowest degree case of a projective hypersurface of general type. In particular, $\aa_0=(\1_{n},2)$ and $\bb_0=(2\cdot\1_{n},1)$.

 \begin{theorem}\label{thm:B-mirror1}
   There exists a partial resolution $\phi:\widehat{\XX}_{(\1,2)}\longrightarrow\XX_{(\1,2)}$, factorizing the one given in Proposition~\ref{prop:risoluzione}, and such that the transformed hypersurface
\begin{equation*}
  \widehat{Y}^\vee=\phi^{-1}(Y^\vee)\subset\widehat{\XX}_{(\1,2)}
\end{equation*}
is quasi-smooth and
   \begin{equation*}
     h^{1}\left(\widehat{\Omega}_{\widehat{Y}^\vee}\right)=-{n+1\choose n-1}+c'_1=-{n+1\choose n-1}+\sum_{h=0}^{n-1}
     (-1)^{n-1-h}{n\choose h+1}\vf_{(\1,2)}(h)
   \end{equation*}
   where $\vf_{(\1,2)}(h)$ admits the following recursive expression
   \begin{eqnarray*}
     \vf_{(\1,2)}(0)&=&1\\
     \forall\,h\in\N\setminus\{0\}\quad\vf_{(\1,2)}(h)&=&{h(n+2)+n\choose n}\\
     &&-\sum_{j=0}^{h-1}\left[{(j+1)(n+1)+h-1\choose n-1}+\vf_{(\1,2)}(j)\right]
   \end{eqnarray*}
   In particular,
   \begin{equation*}
    m_Y={2n+2\choose n}-(n+1)^2=h^{1}\left(\widehat{\Omega}_{\widehat{Y}^\vee}\right)=k_{\widehat{Y}^\vee}
   \end{equation*}
   that is, the generic $Y^\vee\subset\XX_{(\1,2)}$ is a $B$-mirror partner of the generic hypersurface $Y\subset \P^n$ of degree $n+2$. By (\ref{Aside}), this means that $(Y,Y^\vee)$ is a pair of topologically mirror partners.
 \end{theorem}

 \begin{proof}
   Consider the generic hypersurface $Y^\vee\in|D'_{(2\cdot\1,1)}|$ whose singularities are all contained in the prime toric divisor $D'_{n+1}\subset\XX_{(\1,2)}$. Calling $\Si_{(\1,2)}$ the fan of $\XX_{(\1,2)}$, whose fan matrix is given by
   \begin{equation*}
     \L_{(\1,2)}=\left(
                                      \begin{array}{ccc}
                                        \ll_1 & \cdots & \ll_{n+1} \\
                                      \end{array}
                                    \right)
   \end{equation*}
   in the following we will consider the maximal cones
   \begin{equation*}
     \s_i=\langle \ll_j\,|\,j\neq i\rangle\in\Si_{(\1,2)}(n)
   \end{equation*}
   Then, the following \emph{stratum of $Y^\vee$} (see \cite[Def.~1.10]{R-fTV}) is empty
      \begin{equation*}
        Y^\vee_{\s_{n+1}}:=Y^\vee\cap\T\cdot x_{\s_{n+1}}=Y^\vee\cap\{[0:\cdots:0:1]\}=\emptyset
      \end{equation*}
      where $x_{\s_{n+1}}$ is the special point of the cone $\s_{n+1}$,
      as can be immediately deduced from the polynomial $f^\vee$ in (\ref{f}). In other word, recalling the resolution constructed in Proposition~\ref{prop:risoluzione} and obtained by means of the $A$-triangulation defined in (\ref{A}), $Y^\vee$
      can be assumed not passing through the points of $\XX_{(\1,2)}$ which are centers of those blow ups determined by lattice points in $A\cap\s_{n+1}^\circ$, where $\s_{n+1}^\circ$ denotes the relative interior of $\s_{n+1}$. Notice that these are precisely given by
      \begin{equation*}
        A\cap\s_{n+1}^\circ=\{\e_1,\ldots,\e_n\}
      \end{equation*}
      and one can stop the resolution process of $\XX_{\aa_0}$ avoiding the blowups associated with $\{\e_1,\ldots,\e_n\}$, so getting an induced partial resolution $Y'\longrightarrow Y^\vee$. Then, Theorem~\ref{thm:h21Yd}, Lemma~\ref{lem:h11st} and relation (\ref{mY}) give
      \begin{equation*}
        h^1(\widehat{\Omega}_{Y'})=h^{1,1}_{\text{st}}\left(\XX_{(\1,2)}'\right)-|A\cap\s_{n+1}^\circ|=c'_1-n=m_Y+\frac{n^2-n}{2}
      \end{equation*}
      On the contrary,
      \begin{equation*}
        \forall\,1\leq i\leq n\quad [0:\cdots:\underset{i}{1}:\cdots:0]\in Y^\vee\ \Longrightarrow\ Y^\vee_{\s_i}=\T\cdot x_{\s_i}
      \end{equation*}
      meaning that $Y^\vee$ passes through every center of those blow ups composing $\phi$ and determined by lattice points of $A\cap\s_i^\circ$, for $1\leq i \leq n$. The de\-sin\-gu\-la\-ri\-za\-tion process described in \S\ref{ssez:locale} shows that, for any such $i$, at most two of the lattice points in $A\cap\s_i^\circ$ are really needed. Recalling that
      $$\forall\,i=1,\ldots,n\quad |A\cap\s_{i}^\circ|=l^*(\Theta_i)=\binom{d-1}{n-1}=\frac{n^2+n}{2}$$
      with $\Theta_i:=\conv(\ll_1,\ldots,\cancel{\ll_{i}},\ldots,\ll_{n+1})$,
            this is enough to show that the resolution process can be stopped $(n^2-n)/2$ steps before, so getting a quasi-smooth partial resolution $\widehat{Y}^\vee\longrightarrow \widehat{Y}$ such that
      \begin{equation*}
        k_{\widehat{Y}^\vee}=h^{1}\left(\widehat{\Omega}_{\widehat{Y}^\vee}\right)=m_Y
      \end{equation*}
\end{proof}

\oneline \emph{$d=n+\d$ with $\d\geq 3$}.\quad  Let $Y\subset\P^n$ be a generic hypersurface of degree $d=n+\d\geq n+3$: then $\aa_0=(\1_{n},\d)$ and $\bb_0=(\d\cdot\1_{n},1)$.

\begin{theorem}\label{thm:B-mirror2} Let $\phi:\widehat{\XX}_{(\1,\d)}\longrightarrow\XX_{(\1,\d)}$, be the resolution constructed in Proposition~\ref{prop:risoluzione}. Then, the transformed hypersurface
\begin{equation*}
  \widehat{Y}^\vee=\phi^{-1}(Y^\vee)\subset\widehat{\XX}_{(\1,\d)}
\end{equation*}
is smooth and
\begin{equation*}
  h^{1}\left(\Omega_{\widehat{Y}^\vee}\right)=\binom{d+n}{n}-\binom{d}{n}<m_Y
\end{equation*}
Moreover, there exists a birational morphism $f':\widehat{\XX}'_{(\1,\d)}\longrightarrow\widehat{\XX}_{(\1,\d)}$, which is a composition of divisorial blowups, such that, calling $\vf:=\phi\circ f'$, the associated transformed hypersurface $(\widehat{Y}^\vee)':=\vf^{-1}(Y^\vee)=(f')^{-1}(\widehat{Y}^\vee)$ of $Y^\vee$ is smooth and
  \begin{equation*}
        k_{(\widehat{Y}^\vee)'}=h^{1}\left(\Omega_{(\widehat{Y}^\vee)'}\right)= \rk\left(\Cl\left(\widehat{\XX}'_{(\1,\d)}\right)\right)= h^{1}\left(\Omega_{\XX'_{(\1,\d)}}\right)= m_Y
      \end{equation*}
 that is, the generic $Y^\vee\subset\XX_{(\1,\d)}$ is a $B$-mirror partner of the generic hypersurface $Y\subset \P^n$ of degree $d=n+\d$. By (\ref{Aside}), this means that $(Y,Y^\vee)$ is a pair of topologically mirror partners.
\end{theorem}

\begin{proof}
 Consider the generic hypersurface $Y^\vee\in|D'_{(\d\cdot\1,1)}|$ and its transformed hypersurface $\widehat{Y}^\vee=\phi^{-1}(Y^\vee)\subset\widehat{\XX}_{(\1,\d)}$. Local analysis explained in the next \S\ref{ssez:locale} suffices to show that $\widehat{Y}^\vee$ is smooth. Then, Theorem~\ref{thm:h21Yd}, Proposition~\ref{prop:smooth-E} and the following  Lemma~\ref{lem:h11st} give
 \begin{equation*}
   h^{1}\left(\Omega_{\widehat{Y}^\vee}\right)= \rk\left(\Cl\left(\widehat{\XX}_{(\1,\d)}\right)\right)=\hst^{1,1}\left(\XX'_{(\1,\d)}\right)= \binom{d+n}{n}-\binom{d}{n}
 \end{equation*}
Since $\d\ge 3$, one has
\begin{equation*}
  \binom{d}{n}=\binom{n+\d}{n}>(n+1)^2
\end{equation*}
so giving that
  \begin{equation*}
    h^{1,1}\left(\widehat{Y}^\vee\right)<\binom{d+n}{n} - (n+1)^2=m_Y
  \end{equation*}
  Proceed, now, to blowing up $\widehat{\XX}_{(\1,\d)}$ in $s:=\binom{d}{n}-(n+1)^2$  points belonging to $\widehat{Y}$, to getting the birational morphism
  \begin{equation*}
    \xymatrix{f':\widehat{\XX}'_{(\1,\d)}:=B_s\widehat{\XX}_{\1,\d}\ar[r]&\widehat{\XX}_{\1,\d}}
  \end{equation*}
  Then, finally
  \begin{eqnarray*}
    h^{1,1}\left((\widehat{Y}^\vee)'\right)&=&  h^{1,1}\left(\widehat{\XX}'_{(\1,\d)}\right)= s+\hst^{1,1}\left(\XX'_{(\1,\d)})\right) \\
    &=& \binom{d+n}{n} - (n+1)^2 =m_Y
  \end{eqnarray*}
\end{proof}

\subsection{Local analysis of singularities and resolutions}\label{ssez:locale} The present paragraph is devoted to giving a proof of the existence of a smooth resolution $\widehat{Y}^\vee\subset\widehat{\XX}_{\aa_0}$ of a generic $Y^\vee\in|D'_{\bb_0}|$, as claimed in previous Theorems~\ref{thm:B-mirror1} and \ref{thm:B-mirror2}, being $\phi:\widehat{\XX}_{\aa_0}\longrightarrow \XX_{\aa_0}$ the resolution given in Proposition~\ref{prop:risoluzione}\,(3).

Let $Y$ be a generic hypersurface in $\P^n$ of degree $d=n+\d$, with $\d\geq 2$. Then, after acting an automorphism of $\XX_{\aa_0}$, the generic hypersurface $Y^\vee\in|D'_{\aa_0}|$ is defined by the polynomial $f^\vee\in\Cox(\XX_{\aa_0})\cong\C[\x]$ described in (\ref{f}).

First of all recall that $\Sing(Y^\vee)\subset D'_{n+1}$. Then, from here on, we will restrict to consider singularities in the affine open subset $U_n:=\{x_n\neq 0\}\subset\XX_{\aa_0}$: this suffices as the treatment of singularities in the remaining affine subsets $U_i$, $1\leq i\leq n-1$, is completely analogous and $Y^\vee\cap U_{n+1}$ is smooth. Then $Z_n:=Y^\vee\cap U_n$ is the hypersurface defined in $\C^n$ by the polynomial
\begin{eqnarray}\label{fn}
\nonumber
  f_n&:=&\prod_{i=1}^{n-1}z_i^{\d-1}\left(\sum_{i=1}^{n-1} z_i^{n+\d}+1+\psi z_n\prod_{i=1}^{n-1}z_i\right)+z^{n+1}_{n}\\
    &=&\zeta\prod_{i=1}^{n-1}z_i^{\d-1}+z^{n+1}_{n}
\end{eqnarray}
where:
$$\forall\,i=1,\ldots,n-1 \quad z_i=x_{i}/x_n\,,\quad z_n=x_{n+1}/x_n^\d\,,\quad\zeta:=\sum_{i=1}^{n-1} z_i^{n+\d}+1+\psi z_n\prod_{i=1}^{n-1}z_i$$
This means that we are studying the open subset determined by the cone $$\s_n:=\langle\L_{\aa_0}^{\{n\}}\rangle=\langle\ll_1,\ldots,\ll_{n-1},\ll_{n+1}\rangle$$
opposite to the ray $\rho_{n}=\langle\ll_{n}\rangle$.
The singular locus $\Sing(Z_n)$ is contained in the hyperplane $\{z_n=0\}$, which is determined by the torus orbit of the distinguished point of the ray $\rho_{n+1}=\langle\ll_{n+1}\rangle$. Then, a resolution of $Z_n$ has to be obtained as a birational transform induced by a suitable subdivision of $\s_n$, admitting a sub-cone $\s\subset\s_n$ such that:
\begin{itemize}
  \item $\s$ is a cone of a subdivision of $\s_n$ constructed by adding a suitable number of new rays, associated with exceptional divisors of successive blowups,
  \item $\rho_{n+1}$ is still a ray of $\s$.
\end{itemize}
 A (not unique) way of producing a similar resolution is the following. The cone $\s_n$ is spanned by the facet $\Theta_n<^1\D_{\aa_0}$ defined by
 $$\Theta_n:=\conv(\ll_1,\ldots,\ll_{n-1},\ll_{n+1})$$
 Recalling definition (\ref{A}) of the $A$-triangulation of $\D_{\aa_0}$, giving rise to the crepant resolution $\phi':\widehat{\XX}_{\aa_0}\longrightarrow \XX'_{\aa_0}$ in Proposition~\ref{prop:risoluzione}, every interior lattice point of $\Theta_n$ determines a new ray, hence an exceptional divisor of the resolution $\phi'$. Notice that their number is given by
 \begin{equation*}
   l^*(\Theta_n):=|M\cap\Int(\Theta_n)|=\binom{d-1}{n-1}=\binom{n+\d-1}{\d}
 \end{equation*}
Start the resolution process by considering the birational transform of $Z_n$ in the blowup determined by $\ll^{(1)}\in M\cap\Int(\Theta_n)$, locally given by
\begin{equation*}
  \left\{\begin{array}{cc}
           \forall\,i:1\le i\le n-1&z_i= t_{i}z_n \\
           &f_n(\z)=0\\
         \end{array}
  \right\}\ \Longrightarrow\ \zeta^{(1)}z_n^{(n-1)(\d-1)}\prod_{i=1}^{n-1}t_{i}^{\d-1}+z_n^{n+1}=0
\end{equation*}
being $\zeta^{(1)}$ the obvious transform of $\zeta$. The associated birational transform $$B_1Z_n\subset\Spec\left(A_{\s_n^{(1)}}\right)\,,\quad A_{\s_n^{(1)}}:=\C\left[(\s_n^{(1)})^\vee\cap N\right]$$
with $\s_n^{(1)}:=\langle\ll^{(1)},\ll_{i_2},\ldots,\ll_{i_{n-1}},\ll_{n+1}\rangle$ and $\{i_2\ldots,i_{n-1}\}$ a suitable subset of $\{1,\ldots,n-1\}$, is given by the zero locus of
\begin{equation*}
  f^{(1)}_n:=\left\{\begin{array}{cc}
                      \zeta^{(1)}\prod_{i=1}^{n-1}t_{i}+z_n^2\in\C[t_{1},\ldots,t_{n-1},z_n] & \text{if $\d=2$} \\
                       \zeta^{(1)}z_n^{(\d-2)n-\d}\prod_{i=1}^{n-1}t_{i}^{\d-1}+1\in
                       \C[t_{1},\ldots,t_{n-1},z_n] & \text{if $\d\ge 3$}
                    \end{array}\right.
\end{equation*}
Then, singularities are not yet resolved if $\d=2$. In this case, we go on by considering the blowup determined by
$$\ll^{(2)}\in M\cap\Int(\Theta^{(1)}_n)\,, \quad \Theta^{(1)}_n:=\conv\left(\ll^{(1)},\ll_{i_2},\ldots,\ll_{i_{n-1}},\ll_{n+1}\right)$$
and the associated birational transform $B_2Z_n\subset\Spec\left(A_{\s_n^{(2)}}\right)$, with  $$\s_n^{(2)}:=\langle\ll^{(2)},\ll_{j_2},\ll_{j_3},\ldots,\ll_{j_{n-1}},\ll_{n+1}\rangle\,,
\quad\{\ll_{j_2},\ldots,\ll_{j_{n-1}}\}\subset\{\ll^{(1)},\ll_{i_2},\ldots,\ll_{i_{n-1}}\}$$
Notice that such a $\ll^{(2)}$ exists, up to changing the previous choice of $\ll^{(1)}$, since $l^*(\Theta_n)>n$.
Then, this process terminates after at most two blowups, producing a smooth resolution $\widehat{Z}_n\longrightarrow Z_n$. \\
Analogously for other open subsets $Z_i:=Y^\vee\cap U_i$, with $1\le i\le n-1$.

\begin{example}[Resolving the $f$-mirror partner of a septic threefold]
  Consider the generic element $Y:=Y_7\in|D_{\aa_0}|$ with $\aa_0:=(\1_4,3)$ thought of as a framing of $\P^4$. Then the $f$-mirror partner of $Y$ is given by the generic element $Y^\vee\in|D'_{\bb_0}|$, with $\bb_0=(3\cdot\1_4,1)$ thought of as a framing of $\XX_{(\1,3)}\cong\P(\1_4,3)/G_{(\1,3)}$, whose defining polynomial in $\Cox\left(\XX_{(\1,3)}\right)$ is given by
  \begin{equation*}
    f^\vee=\prod_{i=1}^4 x_i^{2}\left(\sum_{i=1}^4 x_i^7 +\psi\,\prod_{j=1}^{5} x_j\right)+x_{5}^{5}
  \end{equation*}
  Then $\Sing(Y^\vee)\subset D'_5$ and we consider the affine open subset
  $U_4\subset\XX_{(\1,3)}$
  and the hypersurface $Z_4=Y^\vee\cap U_4\cong\C^4$, defined by the equation
  \begin{equation*}
    f_4=\zeta z_1^{2}z_2^{2}z_3^{2}+z^{5}_{4}=0
  \end{equation*}
  with $\zeta:=z_1^{7}+z_2^{7}+z_3^{7}+1+\psi z_1z_2z_3z_4$ and
$$z_1=x_{1}/x_4\,,\quad z_2=x_{2}/x_4\,,\quad z_3=x_{3}/x_4\,,\quad z_4=x_{5}/x_4^3$$
Consider the blow up
\begin{equation}\label{blowup}
  \left\{\begin{array}{cc}
           \forall\,i:1\le i\le 3&z_i= t_{i}z_4 \\
           &f_4(z_1,z_2,z_3,z_4)=0\\
         \end{array}
  \right\}\ \Longrightarrow\ z_4^5\left(\zeta^{(1)}z_4t_1^2t_2^2t_3^2+1\right)=0
\end{equation}
with $\zeta^{(1)}=1+\left[(t_1^7+t_2^7+t_3^7)z_4^3+\psi t_1t_2t_3\right]z_4^4$ so giving the birational transform $B_1Z_4\subset\Spec\left(A_{\s_4^{(1)}}\right)\cong\C^4$ defined by the polynomial
\begin{equation*}
  \zeta^{(1)}z_4t_1^2t_2^2t_3^2+1\in\C[t_1,t_2,t_3,z_4]
\end{equation*}
which is a non-singular one. The same procedure can now be repeated for the affine open subsets $U_1,U_2,U_3$, so getting a resolution of $Y^\vee$ after four blowups.
\end{example}

\begin{example}[Resolving the $f$-mirror partner of a sextic threefold] In the same notation as above, the $f$-process is now given by
\begin{equation*}
  \xymatrix{\left(\P^4,\aa_0:=(\1_4,2)\right)\ar@{<~>}[r]&\left(\XX_{(\1,2)},\bb_0:=(2\cdot\1_4,1)\right)}
\end{equation*}
and the generic $f$-mirror partner $Y^\vee$ is defined by
\begin{equation*}
    f^\vee=\prod_{i=1}^4 x_i\left(\sum_{i=1}^4 x_i^6 +\psi\,\prod_{j=1}^{5} x_j\right)+x_{5}^{5}
  \end{equation*}
Then, locally $Z_4=Y^\vee\cap U_4$ is given by
\begin{equation*}
    f_4=\zeta z_1z_2z_3+z^{5}_{4}=0
  \end{equation*}
  with $\zeta:=z_1^{6}+z_2^{6}+z_3^{6}+1+\psi z_1z_2z_3z_4$ and
$$z_1=x_{1}/x_4\,,\quad z_2=x_{2}/x_4\,,\quad z_3=x_{3}/x_4\,,\quad z_4=x_{5}/x_4^2$$
After the blow up given in (\ref{blowup}), the birational transform $B_1Z_4$ is defined by the polynomial
\begin{equation*}
  f'_4:=\zeta^{(1)}t_1t_2t_3+z_4^2\in\C[t_1,t_2,t_3,z_4]
\end{equation*}
with $\zeta^{(1)}=1+\left[(t_1^6+t_2^6+t_3^6)z_4^2+\psi t_1t_2t_3\right]z_4^4$. Then one need a further blowup given by
\begin{equation*}
  \left\{\begin{array}{cc}
           \forall\,i:1\le i\le 3&t_i= s_{i}z_4 \\
           &f'_4(t_1,t_2,t_3,z_4)=0\\
         \end{array}
  \right\}\ \Longrightarrow\ z_4^2\left(\zeta^{(2)}z_4s_1s_2s_3+1\right)=0
\end{equation*}
 so that, the birational transform $B_2Z_4$ is given by
 \begin{equation*}
   \zeta^{(2)}z_4s_1s_2s_3+1\in\C[s_1,s_2,s_3,z_4]
 \end{equation*}
 with $\zeta^{(2)}=1+\left[(s_1^6+s_2^6+s_3^6)z_4^5+\psi s_1s_2s_3\right]z_4^7$. Then $B_2Z_4$ is smooth, meaning that $Y^\vee$ can be resolved after 8 blowups.
\end{example}

\subsection{Two combinatorial Lemmas}\label{ssez:combinatorica} This section is devoted to prove some combinatorial formulas needed to compute Hodge numbers in previous Theorems~\ref{thm:B-mirror0}, \ref{thm:B-mirror1} and \ref{thm:B-mirror2}.

\begin{lemma}\label{lem:combinatorica} For every positive integers $n\geq 1$ and $d\geq n+1$ the following is an identity
 \begin{equation}\label{identita}
   \binom{2d-1}{n}-(n+1){d-1\choose n}=\sum_{i=1}^n (-1)^{n-i}{n+1\choose i+1}{id-d+n\choose n}
 \end{equation}
   \end{lemma}
   \begin{proof}
     Set $P(k):={(k-1)d-d+n\choose n}$ and think it as a polynomial of degree $\leq n$ in $k$. Then, the $(n+1)^{\text{}st}$ finite differences of $P(k)$ vanish, that is,
     \begin{equation*}
       \sum_{k=0}^{n+1} (-1)^k{n+1\choose k}P(k)=0\ \Longrightarrow\ \sum_{k=2}^{n+1} (-1)^k{n+1\choose k}P(k)=-P(0)+(n+1)P(1)
     \end{equation*}
     Multiplying by $(-1)^{n-1}$ gives
     \begin{equation*}
       (-1)^nP(0)-(n+1)(-1)^nP(1)=\sum_{k=2}^{n+1} (-1)^{n-1+k}{n+1\choose k}P(k)
     \end{equation*}
     Notice that
     $$(-1)^nP(0)=\binom{2d-1}{n}\ ,\quad (-1)^nP(1)=\binom{d-1}{n}\ ,\quad \forall\,k\quad (-1)^{n-1+k}=(-1)^{n-1-k}$$
     Then setting $k=i+1$ in the summation gives (\ref{identita})\,.
   \end{proof}

\begin{lemma}\label{lem:h11st}
Consider the Gorenstein partial resolution $\XX_{\aa_0}'\longrightarrow\XX_{\aa_0}$ constructed in Proposition~\ref{prop:risoluzione}. Then, for every positive integers $n\geq 4$ and $d\geq n+2$, the following is an identity
\begin{equation}\label{identita2}
\hst^{1,1}(\XX'_{\aa_0})=c'_1=\binom{d+n}{n}-\binom{d}{n}
\end{equation}
  \end{lemma}

  \begin{proof}
 By Proposition \ref{prop:risoluzione}
 \begin{equation}\label{c1'}
   c'_1=\sum_{h=0}^{n-1}(-1)^{n-1-h}{n\choose 1+h}\vf_{\aa_0}(h)
 \end{equation}
  where, by definition, $\vf_{\aa_0}(h)=\left|\{\m\in M\,|\,\vf_K(\m)=-h\}\right|$, being $\vf_K$ the canonical support function of $\XX'_{\aa_0}$. Clearly $\vf_{\aa_0}(0)=1$. Moreover
  \begin{equation*}
    \forall\,h\in\N\setminus\{0\}\quad\vf_{\aa_0}(h)=\binom{n+hd}{n}-\sum_{j=1}^h\left(\sum_{i=1}^{d-n-1}
    \binom{jd-j+h-i}{n-1}\right)-\sum_{l=0}^{h-1}\vf_{\aa_0}(l)
  \end{equation*}
  giving rise to a recursive equation admitting the following unique solution
  \begin{eqnarray*}
    \vf_{\aa_0}(1) &=&  \binom{d+n}{n}-\binom{d}{n}+n\\
    \forall\,h\geq 2\quad \vf_{\aa_0}(h) &=& \sum_{i=1}^{n-1}\binom{hd+i}{n-1}+\sum_{l=0}^{h-1}\binom{ld+n+h-l-1}{n-1}- \sum_{l=1}^{h-1}\binom{ld+h-l-1}{n-1}
  \end{eqnarray*}
  Define
  \begin{equation}\label{Pk}
    \forall\,k\in\Z\quad P(k):= \sum_{i=1}^{n-1}\binom{(k-1)d+i}{n-1}+\sum_{l=0}^{k-2}\binom{ld+n+k-l-2}{n-1}- \sum_{l=1}^{k-2}\binom{ld+k-l-2}{n-1}
  \end{equation}
  with the following usual conventions:
  \begin{eqnarray}\label{convenzioni}
    \binom{-a}{b} &=& (-1)^{b}\binom{a+b-1}{b} \\
    \nonumber
    \forall\,a,b,c\in\Z\quad\sum_{i=a}^{b-1}f(i)+ \sum_{j=b}^cf(j)&=&\sum_{l=a}^cf(l)
  \end{eqnarray}
  In particular, the second convention implies that
  \begin{equation*}
    \sum_{i=1}^0 f(i)=\sum_{i=0}^{-1}f(i)=0,\ \sum_{i=0}^{-2}f(i)=-f(-1),\ \sum_{i=1}^{-2}f(i)=-f(0)-f(-1),\ \sum_{i=1}^{-1}f(i)=-f(0)
  \end{equation*}
  Therefore, one has
  \begin{eqnarray*}
    P(0)&=&\sum_{i=1}^{n-1}\binom{-d+i}{n-1}-\binom{-d+n-1}{n-1}+\binom{-2}{n-1}+\binom{-d-1}{n-1}\\
        &=&\binom{-d+n}{n}-\binom{-d}{n}+(-1)^n\binom{d-1}{n-1}+(-1)^{n-1}n+(-1)^{n-1}\binom{d+n-1}{n-1}\\
        &=&(-1)^n\left[\binom{d-1}{n}+\binom{d-1}{n-1}\right]+(-1)^{n-1}\left[\binom{d+n-1}{n}+\binom{d+n-1}{n-1}+n\right]\\
        &=&(-1)^{n-1}\left[\binom{d+n}{n}-\binom{d}{n}+n\right]=(-1)^{n-1}\vf_{\aa_0}(1)\\
    P(1)&=& \sum_{i=1}^{n-1}\binom{i}{n-1}+\binom{-1}{n-1} = 1+(-1)^{n-1}\\
    P(2)&=& \sum_{i=1}^{n-1}\binom{d+i}{n-1}+n=\binom{d+n}{n}-\binom{d}{n}+n=\vf_{\aa_0}(1)\\
    P(k)&=&\vf_{\aa_0}(k-1)\quad\text{for $k\geq 3$}
  \end{eqnarray*}
  Definition (\ref{Pk}), taking into account conventions (\ref{convenzioni}), allows one to check that $n$-th finite differences have to vanish for $P(k)$, that is,
  \begin{equation}\label{DF}
    \sum_{k=0}^n (-1)^k\binom{n}{k}P(k)=0
  \end{equation}
 Then, (\ref{c1'}) gives
 \begin{eqnarray*}
   c'_1&=&(-1)^n\sum_{k=1}^n(-1)^k\binom{n}{k}\vf_{\aa_0}(k-1)\\
       &=&(-1)^n\left[-n+\binom{n}{2}\vf_{\aa_0}(1)+
       \sum_{k=3}^n(-1)^k\binom{n}{k}P(k)\right]
 \end{eqnarray*}
 and vanishing (\ref{DF}) allows us to conclude that
\begin{eqnarray*}
  c_1' &=& (-1)^n\left[-n+\binom{n}{2}\vf_{\aa_0}(1))-
       \sum_{k=0}^2(-1)^k\binom{n}{k}P(k)\right] \\
   &=& \vf_{\aa_0}(1)-n =\binom{d+n}{n}-\binom{d}{n}
\end{eqnarray*}
  \end{proof}

  \section{An example with $l\ge 2$\,: $Y_{3,4}\subset\P^5$}\label{sez:l ge 2}

  On the contrary of the hypersurface case, when $l\ge 2$ there is not a general rule for computing complex moduli of a family of projective complete intersections. Moreover, dimensions of polytopes $\D_{\bb_1},\ldots,\D_{\bb_l}$, associated with the $f$-mirror partitioned framing $\cv{\bb}=\sum_{k=1}^l\bb_k$, depend on the particular partitioned framing $\aa=\sum_{k=1}^l\aa_k$ chosen on $\P^n$. Anyway, results in \S\ref{ssez:ProjCI} and \S\ref{ssez:mirrorCI} give useful recipes to perform, case by case, most parts of the topological mirror check.

  \begin{remark}
    Consider the partitioned ftv $(\P^n,\aa=\sum_{i=1}^l\aa_k)$ and let $\aa_0$ be a permutation of $\aa$ whose entry are increasingly ordered. Let $\XX_{\aa_0}$ and $\cv{\XX}_\aa$ be the associated $f$-dual and partitioned $f$-dual, respectively, varieties. By construction, $\XX_{\aa_0}$ and $\cv{\XX}_\aa$ are birational algebraic varieties. Recall the Gorenstein partial resolution $\phi':\XX'_{\aa_0}\longrightarrow\XX_{\aa_0}$ constructed in Proposition~\ref{prop:risoluzione}. Then blow up $\XX'_{\aa_0}$ by adding the rays in $\cv{\Si}_{\aa}(1)\setminus\Si_{\aa_0}(1)$ and then suitably subdividing the fan $\cv{\Si}_\aa$: in this way one gets a partial Gorenstein resolution $\XX''_\aa$ of both $\XX_{\aa_0}$ and $\cv{\XX}_\aa$ and factorizing the resolution $\phi:\widehat{\XX}_{\aa_0}\longrightarrow\XX_{\aa_0}$ constructed in Proposition~\ref{prop:risoluzione}, as follows
    \begin{equation}\label{risoluzioni}
      \xymatrix{&&\,\ \widehat{\XX}_{\aa_0}\ar[d]\ar[dddll]_{\phi}\ar[dddrr]^{\cv{\phi}}&&\\
      &&\XX''_\aa\ar[dl]\ar[ddrr]&&\\
      &\XX'_{\aa_0}\ar[dl]^{\phi'}&&&\\
      \XX_{\aa_0}&&&&\cv{\XX}_{\aa}}
    \end{equation}
    In particular: \emph{$\cv{\phi}:\widehat{\XX}_{\aa_0}\longrightarrow\cv{\XX}_\aa$ is a resolution of singularities}.
  \end{remark}

  As an example, consider the complete intersection $Y_{3,4}\subset\P^5$, giving the minimum degree projective 3-dimensional complete intersection, with $l=2$, of general type. Coherently with prescriptions (\ref{a,ak}) choose:
  \begin{eqnarray*}
    \aa&:=&\left(\begin{array}{cccccc}
          1 & 1 & 1 & 1 & 1 & 2
        \end{array}\right)=\aa_1+\aa_2\quad\text{with}\\
        \aa_1&:=&\left(
                                    \begin{array}{cccccc}
                                      1 & 1 & 1 & 0 & 0 & 0 \\
                                    \end{array}
                                  \right)\\
        \aa_2&:=&\left(
                                                          \begin{array}{cccccc}
                                                            0 & 0 & 0 & 1 & 1 & 2 \\
                                                          \end{array}
                                                        \right)
  \end{eqnarray*}
  By applying algorithm~\ref{algoritmoDnef}, one gets $\D_{\aa_k}=\conv(\L_{\aa_k})$, with $k=1,2$ and
  \begin{equation*}
    \L_{\aa_1}=\left( \begin {array}{cccccc} 2&-1&-1&-1&-1&-1\\
    -1&2&-1&-1&-1&-1\\
    -1&-1&2&-1&-1&-1\\
    0&0&0&3&0&0\\
    0&0&0&0&3&0\end {array} \right)
  \end{equation*}
  \begin{equation*}
    \L_{\aa_2}=\left( \begin {array}{cccccc} 4&0&0&0&0&0\\
    0&4&0&0&0&0\\
    0&0&4&0&0&0\\
    -1&-1&-1&3&-1&-1\\
    -1&-1&-1&-1&3&-1\end {array} \right)
  \end{equation*}
  Therefore $\cv{\D}_\aa=\conv(\D_{\aa_1},\D_{\aa_2})=\conv\left(\cv{\L}_\aa\right)$, with $\cv{\L}_\aa=\left(
                                                                                               \begin{array}{c}
                                                                                                 \L_{\aa_1} \,\vline\ \L_{\aa_2} \\
                                                                                               \end{array}
                                                                                             \right)
                                                                                             $.
  Step (D) in algorithm~\ref{algoritmoDnef} then gives
  \begin{eqnarray*}
    \cv{\bb}&:=&\left(\begin{array}{cccccccccccc}
          1 & 1 & 1 & 1 & 1 & 1 & 2 & 2 & 2 & 2 & 2 & 1\\
        \end{array}\right)=\bb_1+\bb_2\quad\text{with}\\
        \bb_1&:=&\left(
                                    \begin{array}{cccccccccccc}
                                      1 & 1 & 1 & 1 & 1 & 1 & 0 & 0 & 0 & 0 & 0 & 0\\
                                    \end{array}
                                  \right)\\
        \bb_2&:=&\left(
                                                          \begin{array}{cccccccccccc}
                                                            0 & 0 & 0 & 0 & 0 & 0 & 2 & 2 & 2 & 2 & 2 & 1 \\
                                                          \end{array}
                                                        \right)
  \end{eqnarray*}
  so that
  \begin{equation*}
    \cv{\D}_{\bb_1}=\conv\left(
                     \begin{array}{cccc}
                       \0 & \e_1 & \e_2 & \e_3 \\
                     \end{array}
                   \right)
  \end{equation*}
  \begin{equation*}
    \cv{\D}_{\bb_2}=\conv\left(
                     \begin{array}{cccccc}
                       -\1 & \0 & \e_4 & \e_5 & \ll' & \ll'' \\
                     \end{array}
                   \right)
  \end{equation*}
  with
  \begin{equation*}
  \ll'=\left(
                                                      \begin{array}{c}
                                                        -1/4 \\
                                                        -1/4 \\
                                                        -1/4 \\
                                                        -1/4 \\
                                                        5/4 \\
                                                      \end{array}
                                                    \right)\quad,\quad\ll''=\left(
                                                                      \begin{array}{c}
                                                                        -1/4 \\
                                                                        -1/4 \\
                                                                        -1/4 \\
                                                                        5/4 \\
                                                                        -1/4 \\
                                                                      \end{array}
                                                                    \right)
  \end{equation*}
 Then $\dim\cv{\D}_{\bb_1}=3=\dim\cv{\D}_{\bb_2}$ and $\dim\left(\cv{\D}_{\bb_1}+\cv{\D}_{\bb_2}\right)=5$ meaning that $\cv{\D}_{\bb_1},\cv{\D}_{\bb_2}$ are 3-independent and 4-dependent.\\
 Moreover, part (b) in Remark~\ref{rem:CIequazioni} explicitly gives the two polynomials
 \begin{eqnarray}\label{Y_3,4-dual}
   &f_1^\vee = a_1\, x_1x_2x_3x_4x_5x_6+ a_2\,{x_1}^{3}{x_{7}}^{4}+a_3\,{x_2}^{3}{x_{8}}^{4}+a_4\,{x_3}^{3}{x_{9}}^{4}&\\
   \nonumber
   &f_2^\vee = b_1\,x_{6}^{3}x_{12}^{3}+b_2\,x_{4}^{3}x_{7}x_{8}x_{9}x_{10}^{5}x_{11}+b_3\,x_{5}^{3}x_{7}x_{8}x_{9}
   x_{10}x_{11}^{5}+b_4\,x_{7}^2x_{8}^{2}x_{9}^{2}x_{10}^{2}x_{11}^{2}x_{12}&
 \end{eqnarray}
 defining the complete intersection $Y^\vee=Y^\vee_1\cap Y^\vee_2$ in $\cv{\XX}_\aa$\,. The latter is a complete toric variety of Picard number 7 and $Y_1^\vee,Y^\vee_2$ turn out to be divisors of degree $(3,12,12,4,12,4,0),(0,12,12,6,30,15,6)\in\Cl(\cv{\XX}_\aa)\cong\Z^7$, respectively.
 \halfline

 \subsection{Computing Hodge numbers} By Theorem~\ref{thm:h21Y}, for the generic
 $$Y=Y_{3,4}:=Y_1\cap Y_2\subset \P^5\quad\text{with}\quad Y_k\in|D_{\aa_k}|$$
 one has
 \begin{equation*}
    h^p(\cO_Y)=\left\{\begin{array}{cc}
                        1 & \text{for $p=0$} \\
                        l^*(\D_\aa)=6 & \text{for $p=3$} \\
                        0 & \text{otherwise}
                      \end{array}
    \right.
  \end{equation*}
  being $\D_\aa=\D_{\aa_1}+\D_{\aa_2}=\conv(\L_\aa)$ and
  \begin{equation*}
    \L_\aa= \left( \begin {array}{cccccc} 6&-1&-1&-1&-1&-1\\
    -1&6&-1&-1&-1&-1\\
    -1&-1&6&-1&-1&-1\\
    -1&-1&-1&6&-1&-1\\
    -1&-1&-1&-1&6&-1\end {array} \right)
  \end{equation*}
  so giving
  \begin{equation*}
    \Int\left(\D_\aa\right)\cap M=\left\{\0,\e_1,\e_2,\e_3,\e_4,\e_5\right\}
  \end{equation*}
  Moreover, $\D_i=\D_{D_i}=\conv(\L_i)$, with
  \begin{eqnarray*}
    \L_6&=&\left(
           \begin{array}{cccc}
             \0 & \e_1 & \cdots & \e_5 \\
           \end{array}
         \right)\\
    \forall\,i:1\le i\le 5\quad\L_i&=&\L_6 -\left(
                                              \begin{array}{c}
                                                \0_6 \\
                                                \vdots\\
                                                \0_6\\
                                                \hline
                                                -\1_6\\
                                                \hline
                                                \0_6 \\
                                                \vdots \\
                                                \0_6 \\
                                              \end{array}
                                            \right)\,\}\ \text{$i$-th row}
  \end{eqnarray*}
  Then, recalling the definition (\ref{Ka}) of $K_\aa$, one has
  \begin{eqnarray*}
    K_\aa&=&l^*\left(\D_{\aa_1}+\D_\aa\right)+l^*\left(\D_{\aa_2}+\D_\aa\right)-\sum_{i=1}^6 l^*\left(\D_i+\D_\aa\right)\\
    &-&2l^*\left(\D_\aa\right)-l^*\left(2\D_{\aa_1}\right)-l^*\left(2\D_{\aa_2}\right)+
    \sum_{i=1}^6\left(l^*\left(\D_i+\D_{\aa_1}\right)+l^*\left(\D_i+\D_{\aa_2}\right)\right)\\
    &+&l^*\left(\D_{\aa_1}\right)+l^*\left(\D_{\aa_2}\right)-\sum_{i=1}^6l^*\left(\D_i\right)
  \end{eqnarray*}
  But
  \begin{eqnarray*}
    \forall\,i,k&l^*\left(\D_{\aa_k}\right) &= l^*\left(\D_i+\D_{\aa_k}\right)=l^*\left(\D_i\right)=0 \\
    \forall\,i&l^*\left(\D_i+\D_{\aa}\right)&= l^*\left(2\D_{\aa_2}\right) = 21 \\
    & l^*\left(2\D_{\aa_1}\right) &=  1\\
    & l^*\left(\D_{\aa_1}+\D_\aa\right)&=126\\
    & l^*\left(\D_{\aa_2}+\D_\aa\right)&=252
  \end{eqnarray*}
  so that
  \begin{equation*}
    K_\aa=126+252 - 6\cdot21-12-1-21 =218
  \end{equation*}
  Therefore, Theorem~\ref{thm:h21Y}, and in particular relation (\ref{h^n-l-1}), give
  \begin{equation}\label{h^1Omega_Y}
    h^p(\Omega_Y)=\left\{\begin{array}{cc}
                           1 & \text{for $p=1$} \\
                           224 & \text{for $p=2$} \\
                           0 & \text{otherwise}
                         \end{array}
    \right.
  \end{equation}
  On the other hand, for the generic
  $$Y^\vee:=\bigcap_{k=1}^l Y^\vee_k\subset \cv{\XX}_\aa\quad\text{with}\quad Y_k\in|\cv{D}_{\bb_k}|$$
consider the birational transform $\cv{Y}^\vee:=\cv{\phi}^{-1}\left(Y^\vee\right)$ under the resolution $\cv{\phi}$ defined in (\ref{risoluzioni}). Then, Theorem~\ref{thm:h21Yd} gives
\begin{equation*}
    h^p(\cO_{\cv{Y}^\vee})=h^p(\cO_{Y^\vee})=\left\{\begin{array}{cc}
                        1 & \text{for $p=0,3$} \\
                        0 & \text{for $1\le p\le 2$}
                      \end{array}
    \right.
  \end{equation*}
  being $l^*\left(\D_{\cv{\bb}}\right)=1$. Moreover, by Lemma~\ref{lem:h11st}
  \begin{equation*}
    h^1\left(\widehat{\Omega}_{\cv{Y}^\vee}\right)=\rk\left(\Cl\left(\widehat{\XX}_{\aa_0}\right)\right)
    =h^{1,1}_{\text{st}}\left(\XX''_{\aa}\right)=h^{1,1}_{\text{st}}\left(\XX'_{\aa_0}\right)=
    \binom{|\aa|+5}{5}-\binom{|\aa|}{5}= 771
  \end{equation*}

  \subsubsection{Computing moduli numbers}\label{sssez:moduli} Complex moduli of the $f$-mirror family $\mathcal{Y}_{\cv{\bb}}$ of complete intersections $Y^\vee\subset\cv{\XX}_\aa$ can be computed by observing that a variables rescaling allows us to rewrite the defining polynomials $f^\vee_1,f^\vee_2\in\Cox(\cv{\XX}_\aa)$, explicitly written in (\ref{Y_3,4-dual}), as follows
  \begin{eqnarray}\label{eq.3,4}
    &f_1^\vee = \psi\, x_1x_2x_3x_4x_5x_6+ {x_1}^{3}{x_{7}}^{4}+{x_2}^{3}{x_{8}}^{4}+{x_3}^{3}{x_{9}}^{4}&\\
   \nonumber
   &f_2^\vee = x_{6}^{3}x_{12}^{3}+x_{4}^{3}x_{7}x_{8}x_{9}x_{10}^{5}x_{11}+x_{5}^{3}x_{7}x_{8}x_{9}
   x_{10}x_{11}^{5}+x_{7}^2x_{8}^{2}x_{9}^{2}x_{10}^{2}x_{11}^{2}x_{12}&
  \end{eqnarray}
  where the complex parameter $\psi$ cannot be normalized by a further automorphism. Then $\psi$ represents the unique complex modulus of $\mathcal{Y}_\bb$, so that $m_{Y^\vee}=1$. The latter, compared with $h^1(\Omega_Y)=1$, as computed in (\ref{h^1Omega_Y}), implies that $Y^\vee$ is $A$-side mirror partner of $Y$.

 On the other hand, complex moduli of the family $\mathcal{Y}_{\aa}$ of complete intersections $Y=Y_{3,4}\subset\P^5$ can be computed as follows.
  First of all, one has
  \begin{eqnarray*}
    \dim\left(\P H^0(\P^5,\cO_{\P^5}(3))\right)&=&{3+5\choose 3}-1=55\\
    \dim\left(\P H^0(\P^5,\cO_{\P^5}(4))\right)&=&{4+5\choose 4}-1=125
  \end{eqnarray*}
  Then notice that, given a cubic polynomial  $f_3\in H^0(\cO_{\P^5}(3))$, two quartic polynomials $f_4,f'_4\in H^0(\cO_{\P^5}(4))$ cut a out on $\{f_3=0\}$ the same complete intersection if and only if $f_3$ divides both $f_4$ and $f'_4$. The dimension of the subspace of $H^0(\cO_{\P^5}(4))$ generated by quartic polynomials of this kind is
  \begin{equation*}
    \dim\left(H^0(\P^5,\cO_{\P^5}(1))\right)={1+5\choose 1}=6
  \end{equation*}
  Finally we have to quotient by the action of $\Aut(\P^5)=\P\GL(5)$, so giving
  \begin{equation*}
    m_Y=55+125 - 6 - 35= 139
  \end{equation*}
For the $B$-side topological test, one has to perform the following comparison
\begin{equation*}
  h^1\left(\widehat{\Omega}_{\cv{Y}^\vee}\right)=771>139=m_Y
\end{equation*}
Then $Y^\vee$ is a $B$-side mirror partner of $Y_{3,4}$ if there exists a partial resolution $\vf:\XX\longrightarrow\cv{\XX}_\aa$ such that
\begin{itemize}
  \item[(i)] $\rk(\Cl(\XX))=139$,
  \item[(ii)] $\XX$ is $\Q$-factorial and satisfies conditions in Remark~\ref{rem:>>0},
  \item[(iii)] $\vf^{-1}(Y^\vee)$ is quasi-smooth,
\end{itemize}
so that $h^1\left(\widehat{\Omega}_{\vf^{-1}(Y^\vee)}\right)=139$ by Theorem~\ref{thm:h21Yd}. \\
Conditions (i) and (ii) can be obtained by choosing a suitable factorization of
the resolution $\cv{\phi}:\widehat{\XX}_{\aa_0}\longrightarrow\cv{\XX}_{\aa}$ as follows
\begin{equation*}
  \xymatrix{\widehat{\XX}_{\aa_0}\ar[rr]^-{\cv{\phi}}\ar[rd]&&\cv{\XX}_{\aa}\\
            &\XX'\ar[ru]^-{\vf'}&}
\end{equation*}
and then a small birational contraction $\vf'':\XX\longrightarrow\XX'$, determined by the choice of a suitable chamber of the secondary fan of $\XX'$, such that $\XX$ is $\Q$-factorial.\\
Unfortunately, condition (iii) seems to be quite difficult to understand, at least as far as the author's knowledge allows.

\subsection{LG mirror models}\label{ssez:Y34} Starting from equations (\ref{eq.3,4}) of the $f$-mirror model $Y^\vee$ and applying considerations given in \S~\ref{ssez:LG/CI} and Proposition~\ref{prop:LGmirror}, one can construct a generalized LG mirror model and a Givental type LG mirror model, respectively.

\noindent The former is given by $(\cv{\T}_\aa, \ff_\bb^\vee)$, where
  \begin{eqnarray*}
    \ff^\vee_\bb&=&\left({f^\vee_1\over\x^{\bb_1}},{f^\vee_2\over\x^{\bb_2}}\right)
    =\left(\psi+
    {{x_1}^{3}{x_{7}}^{4}+{x_2}^{3}{x_{8}}^{4}+{x_3}^{3}{x_{9}}^{4}\over x_1x_2x_3x_4x_5x_6},\right.\\
    &&\left. 1+{x_{4}^{3}x_{7}x_{8}x_{9}x_{10}^{5}x_{11}+x_{5}^{3}x_{7}x_{8}x_{9}
   x_{10}x_{11}^{5}+x_{6}^{3}x_{12}^{3}\over x_{7}^2x_{8}^{2}x_{9}^{2}x_{10}^{2}x_{11}^{2}x_{12}}\right)
  \end{eqnarray*}
  Therefore, Proposition~\ref{prop:LGmirror} gives the Givental type LG mirror model $(\pi^{-1}(\T_2),F)$ with
  \begin{eqnarray*}
    F &=& {f^\vee_1\over\x^{\bb_1}}+{f^\vee_2\over\x^{\bb_2}}= u_{11}+\cdots+u_{41}+u_{12}+\cdots+u_{42}:\C^8\longrightarrow\C\\
    \pi(\uu) &=&(q_1,q_2):(\C^*)^8\longrightarrow(\C^*)^2\\
    u_{11}&=&\psi\\
    u_{21}&=&{x_1^2x_7^4\over x_2x_3x_4x_5x_6}\\
    u_{31}&=&{x_2^2x_8^4\over x_1x_3x_4x_5x_6}\\
    u_{41}&=&{x_3^2x_9^4\over x_1x_2x_4x_5x_6}\\
    u_{12}&=& 1\\
    u_{22}&=& {x_{4}^{3}x_{10}^3\over x_{7}x_{8}x_{9}x_{11}x_{12}}\\
    u_{32}&=& {x_{5}^{3}x_{11}^{3}\over x_{7}x_{8}x_{9}x_{10}x_{12}}\\
    u_{42}&=&{x_6^3x_{12}^2\over x_{7}^2x_{8}^{2}x_{9}^{2}x_{10}^{2}x_{11}^{2}}\\
    q_1&=&u_{11}\cdots u_{41}=\psi{x_7^4x_8^4x_9^4\over x_4^3x_5^3x_6^3}\\
    q_2&=&u_{12}\cdots u_{42}={x_4^3x_5^3x_6^3\over x_7^4x_8^4x_9^4}
  \end{eqnarray*}
  As for examples~\ref{ex:Y23} and \ref{ex:Y13}, again it turns out that $q_1q_2=\psi$, so that $\pi$ can be thought of as a fibration over the complex parameter of the mirror family, up to the composition with $(q_1,q_2)\in\T_2\mapsto q_1q_2\in\C^*$.

\bibliographystyle{acm}

\end{document}